\documentclass[11pt]{amsart}

\pdfoutput=1

\usepackage{amssymb, amsfonts, latexsym, amsthm, amsmath, verbatim, enumerate}
\pdfoutput=1
\usepackage{soul} 
\usepackage{paralist}
\usepackage[all]{xy}
\usepackage{amscd}
\usepackage{mathrsfs} 
\usepackage[english]{babel}
\usepackage{tikz}
\usepackage{graphicx}
\usepackage{epstopdf}
\definecolor{darkblue}{rgb}{0,0,0.4} 
\usepackage[colorlinks=true, citecolor=darkblue, filecolor=darkblue, linkcolor=darkblue,urlcolor=darkblue]{hyperref}
\usetikzlibrary{arrows}
\usetikzlibrary{calc}
\usetikzlibrary{decorations.pathmorphing}
\usepackage{amscd}
\usepackage{enumitem}
\usepackage[color=blue!20!white]{todonotes}
\makeatletter
\providecommand\@dotsep{5}
\def\listtodoname{List of Todos}
\def\listoftodos{\@starttoc{tdo}\listtodoname}
\makeatother

\tikzstyle{crossing}=[circle,fill=white,minimum height=6pt,inner sep=0pt, outer sep=0pt, style={transform shape=false}]

\numberwithin{equation}{section}

\theoremstyle{plain}
\newtheorem{thm}[equation]{Theorem}

\newtheorem*{thm*}{Theorem}
\newtheorem{prop}[equation]{Proposition}
\newtheorem{cor}[equation]{Corollary}
\newtheorem{lem}[equation]{Lemma}

\theoremstyle{definition}
\newtheorem{rmk}[equation]{Remark}

\newtheorem{defn}[equation]{Definition}

\newcommand{\f}{\mathbb{F}}
\newcommand{\forgot}{\mathcal{F}}
\newcommand{\subinclude}{\mathcal{I}}
\newcommand{\quotient}{\mathcal{Q}}
\newcommand{\hoco}{\mathrm{hocolim}}\newcommand{\hocolim}{\hoco} 

\newcommand{\tH}{{\widetilde{H}}}
\newcommand{\X}{\mathcal{X}}

\newcommand{\Id}{\mathrm{Id}}
\newcommand{\topp}{\mathrm{Top}_*}
\newcommand{\Hom}{\mathrm{Hom}} 
\newcommand{\from}{\colon}
\newcommand{\into}{\hookrightarrow}

\newcommand{\Kh}{\mathit{Kh}} 
\newcommand{\kho}{\Kh_o} 
\newcommand{\kor}{\widetilde{\Kh}_o} 
\newcommand{\unih}{\Kh_u} 

\newcommand{\KhGen}{\mathit{Kg}}

\newcommand{\KhCx}{\mathit{Kc}}
\newcommand{\oddKhCx}{\KhCx_o}
\newcommand{\unic}{\KhCx_u}\newcommand{\uniKhCx}{\unic} 

\newcommand{\khoh}{\X_e}\newcommand{\Khspace}{\khoh} 
\newcommand{\khoo}{\X_o}\newcommand{\oddKhspace}{\khoo} 
\newcommand{\khor}{\widetilde{\X}_o} 
\newcommand{\unis}{\X_u} 
\newcommand{\reduniKhspace}{\widetilde{\X}_u} 

\newcommand{\burn}{\mathscr{B}} 
\newcommand{\oddb}{\burn_{\sigma}} 
\newcommand{\eqb}{\burn_\xi} 
\newcommand{\oddeq}{\mathcal{D}} 
\newcommand{\two}{\underline{2}} 
\newcommand{\Tot}{\mathrm{Tot}} 

\newcommand{\asig}{\mathcal{A}^\sigma} 

\newcommand{\refl}{\mathfrak{r}}
\newcommand{\dig}{\mathfrak{d}} 

\newcommand{\ZZ}{\mathbb{Z}}
\renewcommand{\th}{^{\text{th}}}
\newcommand{\Map}{\mathrm{Map}}
\newcommand{\sym}{\mathrm{sym}}
\newcommand{\dsym}{{2\mathrm{sym}}} 
\newcommand{\Cat}{\mathscr{C}}
\newcommand{\cellC}{C_{\mathrm{cell}}}
\newcommand{\redcellC}{\widetilde{C}_{\mathrm{cell}}}

\newcommand{\AbFunc}{\mathfrak{F}}
\newcommand{\Abelianize}{\mathcal{A}}
\newcommand{\diagram}{L} 
\newcommand{\RR}{\mathbb{R}}
\newcommand{\SteenrodAlg}{\mathcal{A}_2}
\newcommand{\Filt}{\mathcal{F}}
\newcommand{\cell}{\mathcal{C}} 

\newcommand{\CRealize}[1]{\|#1\|}
\newcommand{\Realize}[1]{|#1|}

\newcommand{\op}{\mathrm{op}}
\newcommand{\Mod}{\text{-Mod}}
\newcommand{\bdy}{\partial}
\renewcommand{\emptyset}{\varnothing}
\newcommand{\basis}[1]{\langle #1\rangle}

\title{An odd Khovanov homotopy type}

\author{Sucharit Sarkar}
\email{sucharit@math.ucla.edu}
\address{Department of Mathematics, University of California, Los Angeles, CA 90095}
\thanks{SS was supported by NSF Grant DMS-1643401}

\author{Christopher Scaduto}
\email{cscaduto@scgp.stonybrook.edu}
\address{Simons Center for Geometry and Physics, Stony Brook, NY 11794}
\thanks{CS was supported by NSF Grant DMS-1503100}

\author{Matthew Stoffregen}
\email{mstoff@mit.edu}
\address {Department of Mathematics, Massachusetts Institute of Technology, Cambridge, MA 02142}
\thanks{MS was supported by NSF Grant DMS-1702532.}
\begin{document}

\begin{abstract}
  For each link $L\subset S^3$ and every quantum grading $j$, we
  construct a stable homotopy type $\khoo^j(L)$ whose cohomology
  recovers Ozsv\'ath-Rasmussen-Szab\'o's odd Khovanov homology,
  $\tH^i(\khoo^j(L))=\kho^{i,j}(L)$, following a construction of
  Lawson-Lipshitz-Sarkar of the even Khovanov stable homotopy
  type. Furthermore, the odd Khovanov homotopy type carries a $\ZZ/2$
  action whose fixed point set is a desuspension of the even Khovanov
  homotopy type. We also construct a $\ZZ/2$ action on an even
  Khovanov homotopy type, with fixed point set a desuspension of
  $\khoo^j(L)$.
\end{abstract} 
\maketitle

\tableofcontents

\section{Introduction} \label{sec:intro}
\subsection{Khovanov homologies}
In \cite{kho1} Khovanov categorified the Jones polynomial: to a link
diagram $L$, he associated a bigraded chain complex, whose graded
Euler characteristic is (a certain normalization of) the Jones
polynomial of $L$, and whose (graded) chain homotopy type is an
invariant of the underlying link.  Several generalizations were soon
constructed, such as invariants for tangles
\cite{khovtanglefunctor,natantangle}, various perturbations
\cite{lee,natantangle}, versions for other polynomials \cite{kr1,kr2},
and many others. The categorified invariant carried structure that was
not visible at the decategorified level. To wit, to a link cobordism
in $\RR^3\times[0,1]$, there is an associated map of Khovanov chain
complexes \cite{jacobsson,khinvartangle,natantangle,cmw-disoriented};
and this map, along with Lee's perturbation, was used by Rasmussen in
\cite{rasmus-s} to define a numerical concordance invariant $s$ and to
give a combinatorial proof of a theorem due to Kronheimer and Mrowka
\cite{km1} on the four-ball genus of torus knots (popularly known as
the Milnor conjecture). Khovanov homology itself turns out to be a
more powerful invariant than the Jones polynomial. Indeed, Khovanov
homology is known to detect the unknot \cite{km2}, while the
corresponding question for the Jones polynomial remains wide open.

In \cite{os-double}, Ozsv\'ath and Szab\'o constructed the first
relation between Khovanov homology and Floer-theoretic
invariants---Heegaard Floer homology \cite{os-main} to be
specific---in the form of a spectral sequence from reduced Khovanov
homology of a link to the Heegaard Floer homology of its branched
double cover. The spectral sequence was originally constructed over
$\ZZ/2$, but it was soon realized that its integral lift does not
start from the usual Khovanov homology, but rather from a different
homology theory which has the same $\ZZ/2$ reduction; this new chain
complex was constructed by Ozsv\'ath-Rasmussen-Szabo~\cite{ors} and is
usually called the odd Khovanov complex, and this is the version that
seems closely related to Heegaard Floer type invariants. (The even
theory was later discovered to be related to certain other Floer
theories, such as instanton Floer homology \cite{km2} and symplectic
Khovanov homology \cite{seidel-smith}.)  The two versions were
combined by a pullback into a single unified theory by Putyra
\cite{putyra}, cf.~\cite{putyrashumakovitch}: the unified Khovanov
complex is a chain complex over $\ZZ[\xi]/(1-\xi^2)$ which recovers
the even (respectively, odd) Khovanov chain complex upon setting
$\xi=1$ (respectively, $\xi=-1$).

In this paper, we will typically decorate the objects from the even
theory by the subscript $e$, the ones from the odd theory by the
subscript $o$, and the ones from the unified theory by the subscript
$u$. In particular, we will denote the even, odd, and the unified
Khovanov complexes as $\KhCx_e(L)$, $\KhCx_o(L)$, and $\KhCx_u(L)$,
respectively.

\subsection{Khovanov homotopy types} 

In \cite{lshomotopytype}, Lipshitz and Sarkar associated to a link
diagram $L$ a finite CW spectrum $\khoh(L)$, whose reduced cellular
cochain complex is isomorphic to the Khovanov complex $\KhCx_e(L)$,
taking the (non-basepoint) cells of $\khoh(L)$ to the standard
generators of $\KhCx_e(L)$. The (stable) homotopy type of $\khoh(L)$
is an invariant of the underlying link; specifically, Reidemeister
moves from the diagram $L$ to a diagram $L'$ induce stable homotopy
equivalences $\khoh(L) \to \khoh(L')$. A different construction of an
even Khovanov homotopy type was constructed independently by
Hu-Kriz-Kriz~\cite{hkk}, and the two versions were later shown to be
equivalent~\cite{lls1}. A stable homotopy refinement of Khovanov
homology endowed it with extra structure, such as an action by the
Steenrod algebra~\cite{lssteenrod}, which was then used to construct a
family of additional $s$-type concordance
invariants~\cite{lsrasmussen}, as well as to show that Khovanov
homotopy type is a strictly stronger invariant than Khovanov
homology~\cite{Seed-Kh-square}.

One could ask for a spectrum invariant $\khoo(L)$ satisfying analogous
properties, but with Khovanov homology replaced with odd Khovanov
homology.  The original Lipshitz-Sarkar contruction using the
Cohen-Jones-Segal framed flow categories machine from \cite{cjs} does
not seem to admit an easy generalization: on account of the signs that
appear in the definition of odd Khovanov homology, there is no framed
flow category for the odd theory covering the framed cube flow
category.  However, in \cite{lls1}, Lawson-Lipshitz-Sarkar provided
several more abstract constructions of $\khoh(L)$---similar to the one
from~\cite{hkk}---in order to understand the behavior of the Khovanov
spectrum under disjoint union and connected sum. In this paper, we
will give a slight generalization of their machinery to construct a
finite $\ZZ_2$-equivariant CW spectrum $\khoo(L)=\bigvee_j\khoo^j(L)$
for each oriented link diagram $L$ (Definition~\ref{def:oddkhmain}).

\begin{thm}\label{thm:oddkhmain}
  The (stable) homotopy type of the odd Khovanov spectrum
  $\khoo(L)=\bigvee_j\khoo^j(L)$ from
  Definition~\ref{def:oddkhmain} is independent of the choices in
  its construction and is an invariant of the isotopy class of the
  link corresponding to $L$. Its reduced cellular cochain complex
  agrees with the odd Khovanov complex $\oddKhCx(L)$,
  \[
  \redcellC^i(\khoo^j(L))=\oddKhCx^{i,j}(L),
  \]
  with the cells mapping to the distinguished generators of $\oddKhCx(L)$.  
\end{thm}

We also construct a reduced theory: a finite $\ZZ_2$-equivariant CW
spectrum $\khor(L,p)=\bigvee_j\khor^j(L,p)$ for each oriented link
diagram $L$ with basepoint $p$ (Definition~\ref{def:redkh}).
\begin{thm}\label{thm:redkh}
  The (stable) homotopy type of the reduced odd Khovanov spectrum
  $\khor(L,p)=\bigvee_j \khor^j(L,p)$ from
  Definition~\ref{def:redkh} is independent of the choices in its
  construction and is an invariant of the isotopy class of the pointed
  link corresponding to $(L,p)$.  Its reduced cellular cochain complex
  agrees with the reduced odd Khovanov complex
  $\widetilde{\KhCx}_o(L)$,
  \[
  \redcellC^i(\khor^j(L,p))= \widetilde{\KhCx}_o^{i,j}(L),
  \]
  with the cells mapping to the distinguished generators of
  $\widetilde{\KhCx}_o(L)$.  There is a cofibration sequence 
  \[\khor^{j-1}(L,p)
  \to \khoo^j(L)\to\khor^{j+1}(L,p).\]
\end{thm}

We introduce concordance invariants built from this construction, in
analogy with \cite{lsrasmussen}.  To do so, we show that associated to
a cobordism of links, there exists a map of odd Khovanov spectra (we
do not attempt to show that the map is well-defined); the map induces
a map on the odd Khovanov chain complex, and reduces mod-2 to the
usual cobordism map on $\KhCx(L;\ZZ_2)$. Therefore:
\begin{thm}\label{thm:cobordism-maps}
  The Khovanov cobordism map $\Kh(L;\ZZ_2)\to\Kh(L';\ZZ_2)$ associated
  to a link cobordism $L\to L'$ from
  \cite{jacobsson,khinvartangle,natantangle} is a map of
  $\asig$-modules, where $\asig$ is the free product of two copies of
  the mod-2 Steenrod algebra, and the first (respectively, second)
  copy acts on the mod-2 Khovanov homology by viewing it as the mod-2
  cohomology of the even (respectively, odd) Khovanov homotopy type.
\end{thm}

Moreover, in Definition~\ref{def:evenaction}, we construct an even
stable homotopy type $\khoh'(L)$ (that is, a finite CW spectrum whose
cellular chain complex is the even Khovanov chain complex), equipped
with a $\ZZ_2$-action. This $\ZZ_2$-action is not visible from the
Burnside functor constructed in \cite{lls1}, so in some sense this
$\ZZ_2$-action arises from the odd theory.  We conjecture that the
even space constructed here is stable homotopy equivalent to the
construction of \cite{lshomotopytype}.
\begin{thm}\label{thm:evenintro}
  The (stable) homotopy type of the even Khovanov spectrum $\khoh'(L)$
  from Definition~\ref{def:evenaction} is independent of the choices
  in its construction and is an invariant of the isotopy class of $L$.
  Its reduced cellular cochain complex agrees with the odd Khovanov
  complex $\KhCx_e(L)$,
  \[
  \redcellC^i(\khoh^{\prime\,j}(L))=\KhCx_e^{i,j}(L),
  \]
  with the cells mapping to the distinguished generators of $\oddKhCx(L)$.
\end{thm}

Finally, similar to unified Khovanov homology, we combine
$\khoh(L)$ and $\khoo(L)$ into a single finite $\ZZ_2 \times
\ZZ_2$-equivariant CW spectrum $\unis(L)=\bigvee_j \unis^j(L)$, which
we think of as a `unified Khovanov spectrum'
(Definition~\ref{def:unifiedintro}):
\begin{thm}\label{thm:unifiedintro}
  The (stable) homotopy type of the unified Khovanov spectrum
  $\unis(L)$ from Definition~\ref{def:unifiedintro} is independent
  of the choices in its construction and is an invariant of the
  isotopy class of $L$.  Its reduced cellular cochain complex agrees
  with the unified Khovanov complex $\unic(L)$,
  \[
  \redcellC^i(\unis^j(L))= \unic^{i,j}(L),
  \]
  with the cells mapping to the distinguished generators of
  $\unic(L)$, and the two $\ZZ_2$ actions correspond to multiplication by
  $\xi$ and $-\xi$, respectively.
\end{thm}   

There is also a reduced unified spectrum $\reduniKhspace(L)$ for which the analogue of Theorem \ref{thm:redkh} (Proposition \ref{prop:reduced-unified}) holds.

The different spectra and the different actions admit the following
relationship:
\begin{thm}\label{thm:equivariance} Let $L$ be a link diagram.
\begin{enumerate}[leftmargin=*]
\item\label{itm:thm-equivariance-1} The action of the two
  $\ZZ_2$-factors is free away from the basepoint on $\unis(L)$: $\khoh(L)$ is the geometric
  quotient under the action of the first factor (sometimes called
  $\ZZ_2^+$) and $\khoo(L)$ is the geometric quotient under the second
  factor (sometimes called $\ZZ_2^-$); moreover, the $\ZZ_2$-action on
  $\khoo(L)$ is the quotient of the $\ZZ_2^+$-action on $\unis(L)$.
\item\label{itm:thm-equivariance-2} The geometric fixed-point set of $\khoo^j(L)$ under $\ZZ_2$ is
  precisely $\Sigma^{-1}\khoh^j(L)$, and quotienting by the fixed
  point set produces $\unis(L)$. The induced $\ZZ_2$-action on
  $\unis(L)$ agrees with the $\ZZ_2^+$-action. This produces a
  cofibration sequence
  \[
  \Sigma^{-1}\khoh(L) \to \khoo(L) \to \unis(L),
  \]
  and the induced long exact sequence on cohomology agrees with
  the one constructed in \cite{putyrashumakovitch}.
\item\label{itm:thm-equivariance-3} The Puppe map $\unis(L)\to\khoh(L)$ from the previous
  cofibration sequence is homotopic to the quotient map
  $\unis(L)\to\unis(L)/\ZZ_2^+$.
\item\label{itm:thm-equivariance-4} The geometric fixed-point set of
  $\khoh^{\prime\,j}(L)$ under $\ZZ_2$ is precisely
  $\Sigma^{-1}\khoo^j(L)$, and quotienting by the fixed point set
  produces $\unis(L)$. The induced $\ZZ_2$-action on $\unis(L)$ agrees
  with the $\ZZ_2^-$-action. This produces a cofibration sequence
  \[
  \Sigma^{-1}\khoo(L) \to \khoh^\prime(L) \to \unis(L),
  \]
  and the induced long exact sequence of cohomology agrees with
  the one constructed in \cite{putyrashumakovitch}.
\item\label{itm:thm-equivariance-5} The Puppe map $\unis(L)\to\khoo(L)$ from the previous
  cofibration sequence is homotopic to the quotient map
  $\unis(L)\to\unis(L)/\ZZ_2^-$.
\end{enumerate}
\end{thm}

\subsection{Burnside categories and functors}\label{subsec:tech}
This paper uses the machinery of Burnside functors from
\cite{hkk,lls1}.  There, the dual of the Khovanov chain complex of a
link diagram with $n$ (ordered) crossings is viewed as a diagram of
abelian groups:
\[
\AbFunc_e\from (\two^n)^\op \to \ZZ\Mod,
\]  
where $\two^n$ is the category with objects elements of $\{0,1\}^n$
and a unique arrow $a \to b$ if $a \geq b$.

In order to construct a stable homotopy type, one considers a
certain $2$-category $\burn$, \emph{the Burnside category}, whose
objects are finite sets and whose $1$-morphisms are finite
correspondences.  The $2$-category $\burn$ naturally comes with a
forgetful functor to abelian groups $\burn \to \ZZ\Mod$ by sending a
set $S$ to the free abelian group $\ZZ\basis{S}$ generated by $S$.
The Khovanov stable homotopy type arises from a lift:
\[
\begin{tikzpicture}[baseline={([yshift=-.8ex]current  bounding  box.center)},xscale=2.5,yscale=1.5]
\node (a0) at (0,0) {$\two^n$};
\node (a1) at (1,0) {$\ZZ\Mod$};
\node (b1) at (1,1) {$\burn$};

\draw[->] (a0) -- (a1) node[pos=0.5,anchor=north] {\scriptsize
  $\AbFunc^\op_e$}; \draw[->] (b1) -- (a1) node[pos=0.2,anchor=east] {};
\draw[->,dashed] (a0) -- (b1) node[pos=0.5,anchor=south east]{\scriptsize $F_e$};

\end{tikzpicture}
\]
The \emph{realization} of any functor $F_e\from \two^n \to \burn$ is
then defined as a finite CW spectrum $\Realize{F_e}$ associated to
$F_e$. Indeed, as shown in \cite{lls2}, the stable equivalence class
of the functor $F_e$, modulo shifting by the number of negative
crossings $n_-$, is itself an invariant (which recovers the even
Khovanov homotopy type $\khoh$---the appropriately shifted stable
homotopy type of $\Realize{F_e}$).

In this paper, we first review, in \S\ref{sec:khovanovhomology}, the odd
Khovanov chain complex, viewing it as a diagram:
\[
  \AbFunc_o\from (\two^n)^\op\to \ZZ\Mod.
\] 
Indeed, $\AbFunc_o$ and $\AbFunc_e$ can be combined by a pullback into
a unified functor
\[
  \AbFunc_u\from (\two^n)^\op\to \ZZ_u\Mod.
\] 
where $\ZZ_u=\ZZ[\xi]/(\xi^2-1)$, and $\AbFunc_e$ (respectively,
$\AbFunc_o$) is obtained by setting $\xi=+1$ (respectively, $\xi=-1$).

Then in \S\ref{sec:burn}, we move to some slight generalizations of
the Burnside $2$-category, $\oddb$, the \emph{signed Burnside
  category} in order to take account of the signs appearing in odd
Khovanov homology, and $\eqb$, the \emph{free $\ZZ_2$-equivariant
  Burnside category} in order to take account of the $\xi$-action. 

In \S\ref{sec:box} we show how the realization construction of
\cite{lls1} generalizes to $\oddb$ and $\eqb$.  Roughly, the
realization process of a functor to $\oddb$ is comparable to the
realization process of a functor to $\burn$, except that where a sign
appears, the corresponding cell is glued in by a fixed
orientation-reversing homeomorphism.

And then in \S\ref{sec:oddkh} we construct lifts
\[
\begin{tikzpicture}[ cross line/.style={preaction={draw=white, -, line width=4pt}},
baseline={([yshift=-.8ex]current  bounding  box.center)},xscale=3,yscale=1.5]

\node (ze) at (0,0.5) {$\ZZ\Mod$};
\node (zo) at (1.4,1) {$\ZZ\Mod$};
\node (zu) at (2,0) {$\ZZ_u\Mod$};

\node (two) at ($(ze)!0.5!(zo)$) {$\two^n$};

\node (be) at ($(ze)+(0,2)$) {$\burn$};
\node (bo) at ($(zo)+(0,2)$) {$\oddb$};
\node (bu) at ($(zu)+(0,2)$) {$\eqb$};

\foreach \t/\d/\a/\b in {e//south east/east,o/dashed/north west/east,u/dashed/south west/south}{
\draw[->] (two) -- (z\t) node[pos=0.5,inner sep=0,outer sep=1pt,anchor=\a] {\scriptsize $\AbFunc^\op_\t$};
\draw[cross line, ->,\d] (two) -- (b\t) node[pos=0.4,anchor=\b] {\scriptsize $F_\t$} ;
\draw[->] (b\t) -- (z\t);
}

\draw[->] (zu) edge node[pos=0.8,anchor=west,align=left] {\scriptsize $\xi=-1$} (zo) edge node[pos=0.5,anchor=north] {\scriptsize $\xi=+1$} (ze); 

\draw[cross line, ->] (bu) -- (be) node[pos=0.6,anchor=south] {\scriptsize $\quotient$};
\draw[->] (bo) edge node[pos=0.5,anchor=south east] {\scriptsize $\forgot$} (be) edge
node[pos=0.5,anchor=south west] {\scriptsize $\oddeq$} (bu);

\end{tikzpicture}
\]
where the arrows among the various versions of Burnside categories are
those from Figure~\ref{fig:burnsidecategories}. Note that the lift
\[
F_o\from\two^n\to\oddb
\]
recovers all the other lifts; it decomposes along quantum gradings
$F_o=\amalg_jF_o^j$, and its equivariant equivalence class, after
shifting by $n_-$, recovers all the (correctly shifted) stable
homotopy types obtained by the above-mentioned realization procedure,
and is itself an invariant:
\begin{thm}\label{thm:odd-functor-invariant}
  The equivariant equivalence class of the
  shifted functor $\Sigma^{-n_-}F^j_o$ from
  Definition~\ref{def:kh-odd-burnside-functor} is independent of all the
  choices in its construction and is a link invariant.
\end{thm}

\subsection*{Acknowledgement} We are grateful to Anna Beliakova, Mike Hill, Tyler
Lawson, Francesco Lin, Robert Lipshitz, and Ciprian Manolescu for many helpful
conversations.

\section{Khovanov homologies}\label{sec:khovanovhomology}

In this section we review the definitions and basic properties of
three versions of Khovanov homology for an oriented link $L$: ordinary
or {\emph{even}} Khovanov homology $\Kh(L)=\Kh_e(L)$, defined by
Khovanov \cite{kho1}; {\emph{odd}} Khovanov homology $\kho(L)$ defined
by Ozsv\'{a}th, Rasmussen and Szab\'{o} \cite{ors}; and $\unih(L)$,
the {\emph{unified}} theory of Putyra and Putyra-Shumakovitch
\cite{putyra,putyrashumakovitch}, which generalizes the previous two
theories. These three homological invariants will be upgraded to
Burnside functors in \S\ref{sec:oddkh}.

\subsection{The cube category}\label{subsec:prelim}
We first recall the cube category.  Call $\two=\{0,1\}$ the
one-dimensional cube, viewed as a partially ordered set by setting
$1>0$, or as a category with a single non-identity morphism from $1$
to $0$.

Call $\two^n=\{0,1\}^n$ the $n$-dimensional cube, with the partial
order given by
\[ u=(u_1,\dots,u_n) \geq v=(v_1,\dots,v_n) \text{ if and only if }
\forall \; i \; (u_i \geq v_i).
\] 
It has the categorical structure induced by the partial order, where
$\Hom_{\two^n}(u,v)$ has a single element if $u \geq v$ and is empty
otherwise.  Write $\phi_{u,v}$ for the unique morphism $u \to v$ if it
exists.  The cube carries a grading given by $|v|=\sum_i v_i$.  Write
$u\geqslant_k v$ if $u\geq v$ and $|u|-|v|=k$. When $u\geqslant_1 v$,
call the corresponding morphism $\phi_{u,v}$ an {\emph{edge}}.

\begin{defn}\label{def:signassign} The {\emph{standard sign
      assigment}} $s$ is the following function from edges of $\two^n$
  to $\ZZ_2$. For $u\geqslant_1 v$, let $k$ be the unique element in
  $\{1,\dots,n\}$ with $u_k > v_k$. Then
\[ 
	s_{u,v} \; := \; \sum^{k-1}_{i=1} u_i \bmod{2}.
\]
\end{defn}
Note that $s$ may be viewed as a 1-cochain in
$\cellC^*([0,1]^n;\ZZ_2)$. In general, $s+c$ is called a
\emph{sign assignment} for any $1$-cocycle $c$ in
$\cellC^*([0,1]^n;\ZZ_2)$.

\subsection{Some rings and modules} We will often write $\ZZ_2$
multiplicatively as $\{1,\xi\}$. The integral group ring of
$\ZZ_2$ then has the presentation $\ZZ[\xi]/(\xi^2-1)$,
which we abbreviate to $\ZZ_u$.  There are two basic
$\ZZ_u$-modules $\ZZ_e$ and $\ZZ_o$ obtained from
$\ZZ_u$ by setting $\xi=+1$ and $\xi=-1$, which fit into the
following diagram:
\begin{equation}
\begin{tikzpicture}[baseline={([yshift=-.8ex]current  bounding  box.center)},xscale=2.5,yscale=1.5]
\node (a0) at (0,0) {$\ZZ_u$};
\node (b1) at (-0.5,-0.75) {$\ZZ_e$};
\node (b2) at (0.5,-0.75) {$\ZZ_o$};
\node (c0) at (0,-1.5) {$\ZZ_2$};

\draw[->] (a0) -- (b1) node[pos=0.2,anchor=east] {\tiny$\xi=+1\;\;$};
\draw[->] (a0) -- (b2) node[pos=0.2,anchor=west] {\tiny$\;\;\xi=-1$};
\draw[->] (b1) -- (c0) node[pos=0.5,anchor=south] {};
\draw[->] (b2) -- (c0) node[pos=0.5,anchor=east] {};

\end{tikzpicture}\label{eq:pullback}
\end{equation}
The modules $\ZZ_e$ and $\ZZ_o$ are both infinite cyclic groups, for
which $\xi\in \ZZ_u$ acts as $+1$ on $\ZZ_e$, and by $-1$ on
$\ZZ_o$. All maps in the above diagram are surjections, and in fact
$\ZZ_u$ is a pull-back for this diagram in the category of
rings. Equivalently, $\ZZ_u$ is isomorphic to the subring of
$\ZZ\oplus \ZZ$ consisting of pairs $(a,b)$ with $a\equiv b
\bmod{2}$. Note that the kernel of the map $\xi=+1$ (resp. $\xi=-1$)
in the diagram is isomorphic to $\ZZ_o$ (resp. $\ZZ_e$). In
particular, we have a short exact sequence
$0\to \ZZ_e \to \ZZ_u\to \ZZ_o\to 0$, and an analogous exact sequence
with $e$ and $o$ swapped.

Now let $S$ be any finite set, and let $T(S)$ be the tensor algebra
generated by $S$ over $\ZZ_u$. Let $I$ be the two-sided ideal
of $T(S)$ generated by elements $x\otimes x$ and $x\otimes y - \xi
y\otimes x$ where $x,y\in S$.
\begin{defn}\label{def:uext} 
  Given a finite set $S$, we define the $\ZZ_u$-module
  $\Lambda_u(S):= T(S)/I$.
\end{defn}
We will abuse notation and write $x_1\otimes \cdots \otimes x_n\in
\Lambda_u(S)$ for the equivalence class of the element $x_1\otimes
\cdots \otimes x_n\in T(S)$, whenever each $x_i\in S$. We have the
fundamental relation
\[
x_1\otimes \cdots \otimes x_n \; = \;
\xi^{\text{sign}(\sigma)}x_{\sigma(1)}\otimes \cdots \otimes
x_{\sigma(n)}
\]
for any permutation $\sigma$ of length $n$. Upon setting $\xi=+1$, we
recover $\Lambda_e(S)$, the symmetric algebra on the set $S$ modulo
the ideal generated by squares of elements in $S$. If we set $\xi=-1$,
we recover $\Lambda_o(S)$, the usual exterior algebra on the set
$S$. If we write $\Lambda_2(S)$ for the $\ZZ_2$-exterior
algebra on the set $S$, then these four algebras fit into a pull-back
diamond analogous to Diagram~\eqref{eq:pullback}.

\subsection{Three Khovanov homology theories}\label{sec:threekhfunc}
We will now recall the definition of odd Khovanov homology, as well as
the definition of the unified theory (in the spirit of odd Khovanov
homology).  Let $L$ be a link diagram with $n$ ordered crossings. Each crossing
$\begin{tikzpicture}[scale=0.04,baseline={([yshift=-.8ex]current
    bounding box.center)}] \draw (0,10) -- (10,0); \node[crossing] at
  (5,5) {}; \draw (0,0) -- (10,10);
\end{tikzpicture}$
can be resolved as the \emph{$0$-resolution}
$\begin{tikzpicture}[scale=0.04,baseline={([yshift=-.8ex]current
    bounding box.center)}]
\draw (0,0) .. controls (4,4) and (4,6) .. (0,10);
\draw (10,0) .. controls (6,4) and (6,6) .. (10,10);
\end{tikzpicture}$ or the \emph{$1$-resolution}
$\begin{tikzpicture}[scale=0.04,baseline={([yshift=-.8ex]current
    bounding box.center)}] \draw (0,0)
      .. controls (4,4) and (6,4) .. (10,0); \draw (0,10) .. controls
      (4,6) and (6,6) .. (10,10);
    \end{tikzpicture}$.  We assume that $L$ is decorated by an
    {\emph{orientation of crossings}}, i.e., a choice of an arrow at
    each crossing, $\begin{tikzpicture}[scale=0.04,baseline={([yshift=-.8ex]current
    bounding box.center)}] \draw (0,10) -- (10,0); \node[crossing] at
  (5,5) {}; \draw (0,0) -- (10,10);
  \draw[thick,->,red] (0,5) --(10,5);
\end{tikzpicture}$
or $\begin{tikzpicture}[scale=0.04,baseline={([yshift=-.8ex]current
    bounding box.center)}] \draw (0,10) -- (10,0); \node[crossing] at
  (5,5) {}; \draw (0,0) -- (10,10);
  \draw[thick,<-,red] (0,5) --(10,5);
\end{tikzpicture}$, connecting the two arcs of the $0$-resolution
    at that crossing. Rotating the arrows $90^\circ$ degrees clockwise
    (this requires choosing an orientation of the plane) produces an
    arrow joining the two arcs of the $1$-resolution at that crossing
    as well. That is, a crossing $\begin{tikzpicture}[scale=0.04,baseline={([yshift=-.8ex]current
    bounding box.center)}] \draw (0,10) -- (10,0); \node[crossing] at
  (5,5) {}; \draw (0,0) -- (10,10);
  \draw[thick,->,red] (0,5) --(10,5);
\end{tikzpicture}$ (respectively, $\begin{tikzpicture}[scale=0.04,baseline={([yshift=-.8ex]current
    bounding box.center)}] \draw (0,10) -- (10,0); \node[crossing] at
  (5,5) {}; \draw (0,0) -- (10,10);
  \draw[thick,<-,red] (0,5) --(10,5);
\end{tikzpicture}$) has $0$-resolution $\begin{tikzpicture}[scale=0.04,baseline={([yshift=-.8ex]current
    bounding box.center)}]
\draw (0,0) .. controls (4,4) and (4,6) .. (0,10);
\draw (10,0) .. controls (6,4) and (6,6) .. (10,10);
  \draw[thick,->,red] (3,5) --(7,5);
\end{tikzpicture}$ (respectively, $\begin{tikzpicture}[scale=0.04,baseline={([yshift=-.8ex]current
    bounding box.center)}]
\draw (0,0) .. controls (4,4) and (4,6) .. (0,10);
\draw (10,0) .. controls (6,4) and (6,6) .. (10,10);
  \draw[thick,<-,red] (3,5) --(7,5);
\end{tikzpicture}$) and \emph{$1$-resolution}
$\begin{tikzpicture}[scale=0.04,baseline={([yshift=-.8ex]current
    bounding box.center)}] \draw (0,0)
      .. controls (4,4) and (6,4) .. (10,0); \draw (0,10) .. controls
      (4,6) and (6,6) .. (10,10);
      \draw[thick,red,->] (5,7)--(5,3);
    \end{tikzpicture}$ (respectively, $\begin{tikzpicture}[scale=0.04,baseline={([yshift=-.8ex]current
    bounding box.center)}] \draw (0,0)
      .. controls (4,4) and (6,4) .. (10,0); \draw (0,10) .. controls
      (4,6) and (6,6) .. (10,10);
      \draw[thick,red,<-] (5,7)--(5,3);
    \end{tikzpicture}$).

    We will recall the `top-down' construction of three functors
    \[
    \AbFunc_u : (\two^n)^\op \longrightarrow
    \ZZ_u\Mod, \qquad \AbFunc_e : (\two^n)^\op
    \longrightarrow \ZZ\Mod, \qquad \AbFunc_o :
    (\two^n)^\op \longrightarrow \ZZ\Mod,
    \]
    by first defining the unified functor $\AbFunc_u$, and then
    defining $\AbFunc_e$ and $\AbFunc_o$ by restricting
    scalars $\xi=+1$ and $\xi=-1$, respectively. (Alternatively, one
    can also carry out a bottom-up approach, defining $\AbFunc_e$
    and $\AbFunc_o$, and then defining $\AbFunc_u$ as the
    pullback.)

    For each $v\in\two^n$, let $L_v$ be the complete resolution
    diagram formed by taking the $0$-resolution at the $i\th$ crossing
    if $v_i=0$, or the $1$-resolution otherwise. The diagram $L_v$ is
    a planar diagram of embedded circles and oriented arcs. Write
    $Z(L_v)$ for the set of circles in $L_v$. 

    For objects $v\in\two^n$, set $\AbFunc_u(v) =
    \Lambda_u(Z(L_v))$. For morphisms, first consider the following
    assignment $\AbFunc_u'$ on the edges; the actual functor
    $\AbFunc_u$ will be a slight modification of this
    assignment. Let $\phi_{v,w}$ be an edge in $\two^n$, so that its
    reverse $\phi_{w,v}^\op$ is a morphism in the opposite
    category.  Suppose $\phi_{w,v}^\op$ corresponds to a split
    cobordism, so that some circle $a\in Z(L_w)$ splits into two
    circles $a_1,a_2\in Z(L_v)$, and that the other elements of these
    two sets of circles are naturally identified. Suppose further that
    the arc in $L_v$ associated to this splitting is pointing from
    $a_1$ to $a_2$. Define
    \[
    \AbFunc_u'(\phi^\op_{w,v})(x) \;  = \; (a_1 + \xi a_2)\otimes x 
    \]
	where $\Lambda_u(Z(L_w))$ is viewed embedded in $\Lambda_u(Z(L_v))$ by sending $a$ to either $a_1$ or $a_2$.
    Now suppose instead we
    have a merge cobordism, so that two circles $a_1,a_2\in Z(L_w)$
    merge into one circle $a\in Z(L_v)$, and that the other elements
    in these two sets of circles are naturally identified. Define
    $\AbFunc'_u(\phi_{w,v}^\op)$ to be the
    $\ZZ_u$-algebra map $\Lambda_u(Z(L_w))\to
    \Lambda_u(Z(L_v))$ determined by sending $a_1$ and $a_2$ to $a$,
    and by the identity map on other circle generators.  The
    assignment $\AbFunc'_u$ on the edges does not commute
    across the 2-dimensional faces; rather, it does so only up to
    possible multiplication by $\xi$.  We correct the assignment
    $\AbFunc'_u$ on morphisms as follows.

    The two-dimensional configurations
    can be divided into four categories as follows (with unoriented
    arcs being orientable in either direction).
   \begin{equation}\label{diagram:type}
     \begin{split}
       A:\quad&
       \begin{tikzpicture}[scale=0.06,baseline={([yshift=-.8ex]current  bounding  box.center)}]
         \draw (0,0) circle (5cm);
         \draw (11,0) circle (5cm);
         \draw[thick,red] (-5,0) -- (5,0);
         \draw[thick,red] (11-5,0) -- (11+5,0);
       \end{tikzpicture},
       \begin{tikzpicture}[scale=0.06,baseline={([yshift=-.8ex]current  bounding  box.center)}]
         \draw (0,0) circle (5cm);
         \clip (0,0) circle (5cm);
         \draw[thick,red] (-5,0) circle (4cm);
         \draw[thick,red] (5,0) circle (4cm);
       \end{tikzpicture},
       \begin{tikzpicture}[scale=0.06,baseline={([yshift=-.8ex]current  bounding  box.center)}]
         \node[inner sep=0pt, outer sep=0pt,draw,shape=circle,minimum width=0.6cm,style={transform shape=False}] (a) at (0,0) {};
         \node[inner sep=0pt, outer sep=0pt,draw,shape=circle,minimum width=0.6cm,style={transform shape=False}] (b) at (15,0) {};
         \draw[thick,red,->] (a) to[out=30,in=150] (b);
         \draw[thick,red,->] (a) to[out=-30,in=-150] (b);
       \end{tikzpicture},
       \begin{tikzpicture}[scale=0.06,baseline={([yshift=-.8ex]current  bounding  box.center)}]
         \draw[thick,red] (-5,0) circle (4cm);
         \draw[fill=white] (0,0) circle (5cm);
         \draw[thick,red] (0,5) -- (0,-5);
       \end{tikzpicture}.
       \\
       C:\quad&
       \begin{tikzpicture}[scale=0.06,baseline={([yshift=-.8ex]current  bounding  box.center)}]
         \draw (0,0) circle (5cm);
         \draw (15,0) circle (5cm);
         \draw (30,0) circle (5cm);
         \draw[thick,red] (5,0) -- (10,0);
         \draw[thick,red] (20,0) -- (25,0);
       \end{tikzpicture},
       \begin{tikzpicture}[scale=0.06,baseline={([yshift=-.8ex]current  bounding  box.center)}]
         \draw (0,0) circle (5cm);
         \draw (15,0) circle (5cm);
         \draw[thick,red] (5,0) -- (10,0);
         \draw[thick,red] (15,5) -- (15,-5);
       \end{tikzpicture},
       \begin{tikzpicture}[scale=0.06,baseline={([yshift=-.8ex]current  bounding  box.center)}]
         \node[inner sep=0pt, outer sep=0pt,draw,shape=circle,minimum width=0.6cm,style={transform shape=False}] (a) at (0,0) {};
         \node[inner sep=0pt, outer sep=0pt,draw,shape=circle,minimum width=0.6cm,style={transform shape=False}] (b) at (15,0) {};
         \draw[thick,red,<-] (a) to[out=30,in=150] (b);
         \draw[thick,red,->] (a) to[out=-30,in=-150] (b);
       \end{tikzpicture},
       \begin{tikzpicture}[scale=0.06,baseline={([yshift=-.8ex]current  bounding  box.center)}]
         \draw (0,0) circle (5cm);
         \draw (15,0) circle (5cm);
         \draw (26,0) circle (5cm);
         \draw (41,0) circle (5cm);
         \draw[thick,red] (5,0) -- (10,0);
         \draw[thick,red] (31,0) -- (36,0);
       \end{tikzpicture},
       \begin{tikzpicture}[scale=0.06,baseline={([yshift=-.8ex]current  bounding  box.center)}]
         \draw (0,0) circle (5cm);
         \draw (11,0) circle (5cm);
         \draw (26,0) circle (5cm);
         \draw[thick,red] (16,0) -- (21,0);
         \draw[thick,red] (-5,0) -- (5,0);
       \end{tikzpicture}.
       \\
       X:\quad&
       \begin{tikzpicture}[scale=0.06,baseline={([yshift=-.8ex]current  bounding  box.center)}]
         \node[inner sep=0pt, outer sep=0pt,draw,shape=circle,minimum width=0.6cm,style={transform shape=False}] (a) at (0,0) {};
         \draw[thick,red,<-] (-5,0) --(5,0);
         \draw[thick,red,->] (a) to[out=150,in=90] (-10,0) to[out=-90,in=210] (a);
       \end{tikzpicture}.
       \\
       Y:\quad&
       \begin{tikzpicture}[scale=0.06,baseline={([yshift=-.8ex]current  bounding  box.center)}]
         \node[inner sep=0pt, outer sep=0pt,draw,shape=circle,minimum width=0.6cm,style={transform shape=False}] (a) at (0,0) {};
         \draw[thick,red,<-] (-5,0) --(5,0);
         \draw[thick,red,<-] (a) to[out=150,in=90] (-10,0) to[out=-90,in=210] (a);
       \end{tikzpicture}.
     \end{split}
   \end{equation}
   For the type-A faces, $\AbFunc'_u$ commutes after multiplication by
   $\xi$, for the type-C faces $\AbFunc'_u$ commutes directly, while
   for the type-X and type-Y faces, $\AbFunc'_u$ commutes, both
   directly, and after multiplication by $\xi$.  Define $\psi_X$
   (respectively, $\psi_Y$), an element of $\cellC^2([0,1]^n;\{1,\xi\})$, to be $\xi$ for the type-A or -X faces
   (respectively, type-A or -Y faces), and $1$ for the type-C or -Y
   faces (respectively, type-C or -X faces).

\begin{defn}\label{def:edgeass} 
  A type-X (respectively, type-Y) {\emph{edge assignment}} for the
  diagram $L$ with oriented crossings is a (multiplicative) cochain
  $\epsilon\in \cellC^1([0,1]^n;\{1,\xi\})$ such that
  $\delta\epsilon=\psi_X$ (respectively, $=\psi_Y$).
\end{defn}

Fix an edge assignment $\epsilon$, either of type-X or type-Y. For an edge $\phi_{w,v}^\op$, set
\[
\AbFunc_u(\phi_{w,v}^\op) = \epsilon(\phi_{w,v}^\op)\AbFunc'_u(\phi_{w,v}^\op)
\]
and this defines the functor $\AbFunc_u$.  Setting $\xi=+1$ and
$\xi=-1$ throughout the above construction defines the even and odd
functors $\AbFunc_e$ and $\AbFunc_o$, respectively.

The three Khovanov homology theories are then defined from these
functors as follows. First, we define for $\bullet \in \{u,e,o\}$ a
chain complex:
\[
\KhCx_\bullet(L) \; = \; \bigoplus_{v\in\two^n}
\AbFunc_\bullet(v) , \qquad\;\; \partial_\bullet \; = \;
\sum_{v\geqslant_1 w}
(-1)^{s_{v,w}}\,\AbFunc_\bullet(\phi_{w,v}^\op)
\]
(Here $s$ is the standard sign assignment from
Definition~\ref{def:signassign}.) This complex depends on the ordering
of the crossings, the choice of crossing orientations, and the choice
of edge assignment, as does the corresponding functor. However, these
choices do not affect the resulting chain homotopy type \cite[Theorem
1]{kho1}, \cite[Theorem 1.3]{ors}, \cite[\S7]{putyra}.

\begin{defn}
  For $\bullet\in\{u,e,o\}$ define $\Kh_\bullet(L) =
  H_\ast(\KhCx_\bullet(L),\partial_\bullet)$.
\end{defn}

The unified homology group $\unih(L)$ is a $\ZZ_u$-module,
while the even and odd theories are abelian groups. Each theory has a
bigrading, defined in the usual way, as follows. Let $n_-$ be the
number of negative crossings $\begin{tikzpicture}[scale=0.04,baseline={([yshift=-.8ex]current
    bounding box.center)}] \draw[->] (0,10) -- (10,0); \node[crossing] at
  (5,5) {}; \draw[->] (0,0) -- (10,10);
\end{tikzpicture}$ in the diagram $L$. The homological and quantum
gradings, denoted $i$ and $j$ respectively, are defined on
$x=a_1\otimes \cdots \otimes a_k\in\Lambda_\bullet(Z(L_v))$ with
$a_1,\dots,a_k\in Z(L_v)$ to be
\[
	i(x) \; = \; |v|-n_-,\qquad j(x) \; = \; |Z(L_v)|-2k+|v|+n-3n_-.
\]
We write $\Kh_\bullet^{i,j}(L)$ for the corresponding bigraded
module. We note that the unified theory in \cite{putyra} is called the
{\emph{covering homology}}, and is more specfically obtained from
Example 10.7 of loc.~cit.~by setting $X=Z=1$ and $Y=\xi$. The
terminology {\emph{unified}} is used in \cite{putyrashumakovitch}.

\begin{rmk}\label{rmk:edgeass}
  Our definition of an edge assignment is non-standard, but the
  standard type-X (respectively, type-Y) edge assignment from
  \cite{ors,putyra} may be obtained from our type-X (respectively,
  type-Y) edge assignment by multiplying by $(-1)^{s_{v,w}}$.
\end{rmk}

\subsection{Relations between the theories}

From the definitions, it is clear that the chain complexes
$\KhCx_\bullet(L)$ above fit into a pull-back diagram just as in
Diagram~\eqref{eq:pullback}, 
\begin{equation}
\begin{tikzpicture}[baseline={([yshift=-.8ex]current  bounding  box.center)},xscale=2.5,yscale=1.5]
\node (a0) at (0,0) {$\KhCx_u(L)$};
\node (b1) at (-0.5,-0.75) {$\KhCx_e(L)$};
\node (b2) at (0.5,-0.75) {$\KhCx_o(L)$};
\node (c0) at (0,-1.5) {$\KhCx_2(L)$};

\draw[->] (a0) -- (b1) node[pos=0.2,anchor=east] {\tiny$\xi=+1\;\;$};
\draw[->] (a0) -- (b2) node[pos=0.2,anchor=west] {\tiny$\;\;\xi=-1$};
\draw[->] (b1) -- (c0) node[pos=0.5,anchor=south] {};
\draw[->] (b2) -- (c0) node[pos=0.5,anchor=east] {};

\end{tikzpicture}
\end{equation}
where $\KhCx_2(L)$ denotes Khovanov chain complex with $\ZZ_2$
coefficients. Indeed, we may define $\KhCx_u(L)$ to be the pullback of
$\KhCx_e(L)$ and $\KhCx_o(L)$ over $\KhCx_2(L)$,
\[
\KhCx_u(L)=\{(a,b)\in\KhCx_e(L)\oplus\KhCx_o(L)\mid a\equiv b\bmod{2}\},
\]
which then naturally inherits a $\ZZ_2$-action $\xi(a,b)=(a,-b)$.

We also have a short exact sequence of chain complexes
$0\to \KhCx_e(L)\to \KhCx_u(L)\to \KhCx_o(L)\to 0$. This may be viewed
as arising from tensoring the short exact sequence $0\to
\ZZ_e\to\ZZ_u\to\ZZ_o\to 0$ by the unified chain
complex $\KhCx_u(L)$ over $\ZZ_u$. There is a similar exact
sequence with $e$ and $o$ swapped. Passing to homology yields the
following long exact sequences, cf.~\cite{putyrashumakovitch}:
\begin{equation}\label{eq:les_oe}
	\cdots\to \Kh_e^{i,j}(L) \longrightarrow \Kh^{i,j}_u(L) \longrightarrow \Kh_o^{i,j}(L) \xrightarrow{\phi_{oe}} \Kh_e^{i+1,j}(L) \to\cdots
\end{equation}
\begin{equation}\label{eq:les_eo}
	\cdots\to \Kh_o^{i,j}(L) \longrightarrow \Kh^{i,j}_u(L) \longrightarrow \Kh_e^{i,j}(L) \xrightarrow{\phi_{eo}} \Kh_o^{i+1,j}(L) \to\cdots
\end{equation}

\subsection{Reduced theories} Let $p$ be a basepoint on the planar
diagram $L$. Then the reduced complex $\widetilde{\KhCx}_\bullet(L,p)$
for $\bullet\in \{u,e,o\}$ is the subcomplex of $\KhCx_\bullet(L)$
consisting of elements of the form $a\otimes y$, where $a$ is a circle
containing $p$ in a resolution diagram, and $y$ is any other
element. The complex $\widetilde{\KhCx}_\bullet(L)$ is homologically
graded as a subcomplex of $\KhCx_\bullet(L)$, but there is a shift in its
quantum grading, which is defined as \emph{one more} 
than the formula for $j(x)$ above. The reduced functors
$\widetilde{\AbFunc}_\bullet$ are defined in the same way.

The chain homotopy type depends on the isotopy type of $L$ and which
component of the link $p$ lies in. In the odd case, the basepoint does
not matter, and the chain complex is a direct sum \cite[Prop
1.7]{ors}:
\begin{equation}
	\KhCx_o^{*,j}(L) \; = \; \widetilde{\KhCx}_o^{*,j-1}(L) \oplus \widetilde{\KhCx}_o^{*,j+1}(L) \label{eq:splitting_hom}
\end{equation}
In contrast to the odd case, the unified and even theories do not
split into a direct sum of their reduced theories. Instead, for
$\bullet\in\{u,e\}$, there is a short exact sequence of chain
complexes
\[
0\longrightarrow \widetilde{\KhCx}_\bullet^{*,j+1}(L,p)
\longrightarrow \KhCx^{*,j}_\bullet(L) \longrightarrow
\widetilde{\Kh}_\bullet^{*,j-1}(L,p) \longrightarrow 0.
\]

The reduced Khovanov homology is defined as
\begin{defn}
  For $\bullet\in\{u,e,o\}$ define $\widetilde{\Kh}^{i,j}_\bullet(L,p)
  = H_i(\widetilde{\KhCx}^{\ast,j}_\bullet(L,p),\partial_\bullet)$.
\end{defn}

\subsection{Khovanov generators}\label{subsec:kg} In the sequel, we
will need to fix bases of these chain complexes to facilitate the
construction of the various Khovanov spectra. For even Khovanov
homology, there is a natural basis of generators: the elements
$a_1\otimes \cdots \otimes a_k \in \Lambda_e(Z(L_v))$ where each
$a_i\in Z(L_v)$ is distinct. Since $\Lambda_e(Z(L_v))$ is the quotient
of a symmetric algebra on these generators, their order does not
matter. For the unified and odd cases, order matters, however: recall
that $a_1\otimes a_2 = \xi a_2 \otimes a_1$ in the unified case, and
$a_1\otimes a_2 = - a_2 \otimes a_1$ in the odd case. To fix
generators, we will thus fix at every vertex $v\in\two^n$ a
total ordering $>$ of the set $Z(L_v)$. Once this is done, we write
\[
\KhGen(v) \; = \; \KhGen_\bullet(v) \; := \; \{ a_{1}\otimes \cdots
\otimes a_{k} : \; a_i \in Z(L_v), \; a_1 > \cdots > a_k \} \qquad
\bullet \in \{u,e,o\}
\]
for the set of {\emph{Khovanov generators at $v$}}. As indicated, we
will often omit the subscript $\bullet$ from the notation, as each of
the three sets for a fixed $v$ are naturally identified (once the
circles in $Z(L_v)$ are totally ordered). Note that in the unified
case, the set of Khovanov generators over all $v\in\two^n$
gives a $\ZZ_u$-basis for the chain complex. On the other hand,
a $\ZZ$-basis for the unified chain complex at $v\in\two^n$ is given by
\[
	\KhGen(v) \;\amalg\; \xi \KhGen(v)
\]
where $\xi \KhGen(v)$ is the set of $\xi x$ with $x\in
\KhGen(v)$. Note that $\KhGen(v)$ has $2^{|Z(L_v)|}$ elements. Given a
basepoint $p$ on our diagram $L$, we can also form the set of reduced
generators $\widetilde{\KhGen}(v;p)$ at the vertex $v$, the subset of
$\KhGen(v)$ whose elements each include the circle containing the
basepoint. This set has half the number of elements of $\KhGen(v)$,
and, running over all $v\in\two^n$, forms a basis for the
reduced complex $\widetilde{\KhCx}_\bullet(L,p)$ where $\bullet \in
\{e,o\}$, and a $\ZZ_u$-basis for $\widetilde{\KhCx}_u(L,p)$.

\section{Burnside categories and functors} \label{sec:burn}

In this section, we review the definition of the Burnside category
$\burn$ from \cite{lls1,lls2}, and define some slight modifications:
the signed Burnside category $\oddb$ and the $\ZZ_2$-equivariant
Burnside category $\eqb$. We then discuss functors from the cube
category $\two^n$ to these Burnside categories.

\subsection{The Burnside category}\label{subsec:burnside}
Given finite sets $X$ and $Y$, a correspondence from $X$ to $Y$ is a
triple $(A,s,t)$ for a finite set $A$, where $s,t$ are set maps $s
\from A\to X$ and $t \from A \to Y$; $s$ and $t$ are called the
\emph{source} and \emph{target} maps, respectively. The correspondence
$(A,s,t)$ is depicted:
\[
\begin{tikzpicture}[scale=1]
\node (X) at (-2,0) {$X$};
\node (Y) at (2,0) {$Y$};
\node (A) at (0,1) {$A$};
\draw[->] (A) -- (X) node[midway, above] {$s_A\;\;\;$};
\draw[->] (A) -- (Y) node[midway,above] {$\;\;t_A$};
\end{tikzpicture}
\]

For correspondences $(A,s_A,t_A)$ and $(B,s_B,t_B)$ from $X$ to $Y$
and $Y$ to $Z$, respectively, define the composition $(B,s_B,t_B)\circ
(A,s_A,t_A)$ to be the correspondence $(C,s,t)$ from $X$ to $Z$ given
by the fiber product $C=B \times_Y A =\{ (b,a) \in B \times A \mid
t(a) = s(b)\}$ with source and target maps $s(b,a)=s_A(a)$ and
$t(b,a)=t_B(b)$. There is also the identity correspondence from a set
$X$ to itself, i.e., $(X,\Id_X,\Id_X)$ from $X$ to $X$.  Given
correspondences $(A,s_A,t_A)$, $(B,s_B,t_B)$ from $X$ to $Y$, a
morphism of correspondences $(A,s_A,t_A)$ to $(B,s_B,t_B)$ is a
bijection $f \from A \to B$ commuting with the source and target maps.
There is also the identity morphism from a correspondence to itself.

Composition (of set maps) gives the set of correspondences from $X$ to
$Y$ the structure of a category.  Define the \emph{Burnside category}
$\burn$ to be the weak $2$-category whose objects are finite sets,
morphisms are finite correspondences, and $2$-morphisms are maps of
correspondences.

Recall that in a weak $2$-category, that arrows need only be
associative up to an equivalence, and similarly the identity axiom
holds only after composing with a $2$-morphism.  To be explicit, for
finite sets $X,Y$ and $(A,s,t)$ a correspondence from $X$ to $Y$,
neither $(Y,\Id_Y,\Id_Y) \circ (A,s,t)$, nor $(A,s,t)\circ
(X,\Id_X,\Id_X)$, equals $(A,s,t)$. Rather, there are natural
$2$-morphisms, left and right unitors,
\[
\lambda\from Y \times_Y A \to A, \qquad \rho\from A \times_X X \to A
\]
given by $\lambda(y,a)=a$ and $\rho(a,x)=a$. Further, the composition
$C\circ (B\circ A)$, for $A$ from $W$ to $X$, $B$ from $X$ to $Y$, and
$C$ from $Y$ to $Z$, is not identical to $(C \circ B) \circ A$, rather
there is an associator
\[
\alpha \from (C \times_Y B)\times_X A \to C \times_Y (B \times_X A)
\]
given by $\alpha((c,b),a)=(c,(b,a))$. The categories to follow are
slight variations of this one. The total diagram of Burnside
categories that we will consider in this article is depicted in Figure
\ref{fig:burnsidecategories}.

\subsection{The signed Burnside category} \label{subsec:oddburn} 
Given sets $X$ and $Y$, a \emph{signed correspondence} is a
correspondence $(A,s_A,t_A)$ equipped with a map
$\sigma_A \from A \to \{+1,-1\},$ regarded as a tuple
$(A,s_A,t_A,\sigma_A)$; we call $\sigma_A$ the ``sign'' or ``sign map''
of the signed correspondence:
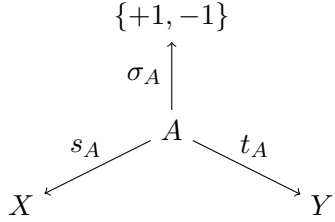
\[
\begin{tikzpicture}[scale=1]
\node (X) at (-2,0) {$X$};
\node (Y) at (2,0) {$Y$};
\node (A) at (0,1) {$A$};
\node (S) at (0,2.5) {$\{+1,-1\}$};
\draw[->] (A) -- (X) node[midway, above] {$s_A\;\;\;$};
\draw[->] (A) -- (Y) node[midway,above] {$\;\;t_A$};
\draw[->] (A) -- (S) node[midway,left] {$\sigma_A$};
\end{tikzpicture}
\]

In the sequel we often write ``correspondence'' for ``signed
correspondence'', where it will not cause any confusion.  We define a
composition $(B,s_B,t_B,\sigma_B)\circ (A,s_A,t_A,\sigma_A)$ of signed
correspondences $(A,s_A,t_A,\sigma_A)$ from $X$ to $Y$, and
$(B,s_B,t_B,\sigma_B)$ from $Y$ to $Z$ by $(C,s,t,\sigma)$, where
$(C,s,t)$ is the composition $(B,s_B,t_B) \circ (A,s_A,t_A)$ and
$\sigma(b,a)=\sigma_B(b)\sigma_A(a)$.  Also, we define the identity
(signed) correspondence by $(X, \Id_X, \Id_X, 1)$ (i.e., the identity
correspondence takes value $1$ on all elements).

We define maps of signed correspondences $f \from (A,s_A,t_A,\sigma_A)
\to (B,s_B,t_B,\sigma_B)$ to be morphisms of correspondences $f \from
(A,s_A,t_A) \to (B,s_B,t_B)$ such that $\sigma_B\circ f=\sigma_A$.  We
may then define the \emph{signed Burnside category} $\oddb$ to be the
weak $2$-category with objects finite sets, morphisms given by signed
correspondences, and $2$-morphisms given by maps of signed
correspondences.  The structure maps $\lambda, \rho,\alpha$ of
\S\ref{subsec:burnside} are easily seen to respect the sign,
confirming that $\oddb$ is indeed a weak $2$-category. There is a
forgetful 2-functor $\forgot\from \oddb \to \burn$ which forgets
signs. There is also an inclusion-induced 2-functor $\subinclude:\burn
\to \oddb$.  We will usually refer to such 2-functors simply as
functors.

\subsection{The $\ZZ_2$-equivariant Burnside category} \label{subsec:equivburn}

We let $\eqb$ denote the $2$-category whose objects are finite, free
$\ZZ_2$-sets, with $\ZZ_2$-equivariant correspondences,
and $2$-morphisms $\ZZ_2$-equivariant bijections of
correspondences. (Recall that we write $\ZZ_2=\{1,\xi\}$.) The
2-category $\eqb$ is a subcategory of the Burnside 2-category for the
group $\ZZ_2$.

There is a forgetful functor $\forgot\from \eqb \to \burn$.  There is
also a ``quotient'' functor $\quotient\from \eqb \to \burn$, which
simply takes the quotient by the action of $\ZZ_2$ on objects,
$1$- and $2$-morphisms.

There is also a strictly unitary ``doubling'' 2-functor $\oddeq\from
\oddb \to \eqb$ consisting of the following data.
\begin{enumerate}[leftmargin=*]
\item For each object $X$ of $\oddb$, we need to specify an object
  $\oddeq(X)$ of $\eqb$. Define $\oddeq(X)=\{1,\xi\}\times X$, with
  the $\ZZ_2$-action on $\{1,\xi\}\times X$ being
  $\xi(1,x)=(\xi,x)$ for all $x\in X$.
\item For any 2 objects $X,Y$ of $\oddb$, we need to specify a
  functor, also denoted $\oddeq$, from $\Hom_{\oddb}(X,Y)$ to
  $\Hom_{\eqb}(\oddeq(X),\oddeq(Y))$ that sends $\Id_X$ to
  $\Id_{\oddeq(X)}$. This functor sends a signed correspondence $A$
  from $X$ to $Y$ to the correspondence $\{1,\xi\}\times A$ from
  $\{1,\xi\}\times X$ to $\{1,\xi\}\times Y$ (the $\ZZ_2$-action is
  similar). The source and target maps on $\{1\}\times A$ are defined
  as
  \begin{align*}
    s(1,a)&=(1,s(a))\\
    t(1,a)&=(\sigma(a),t(a))
  \end{align*}
  (The sign $\sigma(a)$ takes value in
  $\ZZ_2=\{+1,-1\}$, which has been identitified with
  $\ZZ_2=\{1,\xi\}$.)  The source and target maps are then extended
  equivariantly to $\{\xi\}\times A$. Finally,
  $\oddeq$ sends a $2$-morphism $f\from A\to B$ to the
  $2$-morphism $(\Id,f)\from\{1,\xi\}\times A\to\{1,\xi\}\times B$.
\item Finally, for any correspondences $A$ from $X$ to $Y$ and $B$
  from $Y$ to $Z$ in $\oddb$, we need to specify a 2-morphism in
  $\eqb$ from $\oddeq(B)\circ\oddeq(A)=(\{1,\xi\}\times
  B)\times_{(\{1,\xi\}\times Y)}(\{1,\xi\}\times A)$ to $\oddeq(B\circ
  A)=\{1,\xi\}\times (B\times_Y A)$ that is natural in $A$ and
  $B$. Define it to be
  \[
    ((\sigma(a),b),(1,a))\mapsto (1,(b,a)),
  \]
  extended $\ZZ_2$-equivariantly. It is not hard to check that these
  maps satisfy the required coherence relations with the structure
  maps $\lambda,\rho,\alpha$ of $\oddb$ and $\eqb$, as described in
  \cite[Definition 4.1 and Remark
  4.2]{Benabou-other-bicategories}. Therefore, the above is indeed a
  $2$-functor.
\end{enumerate}

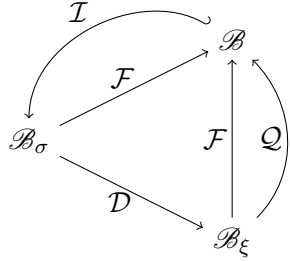
\begin{figure}
\centering
\begin{tikzpicture}[scale=1.35]
\node (b) at (0,1) {$\burn$};
\node (ob) at (-2,0) {$\oddb$};
\node (eqb) at (0,-1) {$\eqb$};

\draw[left hook->] (b) to[out=150,in=90] node[midway,anchor=south east,inner sep=2pt, outer sep=0pt] {$\subinclude$} (ob) ;
\draw[->] (ob) -- node[midway,anchor=south east,inner sep=0pt, outer sep=0pt] {$\forgot$} (b) ;
\draw[->] (ob) -- node[midway,anchor=north east,inner sep=0pt, outer sep=0pt] {$\oddeq$} (eqb) ;
\draw[->] (eqb) -- node[midway,anchor=east,inner sep=2pt, outer sep=0pt] {$\forgot$} (b) ;
\draw[->] (eqb) to[out=45,in=-45] node[midway,anchor=east,inner sep=2pt, outer sep=0pt] {$\quotient$} (b) ;

\end{tikzpicture}
\caption{\textbf{The Burnside categories and some functors between
    them.} We have the relations $\quotient\circ \oddeq = \forgot$ and
  $\forgot\circ\subinclude = \Id$. The $\forgot$'s are forgetful
  functors, $\subinclude$ is a subcategory inclusion, $\quotient$ is a
  quotient functor, and $\oddeq$ stands for doubling.
}\label{fig:burnsidecategories}
\end{figure}

\subsection{Functors from Burnside categories} \label{subsec:2funcs}
For $\bullet\in \{\varnothing,\sigma\}$ we define a functor
$\Abelianize\from \burn_\bullet \to\ZZ\Mod$ by sending a
set $X\in \burn_\bullet$ to the free abelian group
generated by $X$, denoted $\Abelianize(X)$.  For a signed
correspondence $\phi = (A,s,t,\sigma)$ from $X$ to $Y$, we define
$\Abelianize(\phi)\from\Abelianize(X)\to \Abelianize(Y)$ by
\begin{equation}
  \Abelianize(\phi)(x)=\sum_{a\in A\,\mid\,s(a)=x}\sigma(a)t(a)\label{eq:amorph}
\end{equation}
for elements $x \in X$, extended linearly over $\ZZ$.
Similarly, we have a functor $\Abelianize\from\eqb\to
\ZZ_u\Mod$ that sends a free $\ZZ_2$-set $X$ to
$\Abelianize(X)$, which is a free $\ZZ_u$-module; for a
$\ZZ_2$-equivariant correspondence $\phi$, we use the same
formula as Equation~\eqref{eq:amorph}, but exclude $\sigma$, and extend
linearly over $\ZZ_u$.

\subsection{Functors to Burnside categories} We now consider functors from the
cube category $\two^n$ to the Burnside categories thus far
introduced. We let $\burn_\bullet$ be one of the Burnside
categories introduced above, appearing in Figure
\ref{fig:burnsidecategories}, so that
$\bullet\in\{\varnothing,\sigma,\xi\}$. The functors
$F\from\two^n\to \burn_\bullet$ we consider will be strictly unitary
2-functors; that is, they will consist of the following data:
\begin{enumerate}[leftmargin=*]
\item For each vertex $v$ of $\two^n$, an object
  $F(v)$ of $\burn_\bullet$.
\item For any $u\geq v$, a 1-morphism $F(\phi_{u,v})$ in
  $\burn_\bullet$ from $F(u)$ to $F(v)$, such that $F(\phi_{u,u})$
  is the identity morphism $\Id_{F(u)}$.
\item Finally, for any $u\geq v\geq w$, a 2-morphism $F_{u,v,w}$ in
  $\burn_\bullet$ from $F(\phi_{v,w})\circ F(\phi_{u,v})$ to
  $F(\phi_{u,w})$ that agrees with $\lambda$ (respectively, $\rho$)
  when $v=w$ (respectively, $u=v$), and that satisfies, for any
  $u\geq v\geq w\geq z$,
  \[
    F_{u,w,z}\circ_2 (\Id\circ
    F_{u,v,w})=(F_{v,w,z}\circ\Id)\circ_2 F_{u,v,z}
  \]
  (with $\circ$ denoting composition of 1-morphisms and $\circ_2$
  denoting composition of 2-morphisms; and we have suppressed the
  associator $\alpha$).
\end{enumerate}
We will usually use the characterization of these functors in the
lemma to follow.

\begin{lem}\label{lem:212} Let $\burn_\bullet$ be any of the
  Burnside categories with $\bullet\in \{\varnothing,\sigma,\xi\}$.
  Consider the data of objects $F(v)$ for $v \in \two^n$,
  1-morphisms $F(\phi_{v,w})$ for edges $v\geqslant_1 w$, and 2-morphisms
  $F_{u,v,v',w} \from F(\phi_{v,w}) \circ F(\phi_{u,v}) \to
  F(\phi_{v',w}) \circ F (\phi_{u,v'})$ for each 2d face described by
  $u\geqslant_1 v,v'\geqslant_1 w$, such that the following
  compatibility conditions are satisfied:
  \begin{enumerate}[leftmargin=*]
  \item For any 2d face $u,v,v',w$ as above, $F_{u,v,v',w}=F^{-1}_{u,v',v,w}$;
  \item For any 3d face in $\two^n$ on the left, the hexagon on
    the right commutes: \label{cond:c2}
  \end{enumerate}
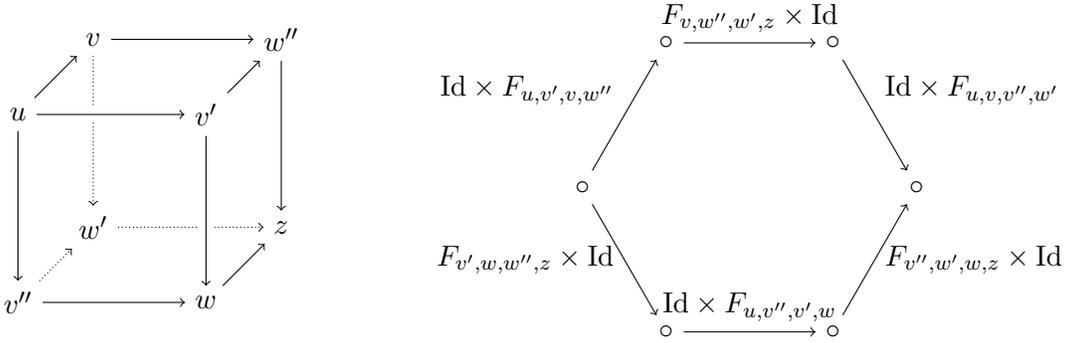
\[
\begin{tikzpicture}[node distance=2.5cm,
  back line/.style={densely dotted},
  cross line/.style={preaction={draw=white, -, line width=6pt}},baseline={(current  bounding  box.center)}]
  \node (u) {$u$};
  \node (v') [right of=u] {$v'$};
  \node [below of=u] (v'') {$v''$};
  \node [right of=v''] (w) {$w$};
  \node (v) [right of=u, above of=u, node distance=1cm] {$v$};
  \node (w'') [right of=v] {$w''$};
  \node (w') [below of=v] {$w'$};
  \node (z) [right of=w'] {$z$};
  \draw[back line, ->] (v) to (w');
  \draw[back line, ->] (v'') to (w');
  \draw[back line, ->] (w') to (z);
  \draw[->] (w'') to (z);
  \draw[cross line, ->] (u) to (v);
  \draw[cross line, ->] (u) to (v');
  \draw[cross line, ->] (u) to (v'');
  \draw[cross line, ->] (v') to (w'');
  \draw[cross line, ->] (v'') to (w);
  \draw[cross line, ->] (v) to (w'');
  \draw[cross line, ->] (w) to (z);
  \draw[cross line, ->] (v') to (w);
\end{tikzpicture}\qquad\qquad
\begin{tikzpicture}[scale=0.6,baseline={(current  bounding  box.center)}]
  \def\radius{3.7cm} 
  \node (h0A) at (60:\radius)   {$\circ$};
  \node (h0C) at (0:\radius)    {$\circ$};
  \node (h1B) at (-60:\radius)  {$\circ$};
  \node (h1A) at (-120:\radius) {$\circ$};
  \node (h1C) at (180:\radius)  {$\circ$};
  \node (h0B) at (120:\radius)  {$\circ$};
  
  \draw[->]
  (h0A) edge node[auto] {$\Id\times F_{u,v,v'',w'}$} (h0C)
  (h1B) edge node[right] {$F_{v'',w',w,z}\times\Id$} (h0C)
  (h1A) edge node[auto] {$\Id\times F_{u,v'',v',w}$} (h1B)
  (h1C) edge node[left] {$F_{v',w,w'',z}\times\Id$} (h1A)
  (h1C) edge node[auto] {$\Id\times F_{u,v',v,w''}$} (h0B)
  (h0B) edge node[auto] {$F_{v,w'',w',z}\times\Id$} (h0A);
\end{tikzpicture}
\]
This collection of data can be extended to a strictly unitary functor
$F\from\two^n\to\burn_\bullet$, uniquely up to natural
isomorphism, so that $F_{u,v,v',w}=F_{u,v',w}^{-1}\circ_2 F_{u,v,w}$.
\end{lem}

\begin{proof}
  The proof is same as that of \cite[Lemma 2.12]{lls1} and \cite[Lemma
  4.2]{lls2}.
\end{proof}

\subsection{Totalization} Given a functor
$F\from\two^n \to \burn_\bullet$ we construct a chain complex
denoted $\Tot(F)$, and called the {\emph{totalization}} of the
functor $F$. The underlying chain group of $\Tot(F)$ is given by
\[
\Tot(F) \;= \; \bigoplus_{v \in \two^n} \Abelianize(F(v)).
\]
We set the homological grading of the summand $\Abelianize(F(v))$ to
be $|v|$. The differential is given by defining the components
$\partial_{u,v}$ from $\Abelianize(F(u))$ to $\Abelianize(F(v))$ by
\[
\partial_{u,v}=\begin{cases}
(-1)^{s_{u,v}}\Abelianize(F(\phi_{u,v})) & \mbox{if } u\geqslant_1 v \\
0 &\mbox{otherwise.}
\end{cases}
\]
Note that for a functor $F\from\two^n\to \eqb$, the totalization
$\Tot(F)$ is a chain complex over $\ZZ_u$.

\subsection{Natural transformations}\label{sec:nat-transform-burn}
The following will serve as the
basic relation between functors from the cube to a Burnside
category. As before, $\bullet\in\{\varnothing,\sigma,\xi\}$.

\begin{defn}\label{def:nattrans}
  A \emph{natural transformation} $\eta \from F_1 \to F_0$ between
  $2$-functors $F_1,F_0\from \two^n \to \burn_\bullet$ is a
  strictly unitary $2$-functor
  $\eta\from \two^{n+1} \to \burn_\bullet$ so that
  $\eta|_{\{1\}\times \two^n}=F_1$ and
  $\eta|_{\{0\}\times \two^n}=F_0$.
\end{defn}

A natural transformation (functorially) induces a chain map between
the chain complexes of Burnside functors, which we write as
$\Tot(\eta)\from\Tot(F_1)\to \Tot(F_0)$. 

Many of the natural transformations we will encounter will be
sub-functor inclusions or quotient functor surjections. Given a
functor $F\from\two^n\to\burn_\bullet$, a \emph{sub-functor}
(respectively, \emph{quotient functor}) $G\from\two^n\to\burn_\bullet$
is a functor that satisfies:
\begin{enumerate}[leftmargin=*]
\item $G(v)\subset F(v)$ for all $v\in\two^n$; if $\bullet=\xi$, the
  $\ZZ_2$-action on $G(v)$ is induced from the $\ZZ_2$-action on
  $F(v)$.
\item $G(\phi_{u,v})\subset F(\phi_{u,v})$ for all $u\geq v$, with the
  source and target maps (and the sign map if $\bullet=\sigma$ or the
  $\ZZ_2$-action if $\bullet=\xi$) being restrictions of the
  corresponding ones on $F(\phi_{u,v})$.
\item $s^{-1}(x)\subset G(\phi_{u,v})$ (respectively,
  $t^{-1}(y)\subset G(\phi_{u,v})$) for all $u\geq v$ and for all $x\in
  G(u)$ (respectively, $y\in G(v)$).
\end{enumerate}
If $G$ is a sub (respectively, quotient) functor of $F$, then there is
a natural transformation $G\to F$ (respectively, $F\to G$), and the
induced chain map $\Tot(G)\to \Tot(F)$ (respectively,
$\Tot(F)\to\Tot(G)$) is an inclusion (respectively, a quotient
map). 
\begin{defn}\label{def:burn-cofib-sequence}
  If $G$ is a sub-functor of $F\from\two^n\to\burn_\bullet$, then the
  functor $H$ defined as $H(v)=F(v)\setminus G(v)$ and
  $H(\phi_{u,v})=\cup_{y\in H(v)}t^{-1}(y)\subset
  F(\phi_{u,v})\setminus G(\phi_{u,v})$ is a
  quotient functor of $F$ (and vice-versa). Such a sequence
  \[
  G\to F\to H
  \]
  is called a \emph{cofibration sequence} of Burnside functors; it
  induces the short exact sequence
  \[
  0\to\Tot(G)\to\Tot(F)\to\Tot(H)\to0
  \]
  of chain complexes.
\end{defn}

The following is another particular example of a natural
transformation which will appear later. Suppose we are given a functor
$F\from\two^n\to\oddb$. For each object $v\in\two^n$, choose a
function $\zeta_v\from F(v)\to \{+1,-1\}$. Define a new functor
$F'\from\two^n\to\oddb$ that is equal to $F$ except that in the
correspondence $F(\phi_{v,w})=(A,s,t,\sigma)$ for $v\geqslant_1 w$ we
change the sign function $\sigma$ to be
$\sigma'(x)=\zeta_v(s(x))\zeta_w(t(x))\sigma(x)$. There is a naturally
induced natural transformation $\eta\from F\to F'$:

\begin{defn}\label{def:signchange}
A \emph{sign reassignment} $\eta$ of $F\from \two^n \to \oddb$ is a natural transformation $\eta\from F\to F'$ as described above, induced by a function $\zeta_v\from F(v) \to \{ \pm 1\}$ for each $v\in \two^n$.
\end{defn}

In the context of Morse theory, a sign reassignment as above
corresponds to changing the orientation on the stable tangent bundle
to a critical point in Morse theory. In the sequel, the appearance of
sign reassignments will be specific to odd Khovanov homology. Such
reassignments are not necessary in the (even) Khovanov setting, since
in that case there is a preferred choice of signs: the (even) Khovanov
complex comes equipped with a choice of generators for which all signs
in the differentials, apart from the standard sign assignment, are
positive, cf.~\S\ref{subsec:kg}. In the odd Khovanov setting, in
which there are generally no such positive bases, sign reassignments
inevitably appear.

\subsection{Stable equivalence of functors} In the sequel, we will be
interested not just in functors
$F\from \two^n\to \burn_\bullet$, but {\emph{stable}} functors,
which are pairs $(F,r)$ with $r\in\ZZ$. We will write
$\Sigma^r F$ for the pair $(F,r)$; its totalization is defined to be
$\Sigma^r\Tot(F)=\Tot(F)[r]$, that is, the chain complex $\Tot(F)$
shifted up by $r$. In this section we describe when two such stable
functors are equivalent, only slightly modifying \cite[Definition
5.9]{lls2} by allowing $\bullet\in\{\varnothing,\sigma,\xi\}$.

A \emph{face inclusion} $\iota$ is a functor $\two^n \to \two^N$ that
is injective on objects and preserves the relative gradings. We remark
that self-equivalences $\iota\from \two^n \to \two^n$ are face
inclusions. Now consider a face inclusion $\iota\from \two^n \to
\two^N$ and a functor $F \from \two^n \to \burn_\bullet$. The
induced functor $F_\iota\from \two^N \to \burn_\bullet$ is
uniquely determined by requiring that $F=F_\iota\circ \iota$, and such
that for $v\in \two^N/ \iota(\two^n)$, we have
$F_\iota(v)=\varnothing$. For a face inclusion $\iota$, we define
$|\iota|=|\iota(v)|-|v|$ for any $v\in\two^n$, which is
independent of $v$ because $\iota$ is assumed to preserve relative
gradings.  As observed in \cite[\S5]{lls2}, for any face inclusion
$\iota$ and functor $F$ as above,
\[
	\Tot(F_\iota) \; \cong \; \Sigma^{|\iota|}\Tot(F)
\]
where the isomorphism is natural up to certain sign choices. With this
background, we state the relevant notion of equivalence for stable
functors.

\begin{defn}\label{def:stableq}
  Two stable functors
  $(E_1 \from \two^{m_1} \to \burn_\bullet, q_1)$ and
  $(E_2\from \two^{m_2} \to \burn_\bullet, q_2)$ are \emph{stably
    equivalent} if there is a sequence of stable functors
  $\{(F_i \from \two^{n_i} \to \burn_\bullet, r_i)\}$
  ($0\leq i \leq \ell$) with $\Sigma^{q_1}E_1=\Sigma^{r_0}F_0$ and
  $\Sigma^{q_2}E_2=\Sigma^{r_\ell}F_\ell$ such that for each pair
  $\{ \Sigma^{r_i}F_i, \Sigma^{r_{i+1}}F_{i+1}\}$, one of the
  following holds:
\begin{enumerate}[leftmargin=*]
\item $(n_i,r_i)=(n_{i+1},r_{i+1})$ and there is a natural
  transformation $\eta\from F_i\to F_{i+1}$ or
  $\eta\from F_{i+1} \to F_i$ such that the induced map $\Tot(\eta)$
  is a chain homotopy equivalence.
\item There is a face inclusion
  $\iota\from \two^{n_i} \hookrightarrow \two^{n_{i+1}}$ such that
  $r_{i+1}=r_i-|\iota|$ and $F_{i+1}=(F_i)_\iota$; or a face inclusion
  $\iota\from \two^{n_{i+1}} \hookrightarrow \two^{n_{i}}$ such that
  $r_{i}=r_{i+1}-|\iota|$ and $F_{i}=(F_{i+1})_\iota$.
\end{enumerate}
We call such a sequence, along with the arrows between
$\Sigma^{r_i}F_i$, a {\emph{stable equivalence}} between the stable
functors $\Sigma^{q_1}E_1$ and $\Sigma^{q_2}E_2$.  If
$\bullet=\sigma$, and if the sequence is such that the maps $\eta$
satisfy $\Tot(\oddeq \eta)$ are chain homotopy equivalences over
$\ZZ_u$ (where $\oddeq\from\oddb\to\eqb$ is from Figure
\ref{fig:burnsidecategories}), we call it a {\emph{equivariant
    (stable) equivalence}}, and say that $\Sigma^{q_i}E_i$ are
\emph{equivariantly equivalent}.
\end{defn}

We note that a stable equivalence from $\Sigma^{q_1}E_1$ to
$\Sigma^{q_2}E_2$ induces a chain homotopy equivalence
$\Tot(\Sigma^{q_1}E_1)\to \Tot(\Sigma^{q_2}E_2)$, well-defined up to
choices of inverses of the chain homotopy equivalences involved in its
construction, and an overall sign.  Note that for $\eqb$, the category
of chain complexes under consideration is over $\ZZ_u$.

We will also need the notion of a map of Burnside functors:

\begin{defn}\label{def:map-of-burnside-functors}
A \emph{map} $\Sigma^{q_1}E_1 \to \Sigma^{q_2}E_2$ of Burnside functors $\Sigma^{q_i}E_i \from \two^{m_i} \to \burn_\bullet$ consists of a sequence of stable functors
  $\{(F_i \from \two^{n_i} \to \burn_\bullet, r_i)\}$
  ($0\leq i \leq \ell$) with $\Sigma^{q_1}E_1=\Sigma^{r_0}F_0$ and
  $\Sigma^{q_2}E_2=\Sigma^{r_\ell}F_\ell$ along with the following:
  \begin{enumerate}[leftmargin=*]
  \item For $i$ even, a stable equivalence from $\Sigma^{r_i}F_i$ to
    $\Sigma^{r_{i+1}}F_{i+1}$.
  \item For $i$ odd, a natural transformation $F_i\to F_{i+1}$, and we
    require $r_i=r_{i+1}$.
  \end{enumerate}
 Similarly, for $\bullet=\sigma$, an \emph{equivariant map} $\Sigma^{q_1}E_1\to \Sigma^{q_2}E_2$ will consist of the same information, but where the stable equivalences are required to be equivariant equivalences.
\end{defn}

\subsection{Coproducts}\label{subsec:prod} Finally, let us describe
the elementary coproduct operation on functors
$\two^n\to \burn_\bullet$, generalizing that from
\cite{lls2}. \cite{lls2} also define a product operation, but we have
no need for that.

Given two 2-functors $F_1, F_2\from \two^n \to \burn_\bullet$ the
\emph{coproduct} 2-functor $F_1\amalg F_2\from\two^n\to\burn_\bullet$
is defined as follows. On objects and 1-morphisms, $F_1\amalg F_2$ is
just the disjoint union: $(F_1 \amalg F_2) (v) = F_1(v) \amalg F_2(v)$
for $v\in\two^n$, with $\ZZ_2$-action defined component-wise if
$\bullet=\xi$, and $(F_1 \amalg F_2)(\phi_{u,v})=F_1(\phi_{u,v})\amalg
F_2(\phi_{u,v})$ for $u\geq v$, with the source map, target map, and sign map
if $\bullet=\sigma$, and $\ZZ_2$-action if $\bullet=\xi$, defined
component-wise. For $u\geq v\geq w$, the associated 2-morphism may be
viewed as a map
\[
  (F_1\amalg
  F_2)_{u,v,w}\from\coprod_{i=1,2}F_i(\phi_{v,w})\times_{F_i(v)}
  F_i(\phi_{u,v}) \longrightarrow \coprod_{i=1,2}
  F_i(\phi_{u,w})
\]
and it is defined component-wise using the bijections $(F_i)_{u,v,w}$
for $i=1,2$.  We have the following immediate property for chain
complexes:
\[
	\Tot(F_1 \amalg F_2)=\Tot(F_1)\oplus \Tot(F_2).
\]

\section{Realizations of Burnside functors}\label{sec:box}

In this section, given a functor $F\from \two^n\to\burn_\bullet$ to
any of the previously defined Burnside categories, we construct an
essentially well-defined finite CW spectrum $\Realize{F}$, which in
the $\ZZ_2$-equivariant case $\bullet=\xi$ is a $\ZZ_2$-equivariant
spectrum. As a first step, we construct finite CW complexes
$\CRealize{F}_{k}$ for sufficiently large $k$, so that increasing the
parameter $k$ corresponds to suspending the CW complex
$\CRealize{F}_{k}$. The finite CW spectrum $\Realize{F}$ is then
defined from this sequence of spaces. The construction of
$\CRealize{F}_{k}$ depends on some auxiliary choices, but its stable
homotopy type does not. Moreover, the spectra constructed from two
stably equivalent Burnside functors will be homotopy equivalent. 

For signed Burnside functors, i.e., when $\bullet=\sigma$, we can
carry out our construction with a $\ZZ_2$-action. For ordinary
Burnside functors, i.e., when $\bullet = \varnothing$, our
construction recovers the ``little boxes'' realization of
\cite[\S5]{lls1}, cf.~\cite[\S7]{lls2}; but if it comes from a signed
Burnside functor, we produce an alternate constructon with an
extra $\ZZ_2$-action.

After reviewing the notion of ``box maps'' used in the little box
realization construction of \cite[\S5]{lls2}, we introduce ``signed
box maps.'' After providing the necessary background on homotopy
colimits, we then use signed box maps to construct the realization
$\Realize{\cdot}$ for functors to the signed Burnside category. This
is all that is needed to construct the odd Khovanov homotopy type. We
then indicate the modifications needed to define the other realization
functors and to construct the various extra $\ZZ_2$-actions.

\subsection{Signed box maps}\label{sec:signed-box-maps}
 We start with the construction of
(ordinary) box maps, following \cite[\S2.10]{lls1}.  Fix an
identification $S^k=[0,1]^k/\partial$ which we maintain though the
sequel, and view $S^k$ as a pointed space.  Let $B$ be a box in
$\RR^k$ with edges parallel to the coordinate axes, that is,
$B=[a_1,b_1] \times \dots \times [a_k,b_k]$ for some $a_i,b_i$.  Then
there is a standard homeomorphism from $B$ to $[0,1]^k$, via
$(x_1,\dots,x_k ) \to (\frac{x_1-a_1}{b_1-a_1},\dots,
\frac{x_k-a_k}{b_k-a_k})$.  Then we have an identification
$S^k\cong B/\partial B$.

Given a collection of sub-boxes $B_1,\dots, B_\ell \subset B$ with
disjoint interiors, there is an induced map
\begin{equation}\label{eq:signedphi}
  S^k = B/\partial B \to B /(B \backslash (\mathring{B_1} \cup \dots \cup \mathring{B_\ell}))= \bigvee^\ell_{a=1} B_a /\partial B_a = \bigvee^\ell_{a=1}S^k \to S^k.
\end{equation}
The last map is the identity on each summand, so that the composition
has degree $\ell$.  As pointed out in \cite{lls1}, this construction is
continuous in the position of the sub-boxes.  We let $E(B,\ell)$
denote the space of boxes with disjoint interiors in $B$, and have a
continuous map $E(B,\ell) \to \Map(S^k,S^k)$.

We can generalize the above procedure to associate a map of spheres to
a map of sets $A \to Y$, as follows.  Say we have chosen sub-boxes
$B_a \subset B$ with disjoint interiors, for $a\in A$.  Then we have a
map:
\begin{equation}\label{eq:mapfromset}
  S^k = B /\partial B \to B /(B\backslash (\bigcup_{a\in A}\mathring{B_a} )) =\bigvee_{a\in A} B_a /\partial B_a = \bigvee_{a\in A}S^k \to \bigvee_{y\in Y} S^k
\end{equation}
where the last map is built using the map of sets $A \to Y$.

More generally, we can also assign a box map to a correspondence of
sets, as follows.  Fix a correspondence $A$ from $X$ to $Y$ with
source map $s$ and target map $t$.  Say that we also have a collection
of boxes $B_x$ for $x\in X$.  Finally, we also choose a collection of
sub-boxes $B_a \subset B_{s(a)}$ with disjoint interiors, for $a\in
A$.  We then have an induced map
\begin{equation}\label{eq:mapfromset2}
\bigvee_{x\in X} S^k \to \bigvee_{y\in Y} S^k,
\end{equation}
by applying, on $B_x$, the map associated to the set map $s^{-1}(x)
\to Y$.  A map as in Equation~\eqref{eq:mapfromset2} is said to \emph{refine}
the correspondence $A$.  Let $E(\{B_x\},s)$ be the space of
collections of labeled sub-boxes $\{B_a \subset B_{s(a)} \mid a \in
A\}$ with disjoint interiors.  Then, choosing a correspondence
$(A,s,t)$ (so that $A$ and $s$ are those appearing in the definition
of $E(\{B_x \},s)$---note that the definition of $E(\{B_x\},s)$ does
not involve the target map $t$)---Equation~\eqref{eq:mapfromset2} gives a map
$E(\{B_x\},s) \to \Map(\vee_{x\in X} S^k ,\vee_{y\in Y} S^k)$.  We
write
\begin{equation}
\Phi(e,A) \in\Map(\bigvee_{x\in X} S^k \to \bigvee_{y\in Y}S^k)\label{eq:phi}
\end{equation}
for the map associated to $e\in E(\{B_x\},s)$ and a compatible
correspondence $(A,s,t)$. The main point is that, for any box map
$\Phi(e,A)$ refining $A$, the induced map on the $k\th$ homology
agrees with the abelianization map
\[
\Abelianize(A)\from\Abelianize(X)=\widetilde{H}_k(\vee_{x\in
  X}S^k)\to\Abelianize(Y)=\widetilde{H}_k(\vee_{y\in Y}S^k).
\]

We now indicate a further generalization of box maps to cover signed
correspondences.  Fix a signed correspondence $(A,s,t,\sigma)$ from
$X$ to $Y$, and let $B_x,\; x \in X$ be some collection of boxes.  Fix
a collection of sub-boxes $B_{a} \subset B_{s(a)}$ for $a\in A$.
There is an induced map just as in Equation~\eqref{eq:mapfromset2},
but whose construction depends on the sign map $\sigma$, as
follows. For $x\in X$, we have a set map $s^{-1}(x) \to Y$, along with
signs for each element of $s^{-1}(x)$.  We modify the box map refining
$s^{-1}(x) \to Y$ (without sign) by precomposing with $\refl$,
reflection in the first coordinate, in boxes with non-trivial sign:
\[
S^k = B/\partial B \to B /(B\backslash (\bigcup_{a\in A}\mathring{B_a} )) = \bigvee_{a\in A} B_a /\partial B_a\xrightarrow{\bigvee \refl_{a}} \bigvee_{a\in A} B_a /\partial B_a = \bigvee_{a\in A}S^k  \to \bigvee_{y\in Y}S^k.
\]
Here $\refl_a=\refl$ if $\sigma(a)=-1$ and $\refl_a=\Id$ if
$\sigma(a)=+1$. This defines the map on the factor on the left of
Equation~\eqref{eq:mapfromset2} indexed by the element $x\in X$.  We
say that a map constructed this way \emph{refines} the signed
correspondence $(A,s,t,\sigma)$.  As before, we can regard it as a map
\[
\Phi(e,A)\in \Map( \bigvee_{x\in X} S^k ,\bigvee_{y\in Y}S^k),
\]
where $e \in E(\{B_x\},s)$, and $(A,s,t,\sigma)$ is a compatible
signed correspondence.  Once again, the induced map on the $k\th$
homology agrees with the abelianization map.

Similarly, for a signed correspondence $(A,s,t,\sigma)$, we can
consider box maps refining the (unsigned) correspondence, and then
precompose by $\dig$, the reflection in the first two coordinates, if
the sign is nontrivial.  We call a map obtained this way a
\emph{doubly signed} refinement of the tuple $(A,s,t,\sigma)$, and
denote it $\widehat{\Phi}(e,A)$.

We next record two lemmas from \cite{lls1} about the spaces
$E(\{B_x\},s)$ that we need later. For a map $s\from A \to X$ define
$E_\sym(\{B_x\},s)$ to be the subspace of $E(\{ B_x\},s)$ in which we
require the box $B_a$ to lie symmetrically in $B_{s(a)}$ with respect
to reflection in the first coordinate $\refl$, for all $a\in A$, and $E_{\dsym}(\{B_x\},s)$ by requiring symmetry in the first two coordinates.

\begin{lem}\label{lem:connctd}
  Consider $s\from A\to X$. If the boxes are $k$-dimensional then
  $E(\{ B_x\},s)$ is $(k-2)$-connected and $E_\sym(\{B_x\},s)$ is
  $(k-3)$-connected and $E_{\dsym}(\{B_x\},s)$ is $(k-4)$-connected.
\end{lem}

\begin{proof}
  The first statement is \cite[Lemma 2.29]{lls1}, and the second
  statement follows from the first since the space of symmetric
  $k$-dimensional boxes, $E^k_\sym(\{B_x\},s)$, is homotopy equivalent
  to the space of $(k-1)$-dimensional boxes, $E^{k-1}(\{B_x\},s)$, from which the third statement also follows.
\end{proof}

\begin{lem}\label{lem:box-composition}
  If $e\in E(\{B_x\},s_A)$ is compatible with a signed correspondence
  $A$ from $X$ to $Y$, and $f\in E(\{B_y\},s_B)$ is compatible with a
  signed correspondence $B$ from $Y$ to $Z$, then there is a unique
  $f\circ_{\sigma_A} e\in E(\{B_x\},s_{B\circ A})$ compatible with
  $B\circ A$, depending only on $e$, $f$, and the sign map $\sigma_A$,
  so that $\Phi(f\circ e,B\circ A)=\Phi(f,B)\circ\Phi(e,A)$; we will
  sometimes drop the subscript $\sigma_A$ (as we did just
  now). Similarly, there is a unique $f\,\widehat{\circ}_{\sigma_A}\,
  e\in E(\{B_x\},s_{B\circ A})$ compatible with $B\circ A$, so that
  $\widehat{\Phi}(f\,\widehat{\circ}\,e,B\circ
  A)=\widehat{\Phi}(f,B)\circ\widehat{\Phi}(e,A)$. Both of these
  assignments $\circ,\widehat{\circ}\from E(\{B_y\},s_B)\times
  E(\{B_x\},s_A)\to E(\{B_x\},s_{B\circ A})$ are continuous and send
  $E_\sym(\{B_y\},s_B)\times E_\sym(\{B_x\},s_A)$ to
  $E_\sym(\{B_x\},s_{B\circ A})$ (and similarly for
  $E_{\dsym}$). Moreover, $\circ$ agrees with $\widehat{\circ}$ when
  restricted to $E_\dsym(\{B_y\},s_B)\times E_\dsym(\{B_x\},s_A)$.
\end{lem}

\begin{proof}
  This `composition map' $E(\{B_y\},s_B)\times E(\{B_x\},s_A)\to
  E(\{B_x\},s_{B\circ A})$, as discussed in \cite[\S2.10]{lls1}, is
  constructed in the unsigned as follows: Given sub-boxes $B_b\subset
  B_{s_B(b)}$ for all $b\in B$ and sub-boxes $B_a\subset B_{s_A(a)}$
  for all $a\in A$, define, for all $(b,a)\in B\times_Y A$, the
  sub-box $B_{(b,a)}\subset B_{s_A(a)}$ to be the `sub-box of the
  sub-box', $B_b\subset B_{s_B(b)=t_A(a)}\cong B_a\subset B_{s_A(a)}$.

  In the signed case, one needs to precompose with a reflection if
  $\sigma_A(a)=-1$. That is, define $B_{(b,a)}$ to be sub-box
  $\refl_a(B_b)\subset B_{s_B(b)}\cong B_a$ of the sub-box $B_a\subset
  B_{s_A(a)}$, where as before, $\refl_a=\refl$ if $\sigma_A(a)=-1$
  and $\refl_a=\Id$ if $\sigma_A(a)=+1$.

  It is clear that $\Phi(f\circ e,B\circ
  A)=\Phi(f,B)\circ\Phi(e,A)$. The slight subtlety lies in the case
  when $(b,a)\in B\circ A$ has sign $+1$, but each of $a$ and $b$ has
  sign $-1$; then each of the box maps $\Phi(e,A)$ and $\Phi(f,B)$
  requires reflecting along the first coordinate, and so their
  composition does not. The case for $\widehat{\Phi}$ is similar.

  It is clear that both $\circ$ and $\widehat{\circ}$ send
  $E_\sym(\{B_y\},s_B)\times E_\sym(\{B_x\},s_A)$ to
  $E_\sym(\{B_x\},s_{B\circ A})$ and $E_\dsym(\{B_y\},s_B)\times
  E_\dsym(\{B_x\},s_A)$ to $E_\dsym(\{B_x\},s_{B\circ A})$. In both
  cases, the definition of $\circ$ does not require any reflections,
  while in the latter case, the definition of $\widehat{\circ}$ does
  not. So in the latter case, $\circ$ and $\widehat{\circ}$ agree.
\end{proof}

\subsection{Homotopy coherence}\label{subsec:homtpy}
In this section, we briefly review homotopy colimits and homotopy
coherent diagrams following \cite[\S2.9]{lls1}. Let $\topp$ be the
category of well-based topological spaces; we will usually work with
finite CW complexes. A \emph{weak equivalence} $X\to Y$ is a map that
induces isomorphisms on all homotopy groups; typically our spaces are
all simply connected, when the definition reduces to being
isomorphisms on all homology groups. A homotopy equivalence is a
special case of weak equivalence, and for CW complexes (the case at
hand), the two notions are equivalent.

We will sometimes also work with spaces equipped with an action by a
fixed finite group $G$ (which is $\ZZ_2$ or $\ZZ_2\times\ZZ_2$ in our
case), and all maps are $G$-equivariant, forming a category
$G\text{-}\topp$. We also require that the inclusions of fixed points
$X^{H}\to X^{H'}$, for all subgroups $H'<H$ of $G$, are cofibrations;
in our case, all the spaces will carry CW structures so that the
actions are CW actions---that is, each group element simply permutes
the cells and respects the attaching maps. A map $X\to Y$ is called a
\emph{weak equivalence} if the induced map $X^H\to Y^H$ is a weak
equivalence for all subgroups $H$ of $G$. A homotopy equivalence in
$G\text{-}\topp$ induces a weak equivalence. For $G$-CW complexes (the
case at hand), the two notions are equivalent by the $G$-Whitehead
theorem, see \cite[Theorem 2.4]{Greenlees-May}. For $G$-CW complexes,
a weak equivalence $X\to Y$ induces a weak equivalence between
quotients of fixed points, $X^{H'}/X^H\to Y^{H'}/Y^H$, for all
subgroups $H'<H$ of $G$, and between orbit spaces, $X/H\to Y/H$, for
all subgroups $H$ of $G$.

Now we recall the notion of a homotopy coherent diagram, which is the
data from which a homotopy colimit is constructed. Fix a finite group
$G$ (typically $0$ or $\ZZ_2$ or $\ZZ_2\times\ZZ_2$). A homotopy
coherent diagram is intuitively a diagram $F\from \Cat\to
G\text{-}\topp$ which is not commutative, but commutative up to
homotopy, and the homotopies themselves commute up to higher homotopy,
and so on, and for which all the homotopies and higher homotopies are
viewed as part of the data of the diagram. Precisely, we have the
following.

\begin{defn}[{\cite[Definition 2.3]{vogt}}]\label{def:homco}
  A homotopy coherent diagram $F\from\Cat\to G\text{-}\topp$ assigns
  to each $x \in \Cat$ a space $F(x) \in G\text{-}\topp$, and for each $n \geq
  1$ and each sequence
    \[
    x_0 \xrightarrow{f_1} x_1 \xrightarrow{f_2} \cdots \xrightarrow{f_n} x_n
    \]
    of composable morphisms in $\Cat$ a continuous $G$-map
    \[
    F(f_n,\dots,f_1)\from [0,1]^{n-1} \times F(x_0) \to F(x_n)\phantom{\xrightarrow{f}}
    \]
    with $F(f_n,\dots,f_1)([0,1]^{n-1}\times \{ *\})=*$.  These maps
    are required to satisfy the following compatibility conditions:
\begin{align}
\nonumber F(f_n,\dots,f_1)(t_1,& \dots,t_{n-1})=  \\ 
& \begin{cases}
F(f_n,\dots,f_2)(t_2,\dots,t_{n-1}), &f_1 = \Id\\
F(f_n ,\dots,\hat{f}_i,\dots,f_1)(t_1,\dots,t_{i-1}\cdot t_i,\dots,t_{n-1}),&f_i=\Id, 1 < i < n\\
F(f_{n-1},\dots,f_1)(t_1,\dots,t_{n-2}),&f_n = \Id\\
F(f_n,\dots,f_{i+1})(t_{i+1},\dots, t_{n-1}) \circ F(f_i,\dots,f_1)(t_1,\dots,t_{i-1}), & t_i=0 \\
F(f_n,\dots,f_{i+1}\circ f_i,\dots,f_1)(t_1,\dots,\hat{t}_{i},\dots,t_{n-1}), & t_i=1.
\end{cases}\label{eq:compat}
\end{align}
When $\Cat$ does not contain any non-identity isomorphisms, homotopy
coherent diagrams may be defined only in terms of non-identity
morphisms and the last two compatibility conditions.
\end{defn}

Given a homotopy coherent diagram, we can define its \emph{homotopy
  colimit} in $G\text{-}\topp$, quite concretely, as follows:
\begin{defn}[{\cite[\S5.10]{vogt}}]\label{def:hoco}
  Given a homotopy coherent diagram $F\from \Cat \to G\text{-}\topp$ the
  \emph{homotopy colimit} of $F$ is defined by
\begin{equation}\label{eq:hoco}
\hoco \; F = \{ * \} \amalg \coprod_{n\geq 0} \coprod_{x_0\xrightarrow{f_1} \cdots\xrightarrow{f_n}x_n} [0,1]^n \times F(x_0) /\sim,
\end{equation}
 where the equivalence relation $\sim$ is given as follows:
\[
(f_n,\dots,f_1;t_1,\dots,t_n;p)\sim\begin{cases}
(f_n,\dots,f_2 ; t_2,\dots,t_n;p),&f_1=\Id\\
(f_n,\dots,\hat{f}_i,\dots,f_1;t_1,\dots,t_{i-1}\cdot t_i,\dots,t_n;p),&f_i=\Id,i>1\\
(f_{n},\dots,f_{i+1};t_{i+1},\dots,t_n;F(f_i,\dots,f_1)(t_1,\dots,t_{i-1},p)), &t_i=0\\
(f_n,\dots,f_{i+1}\circ f_i,\dots,f_1;t_1,\dots,\hat{t}_{i},\dots,t_n;p), &t_i=1,i<n\\
(f_{n-1},\dots,f_1; t_1,\dots,t_{n-1};p), & t_n=1\\
*, & p=*.
\end{cases}
\]
When $\Cat$ does not contain any non-identity isomorphisms, homotopy
colimits may be defined only in terms of non-identity morphisms and the
last four equivalence relations.
\end{defn}

We will need the following properties, most of which are immediate
consequences of the above formulas:
\begin{enumerate}[leftmargin=*,label=(ho-\arabic*)]
\item \label{itm:ho1} Suppose that $F_0,F_1\from \Cat \to
  G\text{-}\topp$ are homotopy coherent diagrams and $\eta\from F_1
  \to F_0$ is a natural transformation, that is, a homotopy coherent
  diagram
  \[\eta\from\two\times\Cat\to\topp
  \]
  with $\eta|_{\{i\}\times\Cat}=F_i$, $i=0,1$.  Then $\eta$ induces a
  map $\hoco\,\eta \from \hoco\,F_1 \to \hoco\,F_0$.  If $\eta(x)$
  is a weak equivalence for each $x\in \Cat$---we will call such an
  $\eta$ a weak equivalence $F_1\to F_0$---then $\hoco\,\eta$ is a
  weak equivalence as well.

  When the spaces involved are $G$-CW complexes (the case at hand), a
  weak equivalence $\eta\from F_1\to F_0$ is also a homotopy
  equivalence \cite[Proposition 4.6]{vogt}, that is, there exists
  $\zeta,\zeta'\from F_0\to F_1$ and
  \[
  \mathfrak{h},\mathfrak{h}'\from\{2\to1\to0\}\times\Cat\to
  G\text{-}\topp,
  \]
  with $\mathfrak{h}|_{\{2\to1\}\times\Cat}=\eta$,
  $\mathfrak{h}|_{\{1\to0\}\times\Cat}=\zeta$,
  $\mathfrak{h}|_{\{2\to0\}\times\Cat}=\Id_{F_0}$, and
  $\mathfrak{h}'|_{\{2\to1\}\times\Cat}=\zeta'$,
  $\mathfrak{h}'|_{\{1\to0\}\times\Cat}=\eta$,
  $\mathfrak{h}'|_{\{2\to0\}\times\Cat}=\Id_{F_1}$.
\item \label{itm:ho2} Suppose that $F_0,F_1\from \Cat \to
  G\text{-}\topp$ are diagrams and that $F_0 \vee F_1 \from \Cat \to
  G\text{-}\topp$ is the diagram obtained by wedge sum: $(F_0 \vee
  F_1)(x)=F_0(x) \vee F_1(x)$ for all $x\in\Cat$, and
  \[(F_0\vee
  F_1)(f_n,\dots,f_1)(t_1,\dots,t_{n-1},p)=F_i(f_n,\dots,f_1)(t_1,\dots,t_{n-1},p)\]
  for all $i=0,1$, $x_0 \xrightarrow{f_1} x_1 \xrightarrow{f_2} \cdots
  \xrightarrow{f_n} x_n$, and $p\in F_i(x_0)$.  Then $\hoco (F_0\vee
  F_1)$ and $(\hoco F_0) \vee (\hoco F_1)$ are naturally homeomorphic.

\item \label{itm:ho4} For any normal subgroup $H$ of $G$, define the
  fixed point diagram $F^H\from\Cat\to G/H\text{-}\topp$ by setting
  $F^H(x)$ to be the fixed points ${F(x)}^H$. Define the quotient
  diagram $F/F^H\from\Cat\to G\text{-}\topp$ by setting $(F/F^H)(x)$
  to be the quotient $F(x)/\{p\sim *\text{ for all }p\in
  F^H(x)\}$. Then there are natural homeomorphisms
  \[
  \begin{tikzpicture}[xscale=4,yscale=1.5]
  \node (hofix) at (0,0) {$\hocolim(F^H)$};
  \node (ho1) at (1,0)  {$\hocolim F$};
  \node (hoquot) at (2,0) {$\hocolim(F/F^H)$};

  \node (fixho) at (0,1) {$\hocolim(F)^H$};
  \node (ho2) at (1,1) {$\hocolim F$};
  \node (quotho) at (2,1) {$\hocolim(F)/\hocolim(F)^H$};

  \draw[->] (hofix) -- (ho1);
  \draw[->] (ho1) -- (hoquot);
  \draw[->] (fixho) -- (ho2);
  \draw[->] (ho2) -- (quotho);

  \draw[<-] (hofix) -- (fixho) node[midway,anchor=west] {$\cong$};
  \draw[<-] (hoquot) -- (quotho) node[midway,anchor=west] {$\cong$};
  \draw[<-] (ho1) -- (ho2) node[midway,anchor=west] {$=$};
  \end{tikzpicture}
  \]
  with the arrows on the bottom row being induced from the natural
  transformations $F^H\to F \to F/F^H$.

  For diagrams of $G$-CW complexes, the arrows on the top row form a
  cofibration sequence since $(\hocolim F,\hocolim(F)^H)$ form a
  CW-pair; moreover, a weak equivalence $F_1\to F_0$ induces weak
  equivalences $F^H_1\to F^H_0$ and $F_1/F^H_1\to F_0/F^H_0$ by the
  $G$-Whitehead theorem.
\item \label{itm:ho5} For any normal subgroup $H$ of $G$, let
  $F/H\from\Cat\to G/H\text{-}\topp$ be the orbit diagram, obtained by setting
  $(F/H)(x)$ to be the orbit $F(x)/\{p\sim h(p)\text{ for all }p\in
  F(x),h\in H\}$. Then $\hocolim(F)/{H}$ and $\hocolim(F/{H})$ are
  naturally homeomorphic, and the map $\hocolim F\to\hocolim(F/{H})$
  induced from the natural transformation $F\to F/{H}$ is identified
  with the quotient map $\hocolim F\to\hocolim(F)/{H}$.

  For diagrams of $G$-CW complexes, a weak equivalence $F_1\to F_0$
  induces a weak equivalence $F_1/H\to F_0/H$ by the $G$-Whitehead
  theorem.
\end{enumerate}

\subsection{Little boxes refinement}
With this background, we are ready to review the little box
realization construction of \cite[\S5]{lls1} and generalize for the
other kinds of Burnside functors introduced here, i.e., functors to
$\burn_\bullet$ for $\bullet\in\{\emptyset,\sigma,\xi\}$.

\begin{defn}\label{def:spacref}
  Fix a small category $\Cat$ and a strictly unitary $2$-functor $F
  \from \Cat \to \burn_\bullet$.  A $k$-dimensional \emph{spatial
    refinement} of $F$ is a homotopy coherent diagram $\widetilde{F}_k
  \from \Cat \to \topp$ such that
\begin{enumerate}[leftmargin=*]
\item For any $u\in \Cat$, $\widetilde{F}_k(u)=\vee_{x\in F(u)} S^k$,
  where $S^k=[0,1]^k/\bdy$.  When $\bullet=\xi$,
  this space has an additional $\ZZ_2$-action---denoted
  $\xi$---induced from the $\ZZ_2$-action on the set $F(u)$.
\item For any sequence of morphisms $u_0 \xrightarrow{f_1} \cdots
  \xrightarrow{f_n} u_n$ in $\Cat$ and any $(t_1,\dots,t_{n-1})\in
  [0,1]^{n-1}$ the map
  \[
  \widetilde{F}_k(f_n,\dots,f_1) (t_1,\dots,t_{n-1})\from \bigvee_{x\in F(u_0)} S^k \to \bigvee_{x\in F(u_n)}S^k
  \]
  is a box map refining the (possibly signed) correspondence $F(f_n
  \circ \dots \circ f_1)$, which is naturally isomorphic to $F(f_n)
  \times_{F(u_{n-1})} \dots \times_{F(u_1)} F(f_1)$; when
  $\bullet=\xi$, we require each $\widetilde{F}_k$ to be
  $\xi$-equivariant. 
\end{enumerate}
Note that for $\bullet=\xi$, the $\xi$-action is a CW action, with the
CW complex structure given by the boxes. 
\end{defn}

This definition reduces to \cite[Definition 5.1]{lls1} when
$\bullet=\varnothing$. The additions here are the signed box map
refinements introduced in \S\ref{sec:signed-box-maps}, which are needed
for functors to $\oddb$ with non-trivial signs in the correspondences,
and the equivariant conditions.

The following is the main technical result in this section. As usual,
$\bullet\in \{\varnothing, \sigma,\xi\}$ in the statement. For
$\bullet=\varnothing$, the ordinary Burnside category, this reduces to
\cite[Proposition 5.2]{lls1}.

\begin{prop}\label{prop:cube}
  Let $\Cat$ be a small category in which every sequence of composable
  non-identity morphisms has length at most $n$, and let $F\from
  \Cat\to\burn_\bullet$ be a strictly unitary 2-functor.
\begin{enumerate}[leftmargin=*]
\item If $k \geq n$, there is a $k$-dimensional spatial refinement of
  $F$ (which is $\xi$-equivariant if $\bullet=\xi$).   \label{itm:real1}
\item If $k \geq n+1$, then any two $k$-dimensional spatial
  refinements of $F$ are weakly equivalent ($\xi$-equivariantly if
  $\bullet=\xi$).  \label{itm:real2}
\item If $\widetilde{F}_k$ is a $k$-dimensional spatial refinement of $F$,
  then the result of suspending each $\widetilde{F}_k(u)$ and
  $\widetilde{F}_k(f_n, \dots, f_1)$ gives a $(k+1)$-dimensional spatial
  refinement of $F$.  \label{itm:real3}
\end{enumerate}
\end{prop}

\begin{proof}
  The proof is parallel to that of \cite[Proposition 5.2]{lls1}.  For
  all values of $\bullet$, the third point is clear.

  The case $\bullet=\sigma$ is an immediate generalization of the case
  $\bullet=\emptyset$ as worked out in \cite{lls1}. We sketch the main
  details. For point (\ref{itm:real1}), given $u\in \Cat$, set
  $\widetilde{F}_k(u) =\vee_{x\in F(u)} S^k$, where the $x$-summand
  is $B_x/\partial$.  We choose for each $f\from u \to v$ in $\Cat$ a
  signed box map $\widetilde{F}_k(f)$ refining the signed
  correspondence $F(f)$; for the identity morphism $\Id_u$, choose
  $\widetilde{F}_k(\Id_u)$ to be the identity map
  $\Id_{\widetilde{F}_k(u)}$, which indeed is a box map refining the
  identity correspondence $\Id_{F(u)}=F(\Id_u)$. Let $e_f \in E(\{B_x
  \mid x \in F(u)\},s_{F(f)})$ be the collection of little boxes
  corresponding to $F(f)$.

  This gives a definition of $\widetilde{F}_k$ on vertices and arrows.
  We now need to define the appropriate coherences among these maps,
  which will be done inductively. Assume for all sequences $v_0
  \xrightarrow{f_1} \cdots \xrightarrow{f_m} v_m$ with $m\leq\ell$, we
  have defined continuous maps
  \[
  e_{f_m ,\dots, f_1}\from  [0,1]^{m-1}\to E(\{B_x \mid x \in F(v_0)), s_{F(f_m\circ \dots \circ f_1)}),
  \]
  and that these maps satisfy Equation~\eqref{eq:compat} (with the
  composition map from Lemma \ref{lem:box-composition} playing the
  role $\circ$). Then for the induction step, given $v_0
  \xrightarrow{f_1} \cdots \xrightarrow{f_{\ell+1}} v_{\ell+1}$, we
  have a continuous map
  \[
  S^{\ell-1}=\partial ([0,1]^\ell) \to E(\{B_x \mid x \in
  F(v_0)\},s_{F(f_{\ell+1}\circ\dots\circ f_1)})
  \]
  defined by
  \begin{equation}\label{eq:homconds3}
    \begin{split}
      (t_1,\dots,t_{i-1},0,t_{i+1},\dots,t_\ell) &\mapsto  e_{f_{\ell+1},\dots,f_{i+1}}(t_{i+1},\dots,t_{\ell})\circ e_{f_i,\dots,f_1}(t_1,\dots,t_{i-1}) \\
      (t_1,\dots,t_{i-1},1,t_{i+1},\dots,t_\ell) &\mapsto
      e_{f_{\ell+1},\dots,f_{i+1}\circ f_i,\dots,
        f_1}(t_1,\dots,t_{i-1},t_{i+1},\dots, t_{\ell}).
    \end{split}
  \end{equation}
  By Lemma \ref{lem:connctd}, this map extends to a map, call it
  $e_{f_{\ell+1},\dots,f_1}$, from $[0,1]^\ell$. By definition, the
  maps
  \[
  \Phi(e_{f_m,\dots,f_1})\from [0,1]^{m-1} \times \bigvee_{x\in F(v_0)} S^k \to \bigvee_{x\in F(v_{m})}S^k
  \]
  assemble to form a homotopy coherent diagram. 

  Next, we address point (\ref{itm:real2}).  Say that we have spatial
  refinements $\widetilde{F}^0_k$ and $\widetilde{F}^1_k$ of $F$.
  Consider the functor $G\from\two\times\Cat\to\burn_\bullet$ defined
  as the composition
  $\two\times\Cat\xrightarrow{\pi_2}\Cat\xrightarrow{F}\burn_\bullet$. We
  need only define a spatial refinement $\widetilde{G}_k$ of $G$ that
  restricts to $\widetilde{F}^i_k$ at $\{i\} \times \Cat$ for
  $i=0,1$. The construction of $\widetilde{G}_k$ proceeds uneventfully
  by induction as before. Then for each $u\in\Cat$,
  $G(\phi_{1,0}\times \Id_u)$ refines the identity correspondence
  $\Id_{F(u)}$ (and indeed, may be chosen to be the identity map), and
  hence is a weak equivalence; therefore, by \ref{itm:ho1},
  $\widetilde{G}_k\from \widetilde{F}^1_k\to\widetilde{F}^0_k$ is a
  weak equivalence as well.

  Next, we consider the case $\bullet=\xi$. Construct a functor
  $G\from\Cat\to\oddb$ so that $\oddeq\circ G$ is naturally isomorphic
  to $F$ by choosing a section of the $\ZZ/2$-quotient map
  $\amalg_{u\in\Cat}F(u)\to \amalg_{u\in\Cat}F(u)/\ZZ_2$; we leave the
  details to the reader. Then construct a box
  map refinement of $G$; let $B_x$ be the box assigned to $x$, for
  $x\in\amalg_{u\in\Cat}G(u)$, and let $e_{f_m ,\dots, f_1}\from
  [0,1]^{m-1}\to E(\{B_x \mid x \in G(v_0), s_{G(f_m\circ \dots \circ
    f_1)})$ denote the $[0,1]^{m-1}$-parameter family chosen for the
  sequence $v_0 \xrightarrow{f_1} \cdots \xrightarrow{f_m} v_m$ during
  the construction.

  We may then define $B_{\{1\}\times x}=\{1\}\times B_x$ and
  $B_{\{\xi\}\times x}=\{\xi\}\times B_x$, for
  $x\in\amalg_{u\in\Cat}G(u)$.  For any $v_0 \xrightarrow{f_1} \cdots
  \xrightarrow{f_m} v_m$ in $\Cat$, and any
  $(t_1,\dots,t_{m-1})\in[0,1]^{m-1}$, the configuration of disjoint
  boxes $e_{f_m ,\dots, f_1}(t_1,\dots,t_{m-1})\in E(\{B_x \mid x \in
  G(v_0), s_{G(f_m\circ \dots \circ f_1)})$ refining the signed
  correspondence $G(f_m\circ \dots \circ f_1)$ can be doubled to get a
  configuration of disjoint boxes
  \[
    d_{f_m,\dots,f_1}(t_1,\dots,t_{m-1})\in E(\{B_y \mid y \in
    F(v_0)=\{1,\xi\}\times G(v_0), s_{F(f_m\circ \dots \circ f_1)}=\Id_{\{1,\xi\}}\times s_{G(f_m,\circ\dots\circ f_1)})
  \]
  refining the unsigned correspondence $F(f_m\circ \dots
  \circ f_1)=\{1,\xi\}\times G(f_m\circ \dots \circ f_1)$. Use these
  $d_{f_m,\dots,f_1}$'s to construct the refinement, which is
  automatically $\xi$-equivariant. 
\end{proof}

\subsection{Realization of cube-shaped diagrams}\label{sec:cube-shaped-realize}
Finally in this section we will discuss how to construct a CW complex
$\CRealize{F}_{k}$, and then a CW spectrum
$\Realize{F}$, from a given diagram $F\from
\two^n\to\burn_\bullet$. Let $\two_+$ be the category with objects
$\{0,1,*\}$ and unique non-identity morphisms $1\to0$ and $1\to *$, and
let $\two^n_+=(\two_+)^n$.

Let $\widetilde{F}_k\from \two^n\to \topp$ be a spatial refinement of $F$ using $k$-dimensional
boxes, and let $\widetilde{F}^+_k\from \two^n_+ \to
\topp$ be the diagram obtained from $\widetilde{F}_k$ by setting
$\widetilde{F}_k^+(x)$ to be a point for all
$x\in\two^n_+\setminus\two^n$.  Let $\CRealize{F}_{k}$ be the
homotopy colimit of $\widetilde{F}_k^+$.  We call
$\CRealize{F}_{k}$ a \emph{realization} of $F\from \two^n
\to\burn_\bullet$ for $\bullet=\{\varnothing,\sigma,\xi\}$.

\begin{cor}\label{cor:burnlimwelldef}
  If $k\geq n+1$, then $\CRealize{F}_{k}$ is well-defined up
  to weak equivalence in $\topp$ (or $\ZZ_2\text{-}\topp$ if
  $\bullet=\xi$). In each case,
  $\CRealize{F}_{k+1}=\Sigma\CRealize{F}_{k}$.
\end{cor}
\begin{proof}
  As in \cite[Corollary 5.6]{lls1}, this follows from Proposition
  \ref{prop:cube} and properties of homotopy colimits
  \ref{itm:ho1}.
\end{proof}

The homotopy colimit $\CRealize{F}_{k}$ may be given several CW
structures.  First, from Definition \ref{def:hoco}, there is the
\emph{standard} CW structure, with cells $[0,1]^m \times B_x$,
parameterized by tuples $(f_m,\dots,f_1)$ subject to some relations.

We have a second CW structure on $\CRealize{F}_{k}$, the
\emph{fine} structure, which is obtained from the standard structure
by subdividing each cell $[0,1]^m\times B_x$ along the central
$(m+k-1)$-dimensional box $[0,1]^m\times B^\refl_{x}$, where
$B^\refl_x\subset B_x$ is the fixed-point set of the reflection
$\refl\from B_x\to B_x$ along the first coordinate. The fine CW
complex structure will be of relevance in \S\ref{subsec:equiv-constructions}.

There is also the \emph{coarse} cell structure of \cite[Section
6]{lls1}.  There they construct a CW structure on $\CRealize{F}_{k}$
for $F$ an unsigned Burnside functor, with cells formed from unions of
standard cells, so that there is exactly one (non-basepoint) cell
$\cell(x)$ for each $x\in\amalg_u F(u)$. In more detail, if $F_x$
denotes the Burnside sub-functor of $F$ generated by $x$, then the
subcomplex $\CRealize{F_x}_k$ of $\CRealize{F}_k$ is the image of
the cell $\cell(x)$.  The construction generalizes without changes to
give a CW structure on $\CRealize{F}_{k}$ for
$F\from\two^n\to\burn_\bullet$, with the same set of cells; when
$\bullet=\xi$, this produces a $\ZZ_2$-CW complex. Unless otherwise
specified, this is the default CW complex structure that we consider
on $\CRealize{F}_k$.

\begin{lem}\label{lem:cofib}
  A cofibration sequence $G\to F\to H$ of functors $\two^n\to
  \burn_\bullet$ (cf.~Definition~\ref{def:burn-cofib-sequence}), upon
  realization, induces a cofibration sequence of spaces. In general,
  any natural transformation $\eta\from F_1\to F_0$ of Burnside
  functors $\two^n\to\burn_\bullet$ induces a map on the realizations.
\end{lem}
\begin{proof}
  Consider the standard CW complex structures.  For the first
  statement, a spatial refinement $\widetilde{F}_k$ of $F$ produces
  spatial refinements $\widetilde{G}_k$ of $G$ and $\widetilde{H}_k$
  of $H$; working with those refinements, it is an immediate
  consequence of the definitions that $\CRealize{G}_k$ is a CW
  subcomplex of $\CRealize{F}_k$ with quotient complex
  $\CRealize{H}_k$.  

  For the second part, if $\eta\from\two^{n+1}\to\burn_\bullet$ is the
  natural transformation, then $(F_0)_{\iota_0}$ is a subfunctor and
  $(F_1)_{\iota_1}$ is the corresponding quotient functor, where
  $\iota_i\from\two^n\to\two^{n+1}$ is the face inclusion to
  $\{i\}\times\two^n$. Therefore, we get a cofibration sequence
  \[
    \CRealize{(F_0)_{\iota_0}}_k\to\CRealize{\eta}_k\to\CRealize{(F_1)_{\iota_1}}_k.
  \]
  However, $\CRealize{(F_0)_{\iota_0}}_k=\CRealize{F_0}_k$, while
  $\CRealize{(F_1)_{\iota_1}}_k=\Sigma\CRealize{F_1}_k$ since
  $\CRealize{F_1}$ is constructed as a hocolim over $\two_+^n$, while
  $\CRealize{(F_1)_{\iota_1}}$ is constructed as a hocolim over
  $\two_+^{n+1}$. Therefore, the Puppe map
  \[
    \CRealize{(F_1)_{\iota_1}}_k=\Sigma\CRealize{F_1}_k=\CRealize{F_1}_{k+1}\to\CRealize{(F_0)_{\iota_0}}_k=\Sigma\CRealize{F_0}_k=\CRealize{F_0}_{k+1}
  \]
  is the required map.
\end{proof}

\begin{prop}\label{prop:totalization}
  If $F\from\two^n\to\burn_\bullet$, then its shifted reduced cellular
  complex $\redcellC(\CRealize{F}_{k})[-k]$ is isomorphic to the
  totalization $\Tot(F)$ with the cells mapping to the corresponding
  generators. If $\eta\from F_1\to F_0$ is a natural transformation of
  Burnside functors, then the map
  $\CRealize{F_1}_{k}\to\CRealize{F_0}_{k}$ is cellular, and the
  induced cellular chain map agrees with $\Tot(\eta)$.
\end{prop}

\begin{proof}
  The first statement is an immediate generalization of the
  corresponding statement for unsigned Burnside functors from
  \cite[Theorem 6]{lls1}. The second statement is also clear from the
  form of the map constructed in Lemma~\ref{lem:cofib}, using similar
  arguments.
\end{proof}

We can then package all these spaces together to construct a
\emph{finite CW spectrum}, by which we mean a pair $(X,r)$ (sometimes
written $\Sigma^r X$), where $X$ is a finite CW complex and $r\in\ZZ$;
one may view it as an object in the Spanier-Whitehead category, or as
$\Sigma^r (\Sigma^\infty X)$, the $r\th$ suspension of the suspension
spectrum of the finite CW complex $X$. One can take the reduced
cellular chain complex of a finite CW spectrum, whose chain homotopy
type is an invariant of the (stable) homotopy type of the CW
spectrum. Then, for a stable Burnside functor
$(F\from\two^n\to\burn_\bullet,r)$, after fixing a $k$-dimensional
spatial refinement $\widetilde{F}_k$, we may define its
\emph{realization} as the finite CW spectrum $\Realize{\Sigma^r
  F}=(\CRealize{F}_k,r-k)$.  

\begin{lem}\label{lem:functor-map-gives-space-map}
  Let $\Sigma^{r_1}F_1\to \Sigma^{r_2}F_2$ be a map of Burnside
  functors $F_1,F_2$.  Then there is an induced map of realizations
  \[
  \Realize{\Sigma^{r_1}F_1}\to \Realize{\Sigma^{r_2}F_2}
  \]
  well-defined up to homotopy equivalence.  If the map of Burnside
  functors is a stable equivalence, then the induced map is a homotopy
  equivalence. If $\bullet=\xi$, the induced map and the homotopy
  equivalence may be taken $\ZZ_2$-equivariant.
\end{lem}
\begin{proof}
  First, associated to a natural transformation of Burnside functors,
  there is a well-defined space map by Lemma \ref{lem:cofib}.  Since
  all maps of Burnside functors are obtained as compositions of
  natural transformations and stable equivalences, we need only show
  that there is a well-defined (up to homotopy) homotopy equivalence
  of the realizations associated to a stable equivalence.

  For this, we must check that maps as in the items of Definition
  \ref{def:stableq} preserve the stable homotopy type of
  $\CRealize{F}_{k}$.  For a natural transformation $\eta\from F_i \to
  F_{i+1}$, Proposition~\ref{prop:totalization} produces a cellular
  map $\CRealize{F_i}_k \to \CRealize{F_{i+1}}_k$, which induces the
  map $\Tot(\eta)$ on the cellular chain complex.  The condition that
  $\Tot(\eta)$ is a chain homotopy equivalence implies that the map of
  spaces is a homotopy equivalence by Whitehead's theorem (we assume
  $k$ is sufficiently large so that all relevant spaces are simply
  connected), so has a homotopy inverse well-defined up to
  homotopy. When $\bullet=\xi$, the $\ZZ_2$-action is free (away from
  the basepoint), so by the $G$-Whitehead theorem, the homotopy
  equivalences may be taken equivariant.
\end{proof}

\subsection{Equivariant constructions}\label{subsec:equiv-constructions}
We now explain how to make the constructions of the previous
sections equivariant.  Recall that each $S^k=[0,1]^k/\partial$ carries a
natural $\ZZ_2$-action by $\refl$---reflection in the first
coordinate, as well as a $\ZZ_2$-action by $\dig$---composition of the
reflections in the first two coordinates (that is, a $180^\circ$
rotation in the first coordinate plane).
\begin{defn}\label{def:spacref-equivariant}
  Fix a small category $\Cat$ and a strictly unitary $2$-functor $F
  \from \Cat \to \burn_\bullet$.  A \emph{reflection-equivariant or
  $\refl$-equivariant $k$-dimensional spatial refinement} of $F$
  is a $\ZZ_2$-equivariant $k$-dimensional spatial refinement
  $\widetilde{F}_k \from \Cat \to \topp$ such that for any sequence of
  morphisms $u_0 \xrightarrow{f_1} \cdots \xrightarrow{f_n} u_n$ in
  $\Cat$ and any $(t_1,\dots,t_{n-1})\in [0,1]^{n-1}$ the map
  \[
  \widetilde{F}_k(f_n,\dots,f_1) (t_1,\dots,t_{n-1})\from \bigvee_{x\in F(u_0)} S^k \to \bigvee_{x\in F(u_n)}S^k
  \]
  is $\refl$-equivariant. When $\bullet=\xi$, we require each
  $\widetilde{F}_k$ to be $\xi$-equivariant as well; in this case,
  $\widetilde{F}_k$ is a $\ZZ_2\times\ZZ_2$-equivariant diagram, with
  the first factor (denoted $\ZZ^+_2$) acting by $\xi$ and the second
  factor (denoted $\ZZ^-_2$) acting by $\refl\circ\xi$.
\end{defn}

\begin{defn}\label{def:spacref-del-equivariant}
  Fix a small category $\Cat$ and a strictly unitary $2$-functor $F
  \from \Cat \to \oddb$.  A \emph{doubly-equivariant $k$-dimensional
    spatial refinement} of $F$ is a $\ZZ_2$-equivariant
  $k$-dimensional \emph{doubly signed} spatial refinement
  $\widehat{F}_k \from \Cat \to \topp$ such that for any sequence of
  morphisms $u_0 \xrightarrow{f_1} \cdots \xrightarrow{f_n} u_n$ in
  $\Cat$ and any $(t_1,\dots,t_{n-1})\in [0,1]^{n-1}$ the map
  \[
  \widehat{F}_k(f_n,\dots,f_1) (t_1,\dots,t_{n-1})\from \bigvee_{x\in F(u_0)} S^k \to \bigvee_{x\in F(u_n)}S^k
  \]
  is equivariant with respect to reflections in the first two
  coordinates; the $\ZZ_2$-action is given by $\refl$ (there is also a
  second $\ZZ_2$-action by reflection in the second coordinate, which
  we will ignore). Note that since we require the box map to be a
  \emph{doubly signed} refinement of $F(f_n\circ\dots\circ
  f_1)$---defined using $\widehat{\Phi}$ instead of $\Phi$ from
  \S\ref{sec:signed-box-maps}---a doubly-equivariant spatial
  refinement is \emph{not} a spatial refinement of $F$ in the usual
  sense.  We will always denote doubly equivariant spatial refinements
  by $\widehat{F}_k$ to distinguish them from $\refl$-equivariant
  spatial refinements $\widetilde{F}_k$, to which they are not
  homotopy equivalent (even nonequivariantly).
\end{defn}

We next record the equivariant version of Proposition \ref{prop:cube}:
\begin{prop}\label{prop:cube-equivariant}
  Let $\Cat$ be a small category in which every sequence of composable
  non-identity morphisms has length at most $n$, and let $F\from
  \Cat\to\burn_\bullet$ be a strictly unitary 2-functor.
\begin{enumerate}[leftmargin=*]
\item If $k\geq n+1$,
  the refinement of Proposition \ref{prop:cube} may also be constructed
  $\refl$-equivariantly, while if $k\geq n+2$, it may be constructed doubly equivariantly.  \label{itm:real1-equivariant}
\item If $k\geq n+2$, the weak equivalence from Proposition \ref{prop:cube}(\ref{itm:real2}) may be
  constructed $\refl$-equivariantly, while if $k\geq n+3$, it may be constructed doubly equivariantly. \label{itm:real2-equivariant}
\end{enumerate}
Moreover, if $\bullet=\xi$, there is a $\xi$- and $\refl$-equivariant
refinement of $F$ for $k\geq n+1$, and the weak equivalence from
Proposition \ref{prop:cube}(\ref{itm:real2}) may be constructed $\xi$-
and $\refl$-equivariantly for $k\geq n+2$.
\end{prop}
\begin{proof}
  The proof is essentially identical to the proof of Proposition
  \ref{prop:cube}. To make each $\widetilde{F}_k(f_n,\dots,f_1)$
  $\refl$- or doubly-equivariant, simply stipulate that each map
  $e_{f_m,\dots,f_1}$ has image contained in $E_\sym(\{B_x\mid x\in
  F(v_0)\},s_{F(f_m\circ\dots\circ f_1)})$ or $E_{\dsym}(\{B_x\mid
  x\in F(v_0)\},s_{F(f_m\circ\dots\circ f_1)})$.
  \end{proof}

\begin{prop}\label{prop:equivdiag-equivariant}
  If $\widetilde{F}$ is an $\refl$-equivariant $k$-dimensional spatial
  refinement of $F\from\Cat\to\oddb$, then the following hold:
  \begin{enumerate}[leftmargin=*]
  \item The fixed point diagram $\widetilde{F}^{\ZZ_2}$ is a
    $(k-1)$-dimensional refinement of $\forgot\circ F\from\Cat\to\burn$.
  \item The orbit diagram $(\widetilde{F}/\widetilde{F}^{\ZZ_2})/{\ZZ_2}$ is a $k$-dimensional
    refinement of $\forgot\circ F\from\Cat\to\burn$.
  \item The quotient diagram $\widetilde{F}/\widetilde{F}^{\ZZ_2}$ is
    a $k$-dimensional refinement of $\oddeq\circ F\from\Cat\to\eqb$,
    with the $\refl$-action on $\widetilde{F}$ inducing the
    $\ZZ_2^+$-action on $\widetilde{F}/\widetilde{F}^{\ZZ_2}$.
  \end{enumerate}
  If $\widehat{F}$ is a doubly-equivariant $k$-dimensional spatial
  refinement of $F\from\Cat\to\oddb$, then the following hold:
  \begin{enumerate}[leftmargin=*]
  \item The fixed point diagram $\widehat{F}^{\ZZ_2}$ is a
    $(k-1)$-dimensional refinement of $F\from\Cat\to\oddb$.
  \item The orbit diagram $(\widehat{F}/\widehat{F}^{\ZZ_2})/{\ZZ_2}$ is a $k$-dimensional
    refinement of $F\from\Cat\to\oddb$.
  \item The quotient diagram $\widehat{F}/\widehat{F}^{\ZZ_2}$ is
    a $k$-dimensional refinement of $\oddeq\circ F\from\Cat\to\eqb$,
    with the $\refl$-action on $\widetilde{F}$ inducing the
    $\ZZ_2^-$-action on $\widetilde{F}/\widetilde{F}^{\ZZ_2}$.
  \end{enumerate}
  If $\widetilde{F}$ is a $k$-dimensional $\refl$-equivariant spatial
  refinement of $F\from\Cat\to\eqb$, the following hold:
  \begin{enumerate}[resume*]
  \item The orbit diagram $\widetilde{F}/{\ZZ_2^+}$ is a
    $k$-dimensional spatial refinement of $\quotient\circ
    F\from\Cat\to\burn$.  \label{itm:real4}
  \item Say $F=\oddeq\circ G$, and $\widetilde{F}$ is
    $\xi$-invariant. Then the orbit diagram
    $\widetilde{F}/{\ZZ_2^-}$ is a $\refl$-equivariant
    $k$-dimensional spatial refinement of $G\from\Cat\to\oddb$
    with the $\ZZ_2^+$-action on $\widetilde{F}$ inducing the
    $\refl$-action on $\widetilde{F}/{\ZZ_2^-}$. \label{itm:real5}
  \end{enumerate}
\end{prop}
\begin{proof}
  This is an immediate consequence of the definitions.
\end{proof}

As in \S\ref{sec:cube-shaped-realize}, given any such equivariant
spatial refinement $\widetilde{F}_k$ for
$F\from\two^n\to\burn_\bullet$, we construct a $G$-equivariant CW
complex $\CRealize{F}_k$ as $\hocolim(\widetilde{F}_k^+)$, where
$G=\ZZ_2$ if $\bullet=\sigma$ and $G=\ZZ_2^+\times\ZZ_2^-$ if
$\bullet=\xi$. By a \emph{$G$-equivariant CW complex}, we mean a
$G$-space carrying two different CW structures, so that it is a $G$-CW
complex with respect to the first structure, and we use the second
structure to define its cellular chain complex; equivalently, it is a
pair of a $G$-CW complex and a CW complex, along with a homeomorphism
connecting the two, and its cellular chain complex is defined to be
cellular chain complex of the second CW complex (and so, does not
automatically carry a $G$-action). In our construction, the fine
structure from \S\ref{sec:cube-shaped-realize} is the first CW
structure, while the standard structure is the second one. A caveat:
our definition of $G$-equivariant CW complex is quite non-standard.

As before, for any stable functor $(F,r)$, after fixing an equivariant
spatial refinement $\widetilde{F}_k$, we let $\Realize{\Sigma^r
  F}=(\CRealize{F}_k,r-k)$ denote the corresponding $G$-equivariant
finite CW spectrum.

Similarly, if $\widehat{F}_k$ is a $k$-dimensional doubly-equivariant
spatial refinement of $F\from\two^n\to\oddb$, we let
$\widehat{\CRealize{F}}_{k}$ denote the $\ZZ_2$-equivariant CW complex
$\hocolim(\widehat{F}_k^+)$. For any stable functor $(F,r)$, we let
$\widehat{\Realize{\Sigma^rF}}$ denote the corresponding finite
$\ZZ_2$-equivariant CW spectrum. We call
$\widehat{\Realize{\Sigma^rF}}$ the \emph{doubly-equivariant
  realization} of $(F,r)$, to avoid confusion with the ordinary
realization of $(F,r)$.

\begin{cor}\label{cor:burnlimwelldef-equivariant}
  Let $F\from\two^n\to\burn_\bullet$.
  \begin{enumerate}[leftmargin=*]
  \item If $\bullet=\sigma$, then
    $\Realize{F}^{\ZZ_2}=\Sigma^{-1}\Realize{\forgot\circ F}$,
    $\Realize{F}/\Realize{F}^{\ZZ_2}=\Realize{\oddeq \circ F}$ (with
    the $\ZZ_2$-action on the left equaling the $\ZZ_2^+$-action on
    the right), and
    $(\Realize{F}/\Realize{F}^{\ZZ_2})/{\ZZ_2}=\Realize{\forgot\circ
      F}$.
  \item If $\bullet=\xi$, then
    $\Realize{F}/\ZZ_2^+=\Realize{\quotient\circ F}$; if
    $F=\oddeq\circ G$, then $\Realize{F}/\ZZ_2^-=\Realize{G}$ (with
    the $\ZZ_2^+$-action on the left equaling the $\ZZ_2$-action on
    the right).
  \item If $\bullet=\sigma$, then
    $\widehat{\Realize{F}}^{\ZZ_2}=\Sigma^{-1}\Realize{F}$,
    $\widehat{\Realize{F}}/\widehat{\Realize{F}}^{\ZZ_2}=\Realize{\oddeq
      \circ F}$ (with the $\ZZ_2$-action on the left equaling the
    $\ZZ_2^-$-action on the right), and
    $(\widehat{\Realize{F}}/\widehat{\Realize{F}}^{\ZZ_2})/{\ZZ_2}=\Realize{F}$.
  \end{enumerate}
\end{cor}
\begin{proof}
This follows from Proposition \ref{prop:cube-equivariant} and the properties of homotopy colimits \ref{itm:ho4}---\ref{itm:ho5}.
\end{proof}

\begin{prop}\label{prop:totalization-equivariant}
  If $F\from\two^n\to\burn_\xi$, then the reduced cellular complex of
  its realization, $\redcellC(\Realize{F})$, is isomorphic, as a
  $\ZZ_u$-module, to the totalization $\Tot(F)$ with the cells mapping
  to the corresponding generators. If $F\from\two^n\to\burn_\sigma$,
  then the reduced cellular complex of its doubly-equivariant
  realization, $\redcellC(\widehat{\Realize{F}})$ is isomorphic to the
  totalization $\Tot(\forgot\circ F)$ with the cells mapping to the
  corresponding generators.
\end{prop}
\begin{proof}
  The statement is, nonequivariantly, just Proposition
  \ref{prop:totalization}. The isomorphism as $\ZZ_u$-modules follows
  from inspection of the proof of \cite[Theorem 6]{lls1}. The second
  statement is proved as in Proposition \ref{prop:totalization}.
\end{proof}

\begin{lem}\label{lem:functor-map-gives-space-map-equivariant}
  Let $\Sigma^{r_1}F_1\to \Sigma^{r_2}F_2$ be an equivariant map
  between stable odd Burnside functors $(F_1\from\two^{n_1}\to\oddb,r_1)$
  and $(F_2\from\two^{n_2}\to\oddb,r_2)$. Then there is an induced
  $\ZZ_2$-equivariant map of equivariant realizations
  \[
  \Realize{\Sigma^{r_1}F_1}\to \Realize{\Sigma^{r_2}F_2}, 
  \]
  an induced $\ZZ_2$-equivariant map of doubly-equivariant
  realizations
  \[
  \widehat{\Realize{\Sigma^{r_1}F_1}}\to \widehat{\Realize{\Sigma^{r_2}F_2}}, 
  \]
  and an induced $\ZZ_2^+\times\ZZ_2^-$-equivariant map of equivariant
  realizations 
  \[
  \Realize{\Sigma^{r_1}\oddeq F_1}\to \Realize{\Sigma^{r_2}\oddeq F_2},
  \]
  all well-defined up to homotopy equivalence.  If the map of Burnside
  functors is an equivariant equivalence, then these induced maps are
  equivariant homotopy equivalences.
\end{lem}
\begin{proof}
  Let us sketch the arguments in the first case. The other two cases
  are similar.

  First, note that Lemma \ref{lem:cofib} can be made to hold
  equivariantly, so that associated to a natural transformation of
  signed Burnside functors there is an equivariant map.  So we only
  need to show that associated to an equivariant equivalence of
  Burnside functors, there is a well-defined equivariant stable
  homotopy equivalence of their realizations.

  It will suffice to show that each of the moves in
  Definition~\ref{def:stableq} induces an equivariant stable homotopy
  equivalence.  For the stabilization move this is clear.  For the
  first move, assume that we have a natural transformation $\eta$,
  with $\Tot(\oddeq \eta)$ a homotopy equivalence over $\ZZ_u$.  By
  Proposition \ref{prop:totalization} and Lemma
  \ref{lem:functor-map-gives-space-map}, the induced map between
  realizations is cellular and induces a non-equivariant homotopy
  equivalence.  The induced map on fixed-point sets is induced by the
  underlying natural transformation $\forgot \eta$ using the
  identification of Corollary \ref{cor:burnlimwelldef-equivariant},
  which is a homotopy equivalence of chain complexes since
  $\Tot(\oddeq \eta)$ is.  By the proof of Lemma
  \ref{lem:functor-map-gives-space-map}, and using Corollary
  \ref{cor:burnlimwelldef-equivariant} again, this induced map on
  fixed-point sets is a homotopy equivalence.  By the $G$-Whitehead
  theorem, the induced map is a $\ZZ_2$-homotopy equivalence.
\end{proof}

We record the behavior under coproducts.
\begin{lem}\label{lem:product-realization-fixed}
  Let $F_1,F_2 \from \two^{n} \to \burn_\bullet$.  Then $\Realize{F_1
    \amalg F_2}$ is equivariantly homeomorphic to
  $\Realize{F_1}\vee\Realize{F_2}$. If $\bullet=\sigma$, then
  $\widehat{\Realize{F_1 \amalg F_2}}$ is equivariantly homeomorphic
  to $\widehat{\Realize{F_1}}\vee\widehat{\Realize{F_2}}$
\end{lem}
\begin{proof}
  The statement is an immediate consequence of \ref{itm:ho2} and
  \S\ref{subsec:prod}.
\end{proof}

\begin{rmk}\label{rmk:other-homotopy-types}
  In fact, for $F\from\two^n\to\oddb$ and any $\ell\geq 0$, the
  constructions of the present section may be carried out using
  reflection in the first $\ell$ coordinates to produce a
  $\ZZ_2^\ell$-equivariant CW-spectrum $\Realize{F}^\ell$; we have
  encountered the first few cases: $\Realize{\forgot
    F}=\Realize{F}^0$, $\Realize{F}=\Realize{F}^1$,
  $\widehat{\Realize{F}}=\Realize{F}^2$.  Its cellular complex equals the
  totalization $\Tot(\forgot F)$ if $\ell$ is even and $\Tot(F)$ if
  $\ell$ is odd.  Proposition \ref{prop:equivdiag-equivariant}, as
  well as its corollaries above, readily generalizes to these
  realizations, and the entire family is related by iterated quotients
  (or fixed point sets) under the various actions.
\end{rmk}

\section{Khovanov homotopy types}\label{sec:oddkh}
In this section, we construct the odd Khovanov Burnside functor, and
the odd Khovanov homotopy type as its realization. We also construct a
reduced odd Khovanov homotopy type and the unified Khovanov homotopy
type. We establish various properties such as fixed
point constructions and cofibration sequences. We also construct
several concordance invariants following standard procedure.

\subsection{The odd Khovanov Burnside functor}\label{subsec:def}
In this section, we define a functor to the signed Burnside category
associated to an oriented link diagram $L$ with oriented crossings.
After ordering the $n$ crossings of $L$, we will identify the vertices
of the hypercube of resolutions of $\diagram$ with the objects of
$\two^n$, and the edges with the length one arrows of $\two^n$.

To define the odd Khovanov Burnside functor $F_o\from\two^n\to\oddb$,
following Lemma \ref{lem:212}, we need only define it on objects,
length one morphisms, and across two-dimensional faces of the cube
$\two^n$.  On objects we set
\[
F_{o}(u)=\KhGen(u).
\]
For each edge $u\geqslant_1 v$ in $\two^n$, and each element
$y \in F_{o}(v)$, write
\[
\AbFunc_{o}(\phi^{\op}_{v,u})(y)=\sum_{x\in F_{o}(u) } \epsilon_{x,y} x,
\]
where $\AbFunc_o$ is the odd Khovanov functor from
\S\ref{sec:threekhfunc}. Note each $\epsilon_{x,y}\in \{ -1,0,1\}$.
Define
\[
F_{o}(\phi_{u,v})=\{ (y,x) \in F_{o}(v) \times F_{o}(u) \mid
\epsilon_{x,y}=\pm 1 \},
\]
where the sign on elements of $F_{Kh'}(\phi_{u,v})$ is given by
$\epsilon_{x,y}$ of the pair, and the source and target maps are the
natural ones.

We need only define the $2$-morphisms across $2$-dimensional faces.
In fact, in contrast to the case of even Khovanov homology, where a
global choice is necessary in order to define the $2$-morphisms
\cite{lshomotopytype}, in odd Khovanov homology there is a unique
choice of $2$-morphisms compatible with the preceding data.  To be
more specific, for any $2$-dimensional face
$u\geqslant_1v,v'\geqslant_1w$, and any pair
$(x,y) \in F_{o}(u) \times F_{o}(w)$, there is a unique bijection
between
\[
A_{x,y}:=s^{-1}(x)\cap t^{-1}(y)\subset F_{o}(\phi_{v,w})\times_{F_{o}(v)} F_{o}(\phi_{u,v})
\] 
and
\[
A_{x,y}':=s^{-1}(x)\cap t^{-1}(y)\subset F_{o}(\phi_{v',w})\times_{F_{o}(v')} F_{o}(\phi_{u,v'})
\]
that preserves the signs.  (That is, the signed sets
$A_{x,y},A_{x,y}'$ both have at most one element of any given sign).

The last assertion may be checked on a case-by-case basis, using
\cite[Figure 2]{ors}.  Away from ladybug configurations (i.e.,
configurations of type X and Y) the involved sets both have at most
one element (and $\AbFunc_o$ commutes across 2d faces), so the result
is automatic.  For ladybug configurations, there are sets
$A_{x,y}, A_{x,y}'$ with two elements, but the elements have opposite
sign, and so there is still a unique matching.

Next, we observe that the compatibility relation demanded by Lemma
\ref{lem:212} is satisfied by $F_{o}$ (using the unique bijections
across $2$-faces).  For this, we must consider $3$-dimensional faces
$\iota \from \two^3 \to \two^n$, and a choice of elements
$x \in F_{o}(\iota(1,1,1)),y\in F_{o}(\iota(0,0,0))$ and the
correspondence $A_{x,y}$.  There are six distinct decompositions of
the arrow $(1,1,1) \to (0,0,0) $ in $\two^3$ into a composition of
nonidentity arrows, corresponding to permutations of $\{1,2,3\}$.
Specifically, if $e_i$ denotes the arrow $1\to 0$ in the $i\th$-factor
of $\two^3$, the permutation $\sigma$ corresponds to the composition
$e_{\sigma(3)}\circ e_{\sigma(2)}\circ e_{\sigma(1)}$.  These
compositions are in turn related by $2$-morphisms
\[
F_{i,j}\from F(e_{i})\circ F(e_{j}) \to F(e_{j})\circ F(e_{i}).
\]
The compatibility relation of Lemma \ref{lem:212} boils down to the
condition that the following diagram commutes:
\begin{center}
\begin{tikzpicture}[xscale=0.8,yscale=0.7]
    \def\radius{3.7cm} 
    \node (h0A) at (60:\radius)   {$F_{e_3}\!\!\circ\! F_{e_2}\!\!\circ\! F_{e_1}$};
    \node (h0C) at (0:\radius)    {$F_{e_2}\!\!\circ\! F_{e_3} \!\!\circ\! F_{e_1}$};
    \node (h1B) at (-60:\radius)  {$F_{e_2}\!\!\circ\! F_{e_1} \!\!\circ\! F_{e_3}$};
    \node (h1A) at (-120:\radius) {$F_{e_1} \!\!\circ\! F_{e_2} \!\!\circ\! F_{e_3}$};
    \node (h1C) at (180:\radius)  {$F_{e_1} \!\!\circ\! F_{e_3} \!\!\circ\! F_{e_2}$};
    \node (h0B) at (120:\radius)  {$F_{e_3} \!\!\circ\! F_{e_1}\!\!\circ\! F_{e_2}$};

    \draw[->]
        (h0A) edge node[auto] {\tiny $F_{32} \times \Id$} (h0C)
        (h0C) edge node[auto] {\tiny $\Id \times F_{31}$} (h1B)
        (h1B) edge node[auto] {\tiny  $F_{21} \times \Id$} (h1A)
        (h1A) edge node[auto] {\tiny $\Id \times F_{23}$} (h1C)
        (h1C) edge node[auto] {\tiny $F_{13} \times \Id$} (h0B)
        (h0B) edge node[auto] {\tiny $\Id \times F_{12}$} (h0A);
\end{tikzpicture}
\end{center}

However, it turns out that for any choice of $x,y$ as above, there is
at most one element of a given sign in
$A^\sigma_{x,y}:=s^{-1}(x)\cap t^{-1}(y)\subset F(e_{\sigma(3)})\circ
F(e_{\sigma(2)})\circ F(e_{\sigma(1)})$, and therefore, the coherence
check is automatic.  To see this is a simple enumeration of all
possible options. In more detail, following \cite{lshomotopytype}, for
3d configurations that do not contain ladybug configurations on any of
their 2d faces, each of the six sets $A^\sigma_{x,y}$ contain at most
one element.  For the remaining configurations, it is shown in
\cite{lshomotopytype} that each of the six sets $A^\sigma_{x,y}$
contain at most two elements. However, since these remaining
configurations contain ladybugs, these two elements must have opposite
signs. (Recall that if $u\geqslant_1 v,v'\geqslant_1w$ is a ladybug
configuration and
$s^{-1}(x)\cap t^{-1}(y)\subset F_o(\phi_{v,w})\circ F_o(\phi_{u,v})$ is
non-empty, then it consists of two oppositely signed points.)
Therefore, each of the six sets $A^\sigma_{x,y}$ contains at most one
element of each sign for the remaining configurations as well.

\begin{defn}\label{def:kh-odd-burnside-functor}
  Define the stable signed Burnside functor associated to an oriented
  link diagram $L$ with $n$ oriented crossings and a choice of edge
  assignment to be $(F_o,-n_-)$, where $F_o\from\two^n\to\oddb$ is the
  functor defined above, and $n_-$ is the number of negative crossings
  in $L$. Since the differential on the Khovanov chain complex
  respects the quantum grading, cf.~\S\ref{sec:khovanovhomology}, the
  odd Khovanov Burnside functor splits as a coproduct of functors, one
  in each quantum grading:
  \[
  F_{o}=\coprod_j F^j_{o}.
  \]
  The total complex of the odd
  Khovanov Burnside functor $\Sigma^{-n_-}F^j_{o}$ agrees with the
  dual of the odd Khovanov chain complex:
  \[
  (\Tot(\Sigma^{-n_-}F^j_{o}))^*=\KhCx^{*,j}_o(L).
  \]
\end{defn}
\begin{defn}\label{def:oddkhmain}
  We define the \emph{odd Khovanov spectrum}
  $\khoo(L)=\bigvee_j\khoo^j(L)$ as a $\ZZ_2$-equivariant finite CW
  spectrum, where $\khoo^j(L)$ is a realization of the stable signed
  Burnside functor $\Sigma^{-n_-}F^j_{o}$
\end{defn}
This odd functor recovers the even functor $F_e\from\two^n\to\burn$
from \cite{lls1,lls2} as follows.
\begin{prop}\label{prop:odd-recovers-even}
  The functors $F_o$ and $F_e$ satisfy $\forgot\circ F_{o}=F_{e}$,
  where $\forgot\from\oddb\to\burn$ is the forgetful functor from
  Figure~\ref{fig:burnsidecategories}.
\end{prop}

\begin{proof}
  The generators of even and odd Khovanov homology are canonically
  identified, cf.~\S\ref{subsec:kg}, so on objects we have a canonical
  identification $\forgot F_{o}(u)=F_{e}(u)=\KhGen(u)$.  Similarly, since the
  differentials agree up to sign (for this identification), we have
  that $\forgot F_{o}(\phi_{u,v})$ is canonically identified with
  $F_{e}(\phi_{u,v})$ for $u\geqslant_1 v$.  So we just need to show
  that the $2$-morphism $\forgot F_{o}(\phi_{v,w}) \circ \forgot
  F_{o}(\phi_{u,v})\to\forgot F_{o}(\phi_{v',w})\circ \forgot
  F_{o}(\phi_{u,v'})$ agrees with the the $2$-morphism
  $F_e(\phi_{v,w}) \circ F_e(\phi_{u,v})\to F_e(\phi_{v',w})\circ
  F_e(\phi_{u,v'})$ for all 2d faces $u\geqslant_1 v,v'\geqslant_1 w$.

  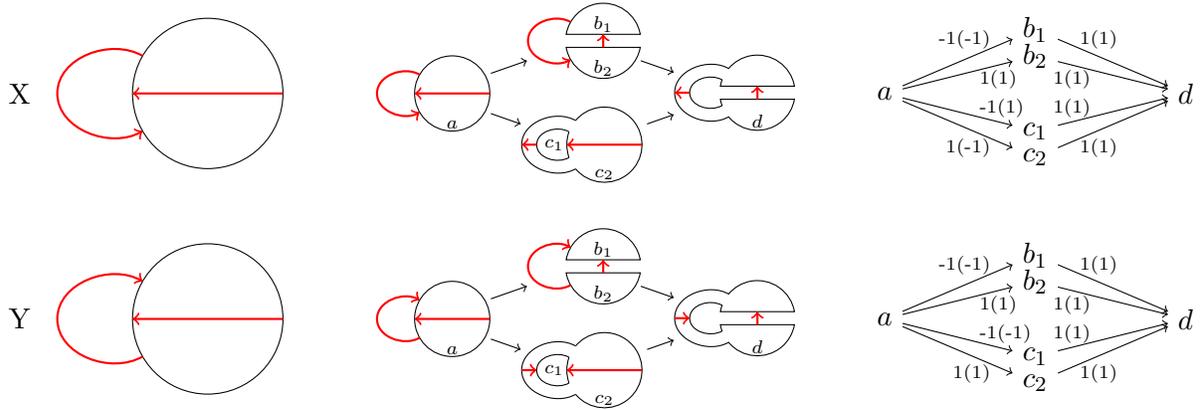
\begin{figure}
    \centering
    \begin{tikzpicture}
      \foreach \typ/\arr/\epone/\epmone [count=\c from 0] in {X/->/-1/1,Y/<-/1/-1}{
        \begin{scope}[yshift=-3*\c cm]
          \node at (-2.5,0) {\typ};
          
          \begin{scope}
            \draw (0,0) circle (1cm);
            \draw[thick,red,<-] (-1,0) --(1,0);
            \draw[thick,red,\arr] (150:1cm) to[out=150,in=90] (-2,0) to[out=-90,in=210] (210:1cm);
          \end{scope}
         
          \node[inner sep=0pt,outer sep=0pt] (a) at (3,0) {\begin{tikzpicture}[scale=0.5]
            \draw (0,0) circle (1cm);
            \draw[thick,red,<-] (-1,0) --(1,0);
            \draw[thick,red,\arr] (150:1cm) to[out=150,in=90] (-2,0) to[out=-90,in=210] (210:1cm);
            \node[anchor=south,inner sep=1pt] at (-90:1cm) {\tiny $a$};
          \end{tikzpicture}};
          
        \node[inner sep=0pt,outer sep=0pt] (b) at (5,0.7) {\begin{tikzpicture}[scale=0.5]
            \draw (10:1cm) arc (10:170:1cm) -- cycle;
            \draw (190:1cm) arc (190:350:1cm) -- cycle;
            \draw[thick,red,->] ($(190:1cm)!0.5!(350:1cm)$) -- ($(10:1cm)!0.5!(170:1cm)$);
            \draw[thick,red,\arr] (150:1cm) to[out=150,in=90] (-2,0) to[out=-90,in=210] (210:1cm);

            \node[anchor=south,inner sep=1pt] at ($(10:1cm)!0.5!(170:1cm)$) {\tiny $b_1$};
            \node[anchor=south,inner sep=1pt] at (-90:1cm) {\tiny $b_2$};
          \end{tikzpicture}};
          
          \node[inner sep=0pt,outer sep=0pt] (c) at (5,-0.7) {\begin{tikzpicture}[scale=0.5]
            \draw (160:1cm) arc (160:200:1cm);
            \draw (220:1cm) arc (220:360+140:1cm);
            \draw (160:1cm) to[out=150,in=90] (-2+0.2,0) to[out=-90,in=210] (200:1cm);
            \draw (140:1cm) to[out=150,in=90] (-2-0.2,0) to[out=-90,in=210] (220:1cm);
            \draw[thick,red,<-] (-1,0) --(1,0);
            \draw[thick,red,\arr] (-2+0.2,0) -- (-2-0.2,0);
            \node[anchor=east,inner sep=1pt] at (180:1cm) {\tiny $c_1$};
            \node[anchor=south,inner sep=1pt] at (-90:1cm) {\tiny $c_2$};
          \end{tikzpicture}};
          
        \node[inner sep=0pt,outer sep=0pt] (d) at (7,0) {\begin{tikzpicture}[scale=0.5]
            \draw (10:1cm) arc (10:140:1cm) to[out=150,in=90] (-2-0.2,0) to[out=-90,in=210] (220:1cm) arc (220:350:1cm) -- (190:1cm) arc (190:200:1cm) to[out=210,in=-90] (-2+0.2,0) to[out=90,in=150] (160:1cm) arc (160:170:1cm) -- cycle;
            \draw[thick,red,->] ($(190:1cm)!0.5!(350:1cm)$) -- ($(10:1cm)!0.5!(170:1cm)$);
            \draw[thick,red,\arr] (-2+0.2,0) -- (-2-0.2,0);
            \node[anchor=south,inner sep=1pt] at (-90:1cm) {\tiny $d$};
          \end{tikzpicture}};

          \draw[->] (a) edge (b) edge (c);
          \draw[<-] (d) edge (b) edge (c);
          
          \begin{scope}[xshift=6cm]
            
            \node (a) at (3,0) {$a$};
            \node (b) at (5,0.85) {$b_1$};
            \node (bb) at (5,0.5) {$b_2$};
            \node (c) at (5,-0.5) {$c_1$};
            \node (cc) at (5,-0.85) {$c_2$};
            \node (d) at (7,0) {$d$};

            \draw[->] (a) edge node[pos=0.8,inner sep=0,outer sep=0,anchor=south east] {\tiny -1(-1)} (b)  edge node[pos=0.7,inner sep=0,outer sep=0,anchor=north west] {\tiny 1(1)} (bb) edge node[pos=0.7,inner sep=0,outer sep=0,anchor=south west] {\tiny -1(\epmone)} (c) edge node[pos=0.8,inner sep=0,outer sep=0,anchor=north east] {\tiny 1(\epone)} (cc);

            \draw[<-] (d) edge node[pos=0.8,inner sep=0,outer sep=0,anchor=south west] {\tiny 1(1)} (b) edge node[pos=0.7,inner sep=0,outer sep=0,anchor=north east] {\tiny 1(1)} (bb) edge node[pos=0.7,inner sep=0,outer sep=0,anchor=south east] {\tiny 1(1)} (c) edge node[pos=0.8,inner sep=0,outer sep=0,anchor=north west] {\tiny 1(1)} (cc);

          \end{scope}

        \end{scope}
      }
    \end{tikzpicture}
    \caption{\textbf{The odd functor for the type-X assignment
        recovers the even functor for the right ladybug matching.}
      Consider the two types of ladybugs, X and Y, and name the
      circles appearing in the various resolutions
      $a,b_1,b_2,c_1,c_2,d$ as shown (their ordering does not
      matter). The coefficients of the relevant portion of the functor
      $\AbFunc_o$ (as well as those of the assignment $\AbFunc'_o$ in
      parentheses) are shown. Since we are considering a type-X
      assignment, $\AbFunc_o$ is chosen to differ from $\AbFunc'_o$ in
      one edge for the X-ladybug, and is chosen to agree with
      $\AbFunc'_o$ for the Y-ladybug. In either case, the unique
      sign-respecting $2$-isomorphism is the matching
      $(a,b_1,d)\leftrightarrow (a,c_1,d), (a,b_2,d)\leftrightarrow
      (a,c_2,d)$, which is the right ladybug matching
      from~\cite{lshomotopytype}.}\label{fig:odd-to-even}
  \end{figure}

  For 2d configurations apart from ladybugs, for any
  $x\in\KhGen(u),y\in\KhGen(w)$, the subset $s^{-1}(x)\cap t^{-1}(y)$
  in each of the correspondences contain at most one element, and so
  the two $2$-morphisms agree.  For ladybugs, one may check directly
  that the $2$-morphism for $\forgot F_{o}$ agrees with that for
  $F_{e}$.  To be more specific, the $2$-morphism for $\forgot F_{o}$
  specified by a type-X sign assignment agrees (using the above
  identifications of objects and $1$-morphisms) with the right ladybug
  matching for $F_{e}$ (a type-Y assignment corresponds to left
  ladybug matching), see Figure~\ref{fig:odd-to-even} for details.  By
  Lemma \ref{lem:212}, $\forgot F_{o}$ is isomorphic to $F_{e}$.
\end{proof}

Therefore, the even Khovanov spectrum from \cite{lshomotopytype} is
$\khoh(L)=\bigvee_j\khoh^{j}(L)$ with $\khoh^{
  j}(L)=\Realize{\Sigma^{-n_-}\forgot F^j_{o}}$. Using the
doubly-equivariant realizations, we have a related spectrum:
\begin{defn}\label{def:evenaction}
  We define a second even Khovanov spectrum, denoted 
  $\khoh'(L)=\bigvee_j\khoh^{\prime\,j}(L)$ as a $\ZZ_2$-equivariant finite CW
  spectrum, where $\khoh^{\prime\,j}(L)=\widehat{\Realize{\Sigma^{-n_-}F^j_{o}}}$, a
  doubly-equivariant realization of the stable signed Burnside functor
  $\Sigma^{-n_-}F^j_{o}$. 
\end{defn}

\begin{defn}\label{def:unifiedintro}
  We define the \emph{unified Khovanov spectrum}
  $\unis(L)=\bigvee_j\unis^j(L)$ as a $\ZZ^+_2\times\ZZ^-_2$-equivariant
  finite CW spectrum, where $\unis^j(L)$ is a realization of the
  stable $\ZZ_2$-equivariant Burnside functor
  $\Sigma^{-n_-}\oddeq F^j_{o}$.
\end{defn}

\begin{rmk}\label{rmk:khovanov-spaces-with-extra-action}
  Following Remark \ref{rmk:other-homotopy-types}, there is in fact a
  family of Khovanov spaces $\X_\ell(L)$, for $\ell\geq 0$, whose
  cellular chain complexes agree with the even Khovanov chain complex
  $\KhCx_e(L)$ if $\ell$ is even and the odd Khovanov chain complex
  $\oddKhCx(L)$ if $\ell$ even is odd. (We have already encountered
  $\X_0(L)=\khoh(L)$, $\X_1(L)=\khoo(L)$, and $\X_2(L)=\khoh'(L)$.)
  There is a natural generalization of Theorem \ref{thm:equivariance}
  to this family of spaces.  We conjecture that $\X_\ell(L)$, up to
  homotopy equivalence, only depends on the parity of $\ell$.
\end{rmk}
\subsection{Relations among the three theories}\label{sec:uni}
In this section, we find relations among the three Khovanov homotopy
types, in terms of geometric fixed points, geometric quotients, and
cofibration sequences.

\begin{proof}[Proof of Theorem~\ref{thm:equivariance}]
  The first statement and the first parts of statements
  (\ref{itm:thm-equivariance-2}) and (\ref{itm:thm-equivariance-4})
  follow from Corollary \ref{cor:burnlimwelldef-equivariant}.  The
  exact sequences are a consequence of the $\ZZ/2$-actions on
  $\khoh'(L)$ and $\khoo(L)$, for which $\Sigma^{-1}\khoo(L)$ and
  $\Sigma^{-1}\khoh(L)$ are respectively the fixed point sets, and
  $\unis(L)$ the quotients.  The inclusion of the fixed-point sets are
  cofibrations in both cases, giving the desired exact sequences.  The
  agreement with the exact sequences of \cite{putyrashumakovitch} at
  the level of cohomology is a consequence of
  (\ref{itm:thm-equivariance-3}) and
  (\ref{itm:thm-equivariance-5}). So it remains to prove
  (\ref{itm:thm-equivariance-3}) and
  (\ref{itm:thm-equivariance-5}). The proofs are similar, so let us
  only consider (\ref{itm:thm-equivariance-3}).

  Consider the Puppe sequence associated to the inclusion
  $\Sigma^{-1}\khoh(L) \into \khoo(L)$. For concreteness, assume
  $\khoo(L)$ has been constructed equivariantly using $k$-dimensional
  boxes, and all the sub-boxes of $[0,1]^k$ involved in the
  construction are of the form $[0,1]\times B$; that is, they extend
  the full length in the first coordinate. Let $X$ denote $\khoo(L)$,
  $Y$ denote the fixed set $\Sigma^{-1}\khoh(L)$, and $Z$ denote the
  quotient $X/Y=\unis(L)$. The Puppe sequence takes the form
  \[
  Y \into X \to X\cup C(Y)\xrightarrow{P}\Sigma Y,
  \]
  with $C$ denoting the cone, where the last map $P$ is quotienting by
  $X$.

  The term $X\cup C(Y)$ is homotopy-equivalent to
  $Z$ by quotienting by $C(Y)$:
  \[
  Q\from X\cup C(Y)\to X/Y =Z.
  \]
  So the Puppe map $Z\to\Sigma Y$ is the homotopy inverse of $Q$,
  composed with $P$.

  Meanwhile, we have the map $R \from \unis(L)=Z \to \khoh(L)=\Sigma
  Y$ given by quotienting by $\ZZ_2^+$. Recall that the
  $\ZZ_2^+$-action on $Z=X/Y$ is induced from the $\ZZ_2$-action on
  $X$. We wish to show that these two maps from $Z$ to $\Sigma Y$ are
  homotopic. Since $Q$ is a homotopy equivalence, it is enough to show
  that the two maps $P,R\circ Q\from X\cup C(Y)\to \Sigma Y$ are
  homotopic.

  Consider the quotient of $X$ by the $\ZZ_2$-action. Since $X$ has
  been constructed using boxes that stretch the full length along the
  first coordinate, it is not hard to see that the quotient is
  $C(Y)$. This produces a quotient map $S\from X\cup C(Y)\to C(Y)\cup
  C(Y)=\Sigma Y$.

  Both the maps $P$ and $R\circ Q$ factor through the above map
  $S$. The first map quotients the first $C(Y)$ factor, while the
  second map quotients the second factor. Either is homotopic to the
  identity map $C(Y)\cup C(Y)\to\Sigma Y$, and so the claim follows.
\end{proof}

\subsection{Invariance}\label{subsec:invar}
The main aim of this section is to prove that changes of
the orientation of the crossings, as well as Reidemeister moves,
result in equivariantly equivalent signed Burnside functors.

\begin{proof}[Proof of Theorem~\ref{thm:odd-functor-invariant}] We
  will now prove that the equivariant equivalence class of the odd
  Khovanov Burnside stable functor from
  Definition~\ref{def:kh-odd-burnside-functor} is independent of the
  choices in its construction, namely the choice of diagram
  $\diagram$, the orientation of the crossings, the edge assignment,
  and the ordering of the generators $a_i$ at each resolution.

We first see that for a fixed diagram $L$, changing the other
auxiliary choices results in sign reassignments (which are sometimes
isomorphisms) of functors from $\two^n$ to $\oddb$.

\begin{itemize}[leftmargin=*]
\item \textbf{Edge assignment:} Let $\epsilon,\epsilon'$ be two
  different edge assignments of the same type for the same oriented
  knot diagram $L$.  As noted in \cite[Lemma 2.2]{ors},
  $\epsilon\epsilon'$ is a (multiplicative) cochain in
  $\cellC^1([0,1]^n,\ZZ_2)$, and hence a coboundary of a $0$-cochain
  $\alpha$ on the cube of resolutions.  That is, there is a map
  $\alpha\from \two^n\to \{ \pm 1\}$, so that for any $v\geqslant_1w$
  $\alpha(v)\alpha(w)=\epsilon(\phi^\op_{w,v})\epsilon'(\phi^\op_{w,v})$.
  If $F_o$ and $F'_{o}$ are the corresponding functors
  $\two^n\to\oddb$, we get that $F'_o$ is obtained from $F_o$ by using
  the sign reassignment associated to $\alpha$.

\item \textbf{Orientations at crossings:} Recall that \cite[Lemma
  2.3]{ors} asserts that for oriented diagrams $(L,o)$ and $(L,o')$
  and an edge assignment $\epsilon$ for $(L,o)$, there exists an edge
  assignment of the same type $\epsilon'$ for $(L,o')$ so that
  $\oddKhCx(L,o,\epsilon)\cong \oddKhCx(L,o',\epsilon')$.  The
  isomorphism constructed in the lemma respects the Khovanov
  generators, and so induces an isomorphism of signed Burnside
  functors.  To be more specific, note that the Khovanov generators
  $\KhGen(L)$ of $\oddKhCx(L,o)$ are independent of the orientation
  $o$ (which only changes the differential).  Then the choice of edge
  assignment $\epsilon'$ is such that the identity morphism
  $\oddKhCx(L,o,\epsilon)\to \oddKhCx(L,o',\epsilon')$ commutes with
  the differentials.  Then the corresponding Burnside functors are
  also naturally isomorphic. (Independence of the orientations of
  crossings can also be proved using Reidemeister II moves twice, as
  in \cite[Figure~4.5]{SSS-geometric-perturb}.)

\item \textbf{Type of edge assignment:} \cite[Lemma 2.4]{ors} proves
  that an edge assignment $\epsilon$ of a decorated link diagram
  $(L,o)$ of type $X$ can also be viewed as a type $Y$ edge
  assignment for some orientation $o'$.  That is, the type-$X$
  Burnside functor associated to $(L,o,\epsilon)$ is already the
  type-$Y$ Burnside functor associated to
  $(L,o',\epsilon)$. (Independence of the type of edge assignment
  can also be achieved by Viro's trick of reflecting the knot diagram
  along the vertical line (which switches the $X$ and the $Y$
  ladybug), and then using a sequence of Reidemeister moves to come
  back to the original diagram, cf.~\cite[Proposition
  6.5]{lshomotopytype}.)

\item \textbf{Ordering of circles at each resolution:} Finally, we
  must check that reordering the circles of a resolution results in an
  equivariantly-equivalent signed Burnside functor.  For this, let
  $\KhGen(u)$ and $\KhGen'(u)$ denote the Khovanov generators for two
  differing orderings of the circles for a fixed link diagram.  These
  orderings are related by a bijection from $\KhGen(u)$ to
  $\KhGen'(u)$.  It is simple to check that these bijections relate
  the two functors $F_o,F'_o\from\two^n\to\oddb$ by a sign reassignment.
\end{itemize}

Next we move on to the main issue for proving well-definedness of the
stable equivalence class of the odd Khovanov Burnside functor:
Reidemeister moves.  For proving invariance under the Reidemeister
moves, the argument mostly follows the proof of \cite[Theorem
1]{lls2}, which lifts Khovanov's invariance proof \cite{kho1} to the
level of Burnside functors. For two diagrams differing by a
Reidemeister move, Khovanov's invariance proof---see also
Bar-Natan~\cite{natancat}---is built using chain maps between Khovanov
complexes, which are either subcomplex inclusions or quotient complex
projections that send Khovanov generators to Khovanov generators and
are chain homotopy equivalences, or their chain homotopy
inverses. Since these maps send Khovanov generators to Khovanov
generators, it is easy to see that the argument lifts to the Burnside
category level~\cite{lls2}; we similarly sketch how in the odd case,
the invariance proof from \cite{ors} lifts to the odd Burnside
functor. (The astute reader will observe that for the Reidemeister III
proof by Khovanov, the chain maps do not send Khovanov generators to
Khovanov generators. This issue is faced in \cite{lshomotopytype}
already.  It seems possible to carry through the approach of
\cite{kho1,natancat,ors}, but at the expense of considering functors
from categories other than cube categories.  However, we will instead
follow the proof of Reidemeister III invariance of
\cite{lshomotopytype} by considering the braid-like Reidemeister III.)

The standard way to prove Reidemeister invariance---applicable in the
even, odd, and unified theory---is the following. Start with the
Khovanov chain complex of one diagram, and perform a sequence of
replacements to arrive at the Khovanov chain complex of the other
diagram, where each replacement is either:
\begin{enumerate}[leftmargin=*,label=(c-\arabic*)]
\item\label{item:merge-cancel} \emph{Replacing the complex with a
    quotient complex associated to a merge.}  Namely, for a merge
  taking circles $a_1,a_2$ in $\diagram_0$ to $a$ in $\diagram_1$,
  there is an acyclic subcomplex spanned, at $\diagram_0$, by
  generators that do not contain $a_1$, and all generators at
  $\diagram_1$; replace by the quotient by this subcomplex.
\item\label{item:split-cancel} \emph{Replacing the complex with the
    subcomplex associated to a split.}  Namely, for a split taking one
  circle $a$ in $L_0$ to $a_1,a_2$ in $L_1$, there is a subcomplex
  spanned by all generators from $\diagram_1$ which do not contain an
  $a_1$ factor, and the corresponding quotient is acyclic; replace by
  this subcomplex.
\end{enumerate}
It is easy to check that the relevant maps---the quotient complex
projection in the first case and the subcomplex inclusion in the
second case---are chain homotopy equivalences in the unified theory
over $\ZZ_u$, and hence also in the odd and the even theory. (These
cancellations are parametrized by \emph{cancellation data} from
\cite[Definition 4.4]{SSS-geometric-perturb}.)

To lift this argument to the Burnside functor level, in the first
case, we will replace the functor by a sub-functor, and in the second
case, by a quotient functor from
\S\ref{sec:nat-transform-burn}. (Recall, the Khovanov complex is the
\emph{dual} of the totalization of the Burnside functor, hence
sub-functors correspond to quotient complexes and quotient functors
correspond to subcomplexes.)

\begin{itemize}[leftmargin=*]
\item \textbf{RI:} Consider the Reidemeister I move from a diagram $L=\begin{tikzpicture}[baseline={([yshift=-.8ex]current bounding
      box.center)}] \node[thick,inner sep=0,outer sep=0,draw,dotted] at (0,0)
    {\begin{tikzpicture}[scale=0.04]\draw[solid] (-2,-2)
        to[out=45,in=-45,looseness=2] (-2,12);\end{tikzpicture}};
\end{tikzpicture}$ to the diagram $L'=\begin{tikzpicture}[baseline={([yshift=-.8ex]current
    bounding box.center)}] \node[thick,inner sep=0,outer sep=0,draw,dotted] at (0,0) {\begin{tikzpicture}[scale=0.04]\draw[solid] (-2,12) to (8,2)
  to[out=-45,in=-90] (12,5) to[out=90,in=45] (8,8); \node[draw=none,crossing] at (5,5) {};
  \draw[solid] (-2,-2) to (8,8);\end{tikzpicture}};
\end{tikzpicture}$. We have that $\KhGen(\diagram')=\KhGen(\diagram'_0) \amalg \KhGen(\diagram'_1)$ where $L'_0=\begin{tikzpicture}[baseline={([yshift=-.8ex]current
    bounding box.center)}] \node[thick,inner sep=0,outer sep=0,draw,dotted] at (0,0) {\begin{tikzpicture}[scale=0.04]\draw[solid] (-2,-2)
        to[out=45,in=-45,looseness=2] (-2,12);\draw[solid] (8,8) to[out=-135,in=135,looseness=2] (8,2)
  to[out=-45,in=-90] (12,5) to[out=90,in=45] (8,8);\end{tikzpicture}};
\end{tikzpicture}$ (respectively, $L'_1=\begin{tikzpicture}[baseline={([yshift=-.8ex]current
    bounding box.center)}] \node[thick,inner sep=0,outer sep=0,draw,dotted] at (0,0) {\begin{tikzpicture}[scale=0.04]\draw[solid] (-2,12) to[out=-45,in=-135,looseness=1.5] (8,8)
  to[out=45,in=90] (12,5) to[out=-90,in=-45] (8,2) to[out=135,in=45,looseness=1.5] (-2,-2);\end{tikzpicture}};
\end{tikzpicture}$) is obtained from $L'$ by resolving the new
crossing by the $0$-resolution (respectively, $1$-resolution). Let $a$
be the new circle in $L'_0$, as shown in the picture.

We may perform a replacement of Type~\ref{item:merge-cancel} by
cancelling the subcomplex of $\uniKhCx(L')$ spanned by all the
generators in $\KhGen(L'_1)$ and only the generators in
$\KhGen(L'_0)$ that do not contain $a$, and after that we will be
left with a quotient complex that is naturally isomorphic to
$\uniKhCx(L)$. That is, we have a quotient complex projection
$\uniKhCx(L')\to\uniKhCx(L)$ that is a chain homotopy equivalence over
$\ZZ_u$. This is induced from a subfunctor inclusion
$F_o(L)\to F_o(L')$, that is, the dual map on the totalizations
\[
(\Tot(\oddeq\circ F_o(L')))^*\to(\Tot(\oddeq\circ F_o(L)))^*
\]
agrees with the map on the unified Khovanov complex, where
$\oddeq\from\oddb\to\eqb$ is the doubling functor from
Figure~\ref{fig:burnsidecategories}. Therefore, the functors $F_o(L)$
and $F_o(L')$ are equivariantly equivalent.

\item\textbf{RII:} The proof for Reidemeister II move is similar,
  except now we have to use both types of cancellations. Say we are
  doing a Reidemeister II from
  $L=\begin{tikzpicture}[baseline={([yshift=-.8ex]current bounding
      box.center)}] \node[thick,inner sep=0,outer sep=0,draw,dotted]
    at (0,0) {\begin{tikzpicture}[scale=0.04]\draw[solid] (-2,12)
        to[out=-45,in=-135] (22,12); \draw[solid] (-2,-2) to[out=45,in=135]
        (22,-2);\end{tikzpicture}};
\end{tikzpicture}$ to
  $L'=\begin{tikzpicture}[baseline={([yshift=-.8ex]current bounding
      box.center)}] \node[thick,inner sep=0,outer sep=0,draw,dotted]
    at (0,0) {\begin{tikzpicture}[scale=0.04]\draw[solid] (-2,12) to
        (8,2) to[out=-45,in=-135] (12,2) to (22,12);
        \node[draw=none,crossing] at (5,5)
        {};\node[draw=none,crossing] at (15,5) {}; \draw[solid]
        (-2,-2) to (8,8) to[out=45,in=135] (12,8) to
        (22,-2);\end{tikzpicture}};
  \end{tikzpicture}$. Once again, $\KhGen(L')$ decomposes as
  $\amalg_{(i,j)\in\{0,1\}^2}\KhGen(L'_{ij})$, where $L'_{ij}$ are the
  partial $(i,j)$ resolutions of $L'$ at the new crossings. Let $a$ be the new circle in $L'_{01}=\begin{tikzpicture}[baseline={([yshift=-.8ex]current bounding
      box.center)}] \node[thick,inner sep=0,outer sep=0,draw,dotted]
    at (0,0) {\begin{tikzpicture}[scale=0.04]\draw[solid] (-2,12) to[out=-45,in=45,looseness=2] (-2,-2); \draw[solid] (22,12) to[out=-135,in=135,looseness=2] (22,-2); \draw[solid] (12,8) to[out=135,in=45] (8,8) to[out=-135,in=135,looseness=2] (8,2) to[out=-45,in=-135] (12,2) to[out=45,in=-45,looseness=2] (12,8);\end{tikzpicture}};
  \end{tikzpicture}$.

  For the merge $L'_{01}\to L'_{11}$, we may cancel the subcomplex
  spanned by $\KhGen(L'_{11})$ and the generators in $\KhGen(L'_{01})$
  that do not contain $a$. The remaining quotient complex has an
  acyclic subcomplex corresponding to the split $L'_{00}\to L'_{01}$,
  spanned by $\KhGen(L'_{00})$ and the remaining generators in
  $\KhGen(L'_{01})$. This produces a chain homotopy equivalence
  between $\uniKhCx(L')$ and $\uniKhCx(L'_{10})$ (modulo shifting the
  homological grading by one), and the latter is naturally identified
  with $\uniKhCx(L)$. Since these subquotient complexes come from
  Burnside sub-functors and Burnside quotient functors, it is
  automatic that the two stable Burnside functors $F_o(L)=F_o(L_{10})$
  and $\Sigma^{-1}F_o(L')$ are equivariantly equivalent.

\item \textbf{RIII:} The proof of Reidemeister III invariance is
  exactly the same as the previous proof. As discussed earlier, we deviate
  from the standard proofs from \cite{kho1,natancat,ors}, but instead
  follow the proof from \cite[Proposition 6.4]{lshomotopytype}. Let
  $L'$ be obtained from $L$ by performing a braid-like Reidemeister
  III move, as in \cite[Figure 6.1c]{lshomotopytype}. Then in the
  six-dimensional partial cube of resolutions of $L'$, one can perform
  as sequence of cancellations---see \cite[Figure 6.4]{lshomotopytype}
  and the subsequent table---of Types~\ref{item:merge-cancel} and
  \ref{item:split-cancel} to produce a chain homotopy equivalence
  between $\uniKhCx(L')$ and $\uniKhCx(L'_{000111})$ (modulo shifting
  gradings by three), and the latter is naturally identified to
  $\uniKhCx(L)$. The subquotient complexes come from Burnside
  sub-functors and Burnside quotient functors, and once again it
  follows that the two stable Burnside functors $F_o(L)=F_o(L_{000111})$
  and $\Sigma^{-3}F_o(L')$ are equivariantly equivalent.
\end{itemize}
We leave it to the reader to convince themselves that the above
equivalences automatically respect the decomposition of the Burnside
functors according to quantum gradings.
\end{proof}

\begin{proof}[Proof of Theorem \ref{thm:oddkhmain}]
Recall that $\khoo(L)=| \Sigma^{-n_-}F_o|=(\CRealize{F_o}_k,-n_--k)$.  By Lemma \ref{lem:functor-map-gives-space-map}, $|\Sigma^{n_-}F_o|$ depends, up to (nonequivariant) stable homotopy equivalence, only on the stable equivalence class of $\Sigma^{n_-}F_o$.  Then by Theorem \ref{thm:odd-functor-invariant}, the stable homotopy class of $|\Sigma^{n_-}F_o|$ is an invariant of $L$.  Proposition \ref{prop:totalization} identifies the cellular chain complex of $|\Sigma^{n_-}F_o|$ as the totalization of $\Sigma^{-n_-}F_o$, which is the dual of the Khovanov complex (see discussion after Definition \ref{def:kh-odd-burnside-functor}), so Theorem \ref{thm:oddkhmain} follows (nonequivariantly).  To see that $\khoo(L)$ is well-defined up to equivariant stable homotopy equivalence, we use Lemma \ref{lem:functor-map-gives-space-map-equivariant} in place of Lemma \ref{lem:functor-map-gives-space-map}.  
\end{proof}

\begin{proof}[Proof of Theorem \ref{thm:evenintro}]
As with the proof of Theorem \ref{thm:oddkhmain}, we see that $\khoh'(L)$, up to equivariant stable homotopy equivalence, depends only on the equivariant equivalence class of $\Sigma^{n_-}F_o$, by Lemma \ref{lem:functor-map-gives-space-map-equivariant}.  Theorem \ref{thm:odd-functor-invariant} establishes that the equivariant equivalence class of $\Sigma^{n_-}F_o$ is a link invariant, and the theorem follows.
\end{proof}

\begin{proof}[Proof of Theorem \ref{thm:unifiedintro}]
The well-definedness follows as in Theorems \ref{thm:oddkhmain} and \ref{thm:evenintro} from Lemma \ref{lem:functor-map-gives-space-map-equivariant}.  For the CW description, we use Proposition \ref{prop:totalization-equivariant}, which establishes that the equality in Theorem \ref{thm:unifiedintro} is an isomorphism of $\ZZ_u$-modules.  The statement about the $\ZZ_2$-actions on the reduced cellular chain complex follows from the construction.
\end{proof}

\subsection{Reduced odd Khovanov homotopy type}\label{subsec:reduced}
We briefly address the reduced theory.  

Given a (generic) point $p$ on a link diagram $L$, there is a natural
sub-functor of $F^j_o(L)$ generated by only those Khovanov generators
that do not contain the circle $c_p$ containing $p$. Let
$\widetilde{F}^{j-1}_{o,+}(L,p)$ denote this subfunctor and
$\widetilde{F}^{j+1}_{o,-}(L,p)$ denote the corresponding quotient
functor.

Next we show that the two reduced functors $\widetilde{F}^j_{o,+}$ and
$\widetilde{F}^j_{o,-}$ are canonically identified.  Arranging for
convenience that the ordering of circles at each resolution has that
$c_p$---the circle containing $p$---is always last, we have a
canonical bijection between $\widetilde{F}_{o,-}$ and
$\widetilde{F}_{o,+}$.  This bijection is compatible with the
$1$-morphisms of the even Burnside functor; however, we must also check that these respect the sign map.  To be
specific, for $u\geqslant_1 v$, the bijection
$\widetilde{F}_{o,-}(\phi_{u,v})\to\widetilde{F}_{o,+}(\phi_{u,v})$
preserves all
signs cf.~\S\ref{sec:threekhfunc}, as the reader may check.  We will refer to either functor as
$\widetilde{F}_o$.

We would expect, based on what happens for odd Khovanov chain complex,
that the unreduced functor $F^j_{o}$ should be stably equivalent to
two copies of the reduced functor,
$\widetilde{F}^{j-1}_o\amalg\widetilde{F}^{j+1}_o$. However, the chain
level splitting from \cite{ors} does not generalize. Indeed, any such
stable equivalence cannot be a equivariant equivalence, cf.~Definition
\ref{def:stableq}, since by Proposition~\ref{prop:odd-recovers-even},
$\forgot F_o=F_e$ (and similarly,
$\forgot \widetilde{F}_o = \widetilde{F}_e$, where $\widetilde{F}_e$
is the reduced even Burnside functor), and the even Burnside functor
(and indeed, the even Khovanov chain complex) does not split as two
copies of its reduced version.

\begin{defn}\label{def:redkh}
  We define the \emph{reduced odd Khovanov spectrum}
  $\widetilde\khoo(L,p)=\bigvee_j\widetilde\khoo^j(L,p)$ as a
  $\ZZ_2$-equivariant finite CW spectrum, where
  $\widetilde\khoo^j(L,p)$ is a realization of the stable signed
  Burnside functor $\Sigma^{-n_-}\widetilde{F}^j_{o}$.
\end{defn}
\begin{defn}\label{def:reduced-unified}
 We define the \emph{reduced unified Khovanov spectrum}
  $\reduniKhspace(L,p)=\bigvee_j\reduniKhspace^j(L,p)$ as a
  $\ZZ_2\times \ZZ_2$-equivariant finite CW spectrum, where
  $\reduniKhspace^j(L,p)$ is a realization of the stable signed
  Burnside functor $\Sigma^{-n_-}\oddeq \widetilde{F}^j_{o}$.
\end{defn}

\begin{proof}[Proof of Theorem \ref{thm:redkh}] 
  Well-definedness of $\khor(L)$, up to equivariant stable homotopy equivalence, will follow from showing that $\Sigma^{-n_-}\widetilde{F}_o$ is well-defined up to equivariant equivalence, depending only on the isotopy class of $(L,p)$.  Isotopy invariance is immediate for
  Reidemeister moves away from the basepoint (using the maps induced by Reidemeister moves on $\Sigma^{-n_-}F_o$, and observing that they preserve $\Sigma^{-n_-}\widetilde{F}_o$).  As observed in
  \cite{Kho-kh-patterns}, any two diagrams for isotopic pointed
  links can be related by Reidemeister moves not crossing the
  basepoint and isotopies in $S^2$, from which well-definedness
  follows.  The cofibration sequence is a consequence of Lemma \ref{lem:cofib}, using the cofibration sequence of Burnside functors:
  \[
  \widetilde{F}^{j-1}_{o,+}(L,p)\to F^j_o(L) \to \widetilde{F}^{j+1}_{o,-}(L,p).
  \]
  Finally, the description of the reduced cellular cochain complex is a consequence of Proposition \ref{prop:totalization}, as in the proof of Theorem \ref{thm:oddkhmain}.
\end{proof}

\begin{prop}\label{prop:reduced-unified}
  The (stable) homotopy type of the reduced unified Khovanov spectrum
  $\khor(L,p)=\bigvee_j \khor^j(L,p)$ from
  Definition~\ref{def:reduced-unified} is independent of the choices in its
  construction and is an invariant of the isotopy class of the pointed
  link corresponding to $(L,p)$.  Its reduced cellular cochain complex
  agrees with the reduced unified Khovanov complex
  $\widetilde{\KhCx}_u(L)$,
  \[
  \redcellC^i(\reduniKhspace^j(L,p))= \widetilde{\KhCx}_u^{i,j}(L),
  \]
  with the cells mapping to the distinguished generators of
  $\widetilde{\KhCx}_u(L)$.  There is a cofibration sequence 
  \[\reduniKhspace^{j-1}(L,p)
  \to \unis^j(L)\to\reduniKhspace^{j+1}(L,p).\]
\end{prop}

\begin{proof}
The proof of Theorem \ref{thm:redkh} goes through mostly unchanged.  The only new observation necessary is that the double of a cofibration sequence of Burnside functors is again a cofibration sequence.
\end{proof}

\subsection{Cobordism maps}\label{subsec:cobordism-maps}
For every smooth link cobordism $L\to L'$ embedded in
$\RR^3\times[0,1]$, there is an induced map on the even Khovanov
complex $\KhCx(L)\to\KhCx(L')$
\cite{jacobsson,natantangle,khinvartangle}, well-defined up to chain
homotopy and an overall sign. (The dependence on the overall sign can
be removed, see~\cite{cmw-disoriented}.)

\cite{lsrasmussen} lifted these maps to the even Khovanov homotopy
types, $\Khspace(L')\to\Khspace(L)$, so that the induced map on the
cellular cochain complex is the previous map, but did not check
well-definedness. It is fairly easy to check that the map
$\Khspace(L')\to\Khspace(L)$ defined in \cite{lsrasmussen} comes
from a map of the even Burnside functors $F_e(L')\to F_e(L)$, so
that the dual of the map on their totalizations is the map
$\KhCx(L)\to\KhCx(L')$.

In this section, we will further lift these to maps of the odd
Burnside functor $F_o(L')\to F_o(L)$, so that the even Burnside
functor map is obtained by applying the forgetful functor
$\forgot\from\oddb\to\burn$. In particular, we will get maps on the
odd Khovanov homotopy type, $\oddKhspace(L')\to\oddKhspace(L)$ and
the odd Khovanov complex, $\oddKhCx(L)\to\oddKhCx(L')$. We will not
check the well-definedness of any of these maps.

The standard way to define these maps is by decomposing the cobordism
as movie, which is a sequence of knot diagrams so that each one is
obtained from the previous one by a planar isotopy, Reidemeister move,
or a Morse critical point, which can be a birth, death, or a
saddle. In \S\ref{subsec:invar}, we have already constructed maps of
odd Burnside functors corresponding to the Reidemeister moves (which
were also equivariant equivalences). So we only need to construct maps
associated to the three Morse singularities.

First we consider the cup and cap cobordisms.  Let $L$ a link diagram,
and $L'=L\amalg U$, the disjoint union of $L$ and an unknot,
introducing no new crossings.  The elementary cobordism from $L$ to
$L'$ is called a \emph{cup} or a \emph{birth}, while that from $L'$ to
$L$ is a \emph{cap} or a \emph{death}.  We construct natural
transformations
\begin{align*}
\Phi_o^\cup & \from F_o(L') \to F_o(L) \\
\Phi_o^\cap & \from F_o(L) \to F_o(L'),
\end{align*}
decreasing quantum grading by $1$, lifting the natural transformations
\begin{align*}
\Phi_e^\cup & \from F_e(L') \to F_e(L) \\
\Phi_e^\cap & \from F_e(L) \to F_e(L')
\end{align*}
for the even Burnside functors from \cite{lsrasmussen}.

In each resolution of $L'$ there is a component corresponding to $U$.
We can write $\KhGen(L')=\KhGen(L)_+\amalg \KhGen(L)_-$ where
$\KhGen(L)_-$ (respectively, $\KhGen(L)_+$) is the subset of
generators in $\KhGen(L')$ which contain $U$ (respectively, do not
contain $U$); either is canonically identified with $\KhGen(L)$ by
ordering the circles at each resolution so that $U$ is
last. Let $F_o(L)_-$ (respectively,
$F_o(L)_+$) be the subfunctor of $F_o(L')$ generated by $\KhGen(L)_-$
(respectively, $\KhGen(L)_+$); either is isomorphic to $F_o(L)$,
modulo a quantum grading shift of $\pm 1$. Then
$F_o(L')=F_o(L)_-\amalg F_o(L)_+$, and so there is a subfunctor
inclusion $F_o(L)_-\to F_o(L')$ and quotient functor projection
$F_o(L')\to F_o(L)_+$.

We then set the cobordism maps according to:
\begin{align*}
\Phi_o^\cup &\from F_o(L')\to F_o(L)_+ \cong F_o(L)\\
\Phi_o^\cap &\from F_o(L) \cong F_o(L)_- \to F_o(L').
\end{align*}

Next, we handle the saddle case.  Let $L_0,L_1$ be $n$-crossing link
diagrams before and after the saddle, as in \cite[Figure
3.2]{lsrasmussen}, and let $F_e(L_0),F_e(L_1)\from\two^n\to\burn$ be
the two even Burnside functors (we have implicitly identified the
crossings in $L_0$ with the crossings in $L_1$).  Associated to the
saddle cobordism, \cite{lsrasmussen} constructs a natural
transformation $\Phi^s_e\from F_e(L_1) \to F_e(L_0)$ as follows. There
is a $(n+1)$-crossing diagram $L$ so that $L_i$ is the $i$-resolution
of $L$ at the new crossing, for $i=0,1$. Then
$\Phi^s_e\from\two^{n+1}\to\burn$ is simply defined to be
$F_e(L)\from\two^{n+1}\to\burn$---the even Burnside functor associated
to $L$.  One easily checks that the natural transformation increases
the quantum grading by $1$.

The generalization to the signed Burnside functor version is
immediate, and we obtain a natural transformation
$\Phi^s_o\from F^s_o(L_1) \to F^s_o(L_0)$.

\begin{lem}\label{lem:main-cobordism}
  Associated to a (movie presentation of a) link cobordism
  $L\to L'$, the map on the odd Burnside functors
  $\Phi_o\from F^{j+\chi(S)}_o(L')\to F^j_o(L)$ lifts the map on the even
  Burnside functors $\Phi_e\from F^{j+\chi(S)}_e(L')\to F^j_e(L)$ that is
  (implicitly) constructed in \cite{lsrasmussen}:
  \[
    \forgot\Phi_o= \Phi_e.
  \]
  In particular, the induced map on the $\ZZ_2$ chain complex,
  $\KhCx^{*,j}(L;\ZZ_2)\to\KhCx^{*,j+\chi(S)}(L';\ZZ_2)$, agrees with
  the Khovanov map (up to chain homotopy).
\end{lem}
\begin{proof}
  This is immediate from the definitions for each of the maps---the
  Reideister invariance maps, as well as the cups, caps, and the
  saddles.
\end{proof}

Let $\SteenrodAlg$ denote the mod-2 Steenrod algebra. Let $\asig$
denote the free product of two copies of $\SteenrodAlg$. It acts on
$\Kh(L;\ZZ_2)$ as follows: the first (respectively, second) copy acts by
viewing $\Kh(L;\ZZ_2)$ as the mod-2 cohomology of $\khoh(L)$ (respectively,
$\khoo(L)$).

\begin{proof}[Proof of Theorem \ref{thm:cobordism-maps}]
  We follow the proof of \cite[Corollary 5]{lsrasmussen}. Let $S\from
  L\to L'$ be the link cobordism and
  $\Kh_S\from\Kh(L;\ZZ_2)\to\Kh(L';\ZZ_2)$ the induced map.  By
  Lemma~\ref{lem:main-cobordism}, $\Kh_S$ comes from a map of Burnside
  functors $\Phi^S_o\from F_o(L')\to F_o(L)$.

  The even and odd realizations give two actions of $\SteenrodAlg$ on
  the mod-2 Khovanov homology, and, in particular, by naturality of
  Steenrod operations on either the odd or even spatial realization,
  we have a commutative diagram
  \[
  \begin{tikzpicture}[xscale=5,yscale=1.5]
    \node (a0) at (0,0) {$\Kh^{i,j}(L;\f)$};
    \node (a1) at (1,0) {$\Kh^{i+n,j}(L;\f)$};
    \node (b0) at (0,-1) {$\Kh^{i,j+\chi(S)}(L';\f)$};
    \node (b1) at (1,-1) {$\Kh^{i+n,j+\chi(S)}(L';\f)$};
    
    \draw[->] (a0) -- (a1) node[pos=0.5,anchor=south] {\scriptsize $\alpha$};
    \draw[->] (a0) -- (b0) node[pos=0.5,anchor=east] {\scriptsize $\Kh_S$};
    \draw[->] (b0) -- (b1) node[pos=0.5,anchor=south] {\scriptsize $\alpha$};
    \draw[->] (a1) -- (b1) node[pos=0.5,anchor=west] {\scriptsize $\Kh_S$};
    
  \end{tikzpicture}
  \]
  for $\alpha$ a stable cohomology operation of degree $n$ coming from
  either copy of $\SteenrodAlg$.  It is then clear that the diagram
  then also commutes for $\alpha$ a linear combination or composition
  of elements of the two copies of $\SteenrodAlg$.  The theorem
  follows directly from these diagrams.
\end{proof}

\subsection{Concordance invariants}\label{sec:conc}
Theorem~\ref{thm:cobordism-maps} allows one to define knot concordance
invariants, once again borrowing arguments directly from
\cite{lsrasmussen}. In this section, we work only with knots, not
links.

For a knot diagram $K$ and field $\f$, there is a spectral sequence
with $E^2$-page Khovanov homology $\Kh^{i,j}(K;\f)$ converging to
$\f^2$ coming from a descending filtration $\Filt$ of the Khovanov
chain complex $\KhCx(K)$ so that
\[
\KhCx^{*,j}(K;\f)=\Filt_j/\Filt_{j+2}.
\]
This was defined by Lee \cite{lee} for fields not of characteristic
$2$, and for all fields by Bar-Natan \cite{natantangle}, see also
\cite{turner,naot}; from now on, fix our field as $\f=\f_2$, the field
with two elements, and we work their variant.

The Rasmussen $s$ invariant of $K$ from \cite{rasmus-s},
cf.~\cite{lsrasmussen}, is defined by
\begin{align*}
s^\f(K)&= \max \{ q\in 2\ZZ+1 \mid H^*(\Filt_q)\to H^*(\Filt_{-\infty}) \cong \f^2 \text{ surjective } \} +1,\\
&= \max\{ q\in 2\ZZ+1 \mid  H^*(\Filt_q)\to H^*(\Filt_{-\infty}) \cong \f^2 \text{ nonzero } \} -1, 
\end{align*}

\begin{defn}\label{def:full}
  Fix $\alpha\in\asig$ of grading $n>0$. Call $q
  \in 2\ZZ+1$ $\alpha$-\emph{half full} if there exist elements
  $\widetilde{a} \in \Kh^{-n,q}(K;\f)$, $\widehat{a} \in \Kh^{0,q}(K;\f)$,
  $a\in H^0(\Filt_q ; \f)$, $\bar{a}\in H^0(\Filt_{-\infty};\f)$ such that
\begin{enumerate}[leftmargin=*]
\item\label{itm:full1} the map $\alpha \from \Kh^{-n,q}(K;\f) \to
  \Kh^{0,q}(K,\f)$ from Theorem~\ref{thm:cobordism-maps} sends $\widetilde{a}$ to
  $\widehat{a}$.
\item The map $H^0(\Filt_q;\f)\to \Kh^{0,q}(K;\f)=H^0(\Filt_q/\Filt_{q+2};\f)$ sends $a$ to $\widehat{a}$.
\item\label{itm:full3} The map $H^0(\Filt_q;\f)\to H^0(\Filt_{-\infty};\f)$ sends $a$ to $\bar{a}$.  
\item $\bar{a}\in H^0(\Filt_{-\infty};\f)=\f\oplus \f$ is nonzero.
\end{enumerate} 
Call $q$ $\alpha$-\emph{full} if there exists tuples
$(\widetilde{a},\widehat{a},a,\bar{a})$ and
$(\widetilde{b},\widehat{b},b,\bar{b})$ as above, with properties
(\ref{itm:full1})--(\ref{itm:full3}), and so that $\bar{a},\bar{b}$ is
a basis of $H^0(\Filt_{-\infty};\f)$.
\end{defn}

\begin{defn}\label{defn:concinv}
  For a knot $K$, define $r^{\alpha}_+(K)=\max \{ q \in 2\ZZ+1 \mid q
  \; \text{is }\alpha\text{-half-full} \}+1$, and
  $s^{\alpha}_+(K)=\max \{ q \in 2\ZZ+1 \mid q \; \text{is }
  \alpha\text{-full} \}+3$.  If $m(K)$ is the mirror of $K$, let
  $r^{\alpha}_-(K)=-r^{\alpha}_+(m(K))$ and
  $s^{\alpha}_-(K)=-s^{\alpha}_+(m(K))$.
\end{defn}

\begin{thm}\label{thm:conc}
  Let $\alpha \in \asig$ and $S$ a connected, embedded cobordism in
  $\RR^3\times [0,1]$ from $K$ to $K'$ of genus $g$.  Then
  \begin{align*}
    | r^{\alpha}_{\pm}(K)-r^{\alpha}_{\pm}(K') | &\leq 2g\\
    | s^{\alpha}_{\pm}(K)-s^{\alpha}_{\pm}(K') | &\leq 2g.
  \end{align*}
  In particular, $|r^{\alpha}_{\pm}(K)|/2,|s^{\alpha}_{\pm}(K)|/2$ are
  concordance invariants and lower bounds for the slice genus $g_4(K)$.
\end{thm}
\begin{proof}
  This follows from Theorem~\ref{thm:cobordism-maps}, arguing as in
  \cite[Theorem 1]{lsrasmussen}.
\end{proof}

\subsection{Questions}
We conclude with some structural questions about the odd Khovanov space:

\begin{enumerate}[leftmargin=*,label=(q-\arabic*)]
\item In \S\ref{sec:conc} we constructed concordance invariants using
  the action of the mod-$2$ Steenrod algebra on
  $\widetilde{H}^*(\khoo(L);\ZZ_2)$.  It is natural to ask for
  concordance invariants defined from homology using different
  coefficient fields.  Indeed, \cite{lsrasmussen} defines such
  invariants using stable cohomology operations with any coefficient
  field.  For this, perhaps one needs an analogue of the Lee spectral
  sequence for odd Khovanov homology.
\item\label{item:reduced-2copies} Ozsv\'ath-Rasmussen-Szab\'o~\cite{ors}
  showed that $\kho^{*,j}(L)=\kor^{*,j-1}(L)\oplus \kor^{*,j+1}(L)$
  for any link $L$.  Is it the case that $\khoo^j(L) \simeq
  \khor^{j-1}(L)\vee \khor^{j+1}(L)$? More specifically, is there a
  stable equivalence between the signed Burnside functors $F^j_o(L)$
  and $\widetilde{F}^{j-1}_o(L)\amalg\widetilde{F}^{j+1}_o(L)$? (Such
  an equivalence has to be non-equivariant.)
\item So far, calculations of the odd Khovanov homotopy type are
  limited.  Is it always a wedge sum of Moore spaces? 
  Do there exist links $L$ for which $\khoo(L)$ is not a wedge sum of
  smash products of Moore spaces?
\item In \cite{putyrashumakovitch} there are short exact sequences:
  \begin{equation*}
    \KhCx_e(L) \to \unic(L) \to \oddKhCx(L) \qquad \text{ and } \qquad \oddKhCx(L) \to \unic(L) \to \KhCx_e(L) 
  \end{equation*}
  At the level of cohomology, these exact sequences are induced from
  the cofibration sequences in Theorem \ref{thm:equivariance}.
  However, the maps in that proposition were not cellular for the
  coarse CW structure.  That leads to the question: are there CW
  cofibration sequences
  \begin{equation*}
    \khoo(L) \to \unis(L) \to \khoh(L) \qquad \text{ and } \qquad \khoh(L) \to \unis(L) \to \khoo(L) 
  \end{equation*}
  (with respect to the coarse CW structure) inducing the maps of
  \cite{putyrashumakovitch} on cellular chain complexes?
\item One of the applications of the technology of \cite{lls1} was to
  show
  \[
  \khoh(m(L))\simeq \khoh(L)^\vee
  \]
  where $m(L)$ is the mirror of $L$, and $\vee$ denotes the
  Spanier-Whitehead dual.  We conjecture, similarly, that
  $\khoo(m(L))\simeq \khoo(L)^\vee$.  The proof of the statement in
  even Khovanov homology involved the TQFT structure of even Khovanov
  homology, and does not immediately generalize to odd Khovanov
  homology.
\item It would be desirable to understand the behavior of the odd
  Khovanov spectra for disjoint unions and connected sums.  Is it
  possible to express the (equivariant) homotopy type $\khoo(L_1
  \amalg L_2)$ in terms of the (equivariant) odd Khovanov spectra of
  $L_1,L_2$?  (In the even theory, it is merely the smash product.)
  Is it possible to express the odd and unified (unreduced) Khovanov
  spectra of $L_1\#L_2$ in terms of the spectra of the component
  links?  For even Khovanov homotopy, this was dealt with as
  \cite[Theorem 8]{lls1}: $\khoh(L_1\# L_2)$ is the derived tensor
  product of $\khoh(L_i)$ over the even Khovanov spectrum of the
  unknot.
\item Is the old even Khovanov spectrum $\khoh(L)$ stable homotopy
  equivalent to the new even Khovanov spectrum $\khoh'(L)$? More
  generally, does the Khovanov spectrum $\X_\ell(L)$ from
  Remark~\ref{rmk:khovanov-spaces-with-extra-action} depend only on
  the parity of $\ell$?
\end{enumerate}

\bibliographystyle{amsalpha}
\bibliography{oddkh}
\vspace{0.5in}

\end{document}